\documentclass[a4paper,11pt]{article} 
\usepackage[utf8]{inputenc}
\usepackage[english]{babel}

\usepackage{amsmath} 
\usepackage{amsthm} 
\usepackage{amssymb} 

\usepackage[mathscr,mathcal]{eucal} 
\usepackage{dsfont}
\usepackage{enumerate} 
\usepackage{hyperref}

\usepackage{tikz}
\usetikzlibrary { decorations.pathmorphing, decorations.pathreplacing, decorations.shapes, shapes.geometric} 

\newtheorem{theorem}{Theorem}[section]

\newtheorem  {lemma}[theorem]         {Lemma}

\newtheorem  {prop}[theorem]   {Proposition}
\newtheorem  {defi}[theorem]    {Definition}
\newtheorem {notation}[theorem]    {Notation}
\newtheorem* {notation*}    {Notation}
\newtheorem {algorithm}[theorem]    {Algorithm}
\newtheorem* {theorem*}      {Theorem}
\newtheorem* {lemma*}        {Lemma}
\newtheorem* {corollary*}    {Corollary}
\newtheorem* {prop*}  {Proposition}
\newtheorem* {defi*}   {Definition}
\newtheorem* {remark*}       {Remark}
\newtheorem  {remark}[theorem]        {Remark}

\newtheorem* {claim*}        {Claim}

\newtheorem* {assum*}        {Assumption}
\newtheorem  {assum}[theorem]        {Assumption}

\def \eps{\epsilon}

\def \Z {\mathbb{Z}}
\def \N {\mathbb{N}}

\def \P {\mathbb P}
\def \E {\mathbb E}
\def \var {\text{Var}}
\def \cov {\text{Cov}}
\def \unif {\mathrm{Unif}}
\def \ber {\mathrm{Ber}}

\def \Lvert {\middle|\:}
\def \given {\;|\;}
\def \Given {\;\Lvert\;}

\def \wst {\mathsf{WST}}
\def \mst {\mathsf{MST}}
\def \maxst {\mathsf{MaxST}}
\def \ust {\mathsf{UST}}

\def \path {\mathrm{Path}}
\def \diam {\mathrm{Diam}}
\def \lp {\mathrm{lp}}

\def \c {\boldsymbol{\mathit{c}}}
\def \C {\mathcal{C}}
\def \bcond{\mathrm{bigc}}
\def \mcond{\mathrm{mediumc}}
\def \scond{\mathrm{smallc}}

\author{\'Agnes K\'usz \footnote{Department of Stochastics, Institute of Mathematics, Budapest University of Technology and Economics, M\H{u}egyetem rkp.~3., Budapest 1111 Hungary. HUN-REN Alfr\'ed R\'enyi Institute of Mathematics, Re\'altanoda u. 13-15, Budapest 1053 Hungary. kusza [at] renyi.hu, kuszagnes1205 [at] gmail.com}}
\title{The diameter of random spanning trees interpolating between the UST and the MST of the complete graph}
\date{}

\begin{document}
\maketitle

\begin{abstract}
We introduce $\wst^{\beta_n}(K_n)$ as the weighted spanning tree of the complete graph $K_n$ w.r.t.~the random electric network of conductances $\{\exp(-\beta_nU_{e})\}_{e\in E(K_n)}$ with $\unif[0,1]$ i.i.d.~$U_e$'s. 

Moving from $\beta_n\equiv 0$ to faster and faster growing $\beta_n$'s, the model interpolates between the \emph{uniform} and the \emph{minimum} spanning trees: $\wst^{0}(K_n)=\ust(K_n)$, and there are phase transitions for $\wst^{\beta_n}(K_n)$ behaving more and more like
$\mst(K_n)$
\begin{itemize}
    \item[] around $\beta_n=n^{3+o(1)}$ regarding the agreement of the two standard algorithms generating these models : Aldous-Broder and Prim's invasion algorithms, 
    \item[] around $\beta_n=n^{2+o(1)}$ regarding the models consisting of exactly the same edges, and 
    \item[] around $\beta_n=n^{1+o(1)}$ regarding the expected total length $\E\left[\sum_{e\in \wst^{\beta_n}(K_n)}U_e\right]$.
\end{itemize}

But most importantly, we study the global geometry of the model: we prove that the typical diameter of $\wst^{\beta_n}(K_n)$ grows like $\Theta(n^{1/3})$ for $\beta_n\ge n^{4/3+o(1)}$ likewise the $\mst(K_n)$ case, and it grows like $\Theta(n^{1/2})$ for $\beta_n\le n^{1+o(1)}$ similarly to the $\ust(K_n)$ case. For $\beta_n=n^{\alpha}$ with $1<\alpha<4/3$, the behavior of the typical diameter is a more delicate open question, but we conjecture that its exponent strictly between 1/2 and 1/3.
\end{abstract}

\tableofcontents

\section{Introduction}\label{sec:intro}

\subsection{Motivation for the model}

A spanning tree of a graph $G$ is a connected subgraph of $G$ which does not contain any cycles. We introduce the following random spanning tree model.

\begin{defi}\label{def:wst_beta}
We consider a graph $G$ and $\beta\ge 0$. Given some i.i.d.~$\unif[0,1]$ random variables $\mathbf{U}:=\{U_e\}_{e\in E(G)}$, the probability of a spanning tree $T$ of $G$ under the 
 \emph{weighted spanning tree} model $\wst^{\beta}(G)$ is defined as
 \begin{align*}
     \P\left(\wst^{\beta}(G)=T\Given\mathbf{U}\right)&:= \frac{\exp\left(-\beta\sum_{e\in T}U_e\right)}{Z_{\beta}(G\given \mathbf{U})} \text{, where }\ \\
     Z_{\beta}(G\given \mathbf{U})&:=\sum_{\substack{T'\text{ is a spanning }\\\text{ tree of } G}}\exp\left(-\beta\sum_{e\in T'}U_e\right).
 \end{align*}

In this paper, we always choose $G$ to be the complete graph $K_n$ on $n$ vertices, $\beta=\beta_n$ dependent on $n$, and we often abbreviate $\wst^{\beta_n}(K_n)$ as $\wst_n^{\beta_n}$.

We refer to $U_e$ as the label of $e$ and to $\exp(-\beta_n U_e)$ as the weight or the conductance of $e$ in our random electric network.
\end{defi}

Recall the definition of the uniform spanning tree model.
\begin{defi}
    For a finite, connected graph $G$, the uniform spanning tree $\ust(G)$ is defined as the uniform distribution on the set of spanning trees of $G$. We often write
    $\ust_n:=\ust(K_n)$.
\end{defi}

For $\beta=0$, the model $\wst^{0}(G)$ becomes $\ust(G)$ by definition. Many properties of $\ust(G)$ are understood for several graphs $G$. 

We are especially interested in $\diam(\ust(G))$, which was firstly studied in \cite{szekeres1983distribution} for $K_n$, where the limiting distribution of $\diam(\ust(K_n))/n^{1/2}$ was determined. For several graphs $G$ other than $K_n$, it has been shown that the diameter of $\ust(G)$ is typically of order $|V(G)|^{1/2}$. For example, for any regular $G$ with a spectral gap separated from 0, $\E[\diam(\ust(G))]$ is of order $|V(G)|^{1/2}$ up to a polylogarithmic factor \cite{aldous, chung2012diameter}. For the $d$-dimensional torus $\Z^d_n$ with $d\ge 5$, the distance w.r.t.~$\ust(\Z^d_n)$ between two uniformly chosen vertices is $\Theta(|V(\Z^d_n)|^{1/2})$, and the joint distribution of the distances of $k$ uniformly chosen vertices is also understood \cite{peres2004scaling}. Recently, the typical diameter of $\ust(G)$ has been proven to be $|V(G)|^{1/2}$---in the sense of tightness of $\diam(\ust(G))/|V(G)|^{1/2}$ as $|V(G)|\rightarrow\infty$---for any graph with linear minimum degree \cite{alon2022diameter} and for any high dimensional graph \cite{michaeli2021diameter}, including balanced expanders, tori of dimension $d\ge 5$ and hypercubes.

Understanding the order of the typical diameter is essential for scaling limit results w.r.t.~the Gromov-Hausdorff
topology. For $G=K_n$, the scaling limit of $\ust(G)$ is the \emph{Continuum Random Tree} introduced by Aldous \cite{aldous1991continuumI, aldous1991continuumII, aldous1993continuum}, and this scaling limit is universal for high dimensional vertex transitive graphs \cite{archer2024ghpHigh} and for dense graphs \cite{archer2024ghpDense}. 

Complementing the global perspective, the local limit of $\ust(G)$ for many graphs has also been determined: it is the Poisson(1) branching process conditioned to survive forever; see, e.g.,~in \cite{hladky2018local, nachmias2022local}. 

Recently in \cite{makowiec2023diameter}, for i.i.d.~edge weights of any fixed distribution, it has been proven that the typical diameter of the weighted spanning tree of $G$ is of order $|V(G)|^{1/2+o(1)}$ if $G$ is a bounded degree expander or a finite box of $\Z^d$ with $d\ge 5$. Their tool was to generalize the results of \cite{michaeli2021diameter} to the $\wst$ on any electric network satisfying the analogous high dimensionality conditions. 

In \cite{makowiec2023diameter}, they worked on high dimensional bounded degree graphs, and the emphasis was on the existence of a large set of edges of comparable weights such that the effect of the incomparable edges on the typical diameter was small, resulting in similar behavior as for the $\ust$. On $K_n$, since the degrees tend to infinity, there can be more edges with large weights making it impossible to handle the diameter of $\wst$ this way: for large $\beta_n$'s, we will see later that $\wst^{\beta_n}(K_n)$ behaves unlike the $\ust$ but like the $\mst$, defined below.

Very recently, papers \cite{makowiec2024diameter, makowiec2024local} with some of their questions and results overlapping with ours have been uploaded to arXiv. See Subsection \ref{subsec:similar_papers}.

\begin{defi}\label{def:mst}
Consider a graph $G$ and some i.i.d.~$\unif[0,1]$ random variables $\mathbf{U}:=\{U_e\}_{e\in E(G)}$. Given $\mathbf{U}$, the \emph{minimum spanning tree} $\mst(G)$ is the spanning tree $T$ of $G$ which minimizes $\sum_{e\in T}U_e$. We often abbreviate $\mst_n:=\mst(K_n)$.
\end{defi}
Observe that $\mst_n$ is well-defined since, almost surely, all the $\sum_{e\in T}U_e$'s are different as $T$ runs on the spanning trees of $K_n$.

The properties of $\mst$ really differ from the properties of the $\ust$. The expected diameter $\E[\diam(\mst(K_n))]$ grows like $\Theta(n^{1/3})$ \cite{addario2009critical}. Moreover, the Gromov-Hausdorff scaling limit \cite{addario2017scaling, addario2021geometry}, and the local limit \cite{addario2013local} of $\mst(K_n)$ differ from the ones for $\ust(K_n)$. The asymptotics of its expected total length $\E\left[\sum_{e\in \mst_n}U_e\right]$ were shown to be $\zeta(3)$ \cite{frieze1985value, aldous1990random, steele2002minimal}.

For any fixed $\ell$ and $\beta_n\rightarrow\infty$, we have that $\P(\wst^{\beta_n}_{\ell}=\mst_{\ell})\rightarrow 1$ as $n\rightarrow \infty$, since the weights assigned by $\wst^{\beta_n}_{\ell}$ to the spanning trees of $K_{\ell}$ become more and more different. In this paper we study the following question: as $n\rightarrow\infty$, how fast should $\beta_n$ tend to infinity to make the behavior of $\wst^{\beta_n}_n$ and $\mst_n$ similar?

The $\wst$ and the $\mst$ are related to different areas of probability theory, making our question even more exciting. The weighted spanning trees can be generated nicely using random walks on electric networks, so the behavior of the $\wst$ can be handled by understanding the random walk trajectories and some properties of the underlying electric network. On the other hand, the properties of the $\mst$ are deeply connected to the component structure of the Erdős-Rényi graphs. 

In this paper, we compare $\wst_n^{\beta_n}$ and $\mst_n$ by determining how fast $\beta_n$ should tend to infinity to ensure that the random walk on our random electric network follows the component structure of the Erdős-Rényi graphs in different senses.

\subsection{Main results and open questions}
We generate $\wst_n^{\beta_n}$ by the Aldous-Broder algorithm \cite{aldous, broder1989generating} on the random electric network of conductances $\{\exp(-\beta_nU_e)\}_{e\in E(K_n)}$: given i.i.d.~labels $U_e$, we start a random walk from any vertex $s\in V(K_n)$ with transition probabilities $p(x,y):=\frac{\exp(-\beta_nU_{(x,y)})}{\sum_{v\in V(K_n)\setminus\{x\}}\exp(-\beta_nU_{(x,v)})}$. We build a spanning tree started from the subtree $\{s\}$ containing only one vertex. Whenever the random walk steps to a vertex which has not been visited before, it adds the vertex and the current edge---the edge through which the new vertex is first explored.

We generate $\mst_n$ by Prim's invasion algorithm \cite{prim1957shortest}: starting from the tree $T_0:=\{s\}$, $T_{k+1}$ is obtained from $T_k$ by adding the edge with minimal $U_e$ running between $V(T_k)$ and $V(K_n)\setminus V(T_k)$. 

An introduction to these models and algorithms on general electric networks is given in Subsection \ref{subsec:det_intro}.

If $\beta_n$ is large, then both the Aldous-Broder and Prim's invasion algorithms connect a new vertex to the current subtree by an edge with small $U_e$ in each step. By determining how different the labels of the edges built by these two algorithms can be, we compare the properties of $\wst_n^{\beta_n}$ and $\mst_n$. We remark that, although our approach gives strong results by the theoretical study of the random walk, the computer simulations of the Aldous-Broder algorithm are extremely slow if the conductances are of very different order, since the cover time of the electric network becomes huge.

\paragraph{Agreement of the algorithms.} Around $\beta_n=n^{3+o(1)}$, there is a phase transition regarding whether the probability of the two algorithms adding the same edge in every step tends to 0 or 1.

\begin{theorem} \label{thm:ab_vs_inv}
We consider some i.i.d.~$\unif[0,1]$ random variables $\mathbf{U}:=\{U_e\}_{e\in E(K_n)}$, and $s\in V(K_n)$. We denote by $f_i$ and $g_i$ the $i^{th}$ edges added by the Aldous-Broder and Prim's invasion algorithms respectively, and by $\P_s(\cdot)$ the coupling between these algorithms started at $s$ obtained by using the same $\mathbf{U}$. 

Then the following holds.
\begin{itemize}
\item[a)] If $\beta_n\gg n^{3}\log n$, then the Aldous-Broder and Prim's invasion algorithms add the same edges in every step with probability $\rightarrow 1$, i.e.,
$$\P_{s}(f_{i}= g_{i}\, \forall i\in \{1, \ldots, n-1\})=1.$$
\item[b)] If $\beta_n\ll n^{3}$, then the two algorithms typically differ in some steps, i.e.,
$$\P_{s}(f_{i}= g_{i}\, \forall i\in \{1, \ldots, n-1\})=0.$$
\end{itemize}
\end{theorem}
The proof of this theorem is in Subsections \ref{subsec:ideas} and \ref{subsec:ab_vs_invasion}.

Note that $\P_s(\cdot)$ couples not only the algorithms, but also the models $\wst_{n}^{\beta_n}$ and $\mst_n$: they use the same $\mathbf{U}$ in Definitions \ref{def:wst_beta} and \ref{def:mst}.

\paragraph{Agreement of the models.}Around $\beta_n=n^{2+o(1)}$, there is a phase transition regarding whether $\P(\wst^{\beta_n}_n=\mst_n)$ is tending to 0 or 1. 
\begin{theorem} \label{thm:wst_vs_mst}
We consider some i.i.d.~$\unif[0,1]$ random variables $\mathbf{U}:=\{U_e\}_{e\in E(K_n)}$. We denote by $\P(\cdot)$ the coupling of $\wst_{n}^{\beta_n}$ and $\mst_n$ obtained by using the same $\mathbf{U}$.

Then the following holds.
\begin{itemize}
\item[a)] If $\beta_n\gg n^{2}\log n$, then
$$\lim_{n\rightarrow\infty}\P(\wst_{n}^{\beta_n}=\mst_n)=1.$$
\item[b)] If $\beta_n\ll n^{2}$, then
$$\lim_{n\rightarrow\infty}\P(\wst_{n}^{\beta_n}\ne \mst_n)=1.$$
\end{itemize}
\end{theorem}
The proof can be found in Subsection \ref{subsec:wst_vs_mst}.

Theorems \ref{thm:ab_vs_inv} and \ref{thm:wst_vs_mst} give us for $n^{2}\log n\ll \beta_n\ll n^3$, that the Aldous-Broder and Prim's invasion algorithms typically build the same edges but in different orderings.

\paragraph{Typical diameter.}
For $\beta_n\ll n^2$, the models $\wst^{\beta_n}_n$ and $\mst_n$ typically differ, so we are interested in when some of their properties are still similar. For $n\log n\ll \beta_n\ll n^2$, we show in (\ref{formula:expected_diff}) that these models differ only by $o(n)$ edges, so it is reasonable to think that they behave the same way locally. 

However, it is natural to expect that, for determining the typical diameter of $\wst^{\beta_n}_n$, we have to study our random electric network more carefully than for understanding its local properties. As an analogy, the local limit of $\ust(G_n)$ for any sequence of $d_n$-regular graphs $G_n$ with $d_n\rightarrow \infty$, is the same \cite{nachmias2022local}, but the order of the typical diameter can vary, detailed in Subsection 7.1 of \cite{nachmias2022local}. In some cases, there are no such counterexamples, e.g., in Corollary 1.8 of \cite{hladky2024graphon}, it shown for dense graphs without a sparse cut, that the local convergence of the $\ust$ implies the GHP convergence for the scaling limit. Of course, if $\beta_n$ grows fast, then our random electric network contains sparse cuts, so we do not hope that this result can be generalized here, but we expect that although locally $\wst^{\beta_n}_n$ behaves like $\mst_n$ for $\beta_n\gg n\log n$, we think that it is not the case for the global geometry. So, how fast $\beta_n$ should grow to ensure that the typical diameter of $\wst^{\beta_n}_n$ grows like for $\mst_n$?

\begin{theorem} \label{thm:has_diam_mst}If 
$\beta_n\gg n^{4/3}\log n$, then for any $\eps>0$ there exist $c=c(\eps)$ and $C=C(\eps)$ such that
 $$\P\left(cn^{1/3}\le \diam(\wst_{n}^{\beta_n})\le Cn^{1/3}\right)\ge 1-\eps. $$
\end{theorem}
The proof is in Subsection \ref{subsec:wst_diam_abr}, where we follow the ideas of \cite{addario2009critical}. 

Their proof is based on that the component structure of the Erdős-Rényi graph $G_{n, p}$ with edges $E(G_{n, p}):=\{e\given U_e\le p\}$ is the same as the component structure of $\mst_n\cap G_{n, p}$. They could use this connection to determine $\E[\diam(\mst_n)]$ by understanding the longest paths of the some subgraphs of the Erdős-Rényi graphs.

We modify their method to work on $\wst_{n}^{\beta_n}$: by Lemma~\ref{lemma:E-R_comps_A-B}, any connected component of $G_{n, p}$ is connected w.r.t.~$E(G_{n, p+\frac{\log n}{\beta_n} })\cap \wst_{n}^{\beta_n}$, where $p+\frac{\log n}{\beta_n}$ is in the critical window $\left[\frac{1}{n}-\frac{O(1)}{n^{4/3}}, \frac{1}{n}+\frac{O(1)}{n^{4/3}}\right]$ of the Erdős-Rényi graph if $p$ is in the critical window and $\beta_n\gg n^{4/3}\log n$.

It will turn out that this is the crucial property needed for our proof. We think that the break down of our proof for $\beta_n=O(n^{4/3}\log n)$ is not a coincidence, and it is in line with the conjecture of Remark 6.2 of \cite{makowiec2023diameter} which can be translated to our setting by Lemma~\ref{lemma:connection-gamma-alpha}: for $\forall \alpha\in \left(1, \frac{4}{3}\right)$ and  $\beta_n=n^{\alpha}$, it is conjectured that the diameter of $\wst_n^{\beta_n}$ grows like $n^{a}$ for an $a\in \left(\frac{1}{3}, \frac{1}{2}\right)$. We detail this in Remark \ref{rmrk:intermediate}.

If $\beta_n\ll \frac{n}{\log n}$, then the typical diameter of $\wst_{n}^{\beta_n}$ grows like the typical diameter of $\ust_n$.
\begin{theorem} \label{thm:has_diam_ust}
If $\beta_n\ll \frac{n}{\log n}$, then for any $\eps>0$ there exist $c=c(\eps), C=C(\eps)$ such that
 $$\P\left(cn^{1/2}\le \diam(\wst_{n}^{\beta_n})\le Cn^{1/2}\right)\ge 1-\eps. $$   
\end{theorem}
The proof is in Subsection \ref{subsec:ust_diam}, and it is based on that our random electric network with conductances $\{\exp(-\beta_nU_e)\}_e$ is a balanced expander.

\paragraph{Expected total length and local limit.} For a spanning tree $T$, we denote its total length by $L(T):=\sum_{e\in T}U_e$. Of course, $\E\left[L(\ust_n)\right]=\frac{n-1}{2}$, and it is known that $\E\left[L(\mst_n)\right]\sim \zeta(3)$ as $n\rightarrow \infty$ \cite{frieze1985value, aldous1990random, steele2002minimal}.

How does $\E\left[L( \wst_n^{\beta_n})\right]$ behave? For example, if $\beta_n$ grows fast enough, then we will show that, with probability $\rightarrow 1$ fast, $\wst_n^{\beta_n}$ contains edges only with small labels, and it differs from $\mst_n$ only by few edges. Hence, these models have asymptotically the same expected total lengths.

\begin{theorem}\label{thm:length}We use the notation $L(T):=\sum_{e\in E(T)}U_e$.
    \begin{itemize}
        \item[a)] If $\beta_n\gg n\log n$, then $\E\left[L(\wst_n^{\beta_n})\right]\rightarrow \zeta(3)$ as $n\rightarrow\infty$. 
        
        \item[b)]If $\beta_n=O(n\log n)$, then $\E\left[L(\wst_n^{\beta_n})\right]=O\left(\frac{n\log n}{\beta_n+\log n}\right)$.
        \item[c)] If $\beta_n\ll \frac{n}{\log n}$, then $\E[L(\wst_n^{\beta_n})]=(1+o(1))\frac{n}{\beta_n}\frac{1-\beta_ne^{-\beta_n}-e^{-\beta_n}}{1-e^{-\beta_n}}$.
    \end{itemize}
\end{theorem}
We prove this theorem in Subsections \ref{subsec:mst_length} and \ref{subsec:ust_length}. 

In Theorem \ref{thm:length} c) of the first version of our paper, we have only shown that $\E\left[L(\wst_n^{\beta_n})\right]=\Omega\left(\frac{n}{\beta_n+1}\right)$. We asked whether this result can be strengthened which was answered positively in Theorem 1.4 of \cite{makowiec2024local}: they have obtained for the expected total length that $\E[L(\wst_n^{\beta_n})]=(1+o(1))\frac{n}{\beta_n}\frac{1-\beta_ne^{-\beta_n}-e^{-\beta_n}}{1-e^{-\beta_n}}$ if $\beta_n\ll \frac{n}{\log n}$. Their proof is similar to ours, and their improvement is that they have determined the asymptotics of the effective resistances between any pairs of vertices with probability $1-O(n^{-2})$, resulting in the asyptotics of the expected total length.

The main result of \cite{makowiec2024local} is that the local limit of $\wst_n^{\beta_n}$ is the same as the local limit of $\ust_n$ if $\beta_n\ll \frac{n}{\log n}$, and it is the same as the local limit of $\mst_n$ if $\beta_n\ge n (\log n)^{\gamma(n)}$, where $\gamma(n)\rightarrow\infty$ arbitrarily slowly. We show the agreement of the local limits of $\wst_n^{\beta_n}$ and $\mst_n$ even for $\beta_n\gg n \log n$, which is an improvement of the condition on $\beta_n$ in the case of fast-growing parameters:

\begin{theorem}\label{thm:local_limit} We denote by $o$ a uniformly chosen vertex from $V(K_n)$, by $B_{\mst_n}(o, r)$ and by $B_{\wst_n^{\beta_n}}(o, r)$ the ball of radius $r$ w.r.t.~the (unweighted) graph metric in $\mst_n$ and $\wst_n^{\beta_n}$, respectively.

If $\beta_n \gg n\log n$, then the local limit of $\wst_n^{\beta_n}$ is the same as the local limit of $\mst_n$, since
    $$\E\left[\P\left(B_{\mst_n}(o, r)=B_{\wst_n^{\beta_n}}(o, r)\Given \{U_e\}_{e\in E(K_n)}\right)\right]\rightarrow 1.$$
\end{theorem}
Similarly to their proof, we show that, conditioning on the labels $\{U_e\}_{e\in E(K_n)}$, we have that $|E(\wst^{\beta_n}_n)\triangle E(\mst_n)|=o(1)$ with probability $\rightarrow 1$ for $\beta_n\gg n\log n$ which is also an improvement in $\beta_n$: in Theorem 1.1 of \cite{makowiec2024local}, they have shown this for $\beta_n\gg n(\log n)^2$. Our proof is in Subsection \ref{subsec:mst_length}.

\paragraph{Our results on the model \cite{makowiec2023diameter}.}  For i.i.d.~edge weights of any fixed distribution, it is proven that the typical diameter of the weighted spanning tree of $G$ is of order $|V(G)|^{1/2+o(1)}$ for bounded degree expander graphs and finite boxes of $\Z^d$ with $d\ge 5$ \cite{makowiec2023diameter}. To illustrate that the condition about the bounded degrees cannot be relaxed, in their Proposition 6.1, they gave an example of an electric network on $K_n$ on which the $\wst$ and the $\mst$ typically agree. However, unlike our Theorem \ref{thm:wst_vs_mst}, they did not study the question of the phase transition of the agreement of models, and their method is not delicate enough to obtain this. 

In Section 6.1, they asked about the behavior of the $\wst$ for the random edge weights $w_e=\exp(U_e^{-\gamma})$ on the complete graph $K_n$. They conjectured that for some parameters, the $\wst$ and the $\mst$ typically differ, but their typical diameters grow similarly, which is in line with our results.

\begin{theorem}\label{thm:fixed_distr}
Consider $\c_{\mathrm{fix}, n, \gamma}:=\{\exp(U_e^{-\gamma})\}_{e\in E(K_n)}$, $\wst_{n}^{\mathrm{fix}, \gamma}:=\wst(\c_{\mathrm{fix}, n, \gamma})$ and $s\in V(K_n)$. We denote by $f_i$ and $g_i$ the $i^{th}$ edges added by the Aldous-Broder and Prim's invasion algorithms respectively, and by $\P(\cdot)$ and $\P_s(\cdot)$ the coupling between these models and algorithms started at $s$ obtained by using the same $\mathbf{U}$.
    \begin{itemize}
        \item[a)] For $\gamma>2$,
        $$\P_{s}(f_{i}= g_{i}\, \forall i\in \{1, \ldots, n-1\})=\lim_{n\rightarrow\infty}\P(\wst_{n}^{\mathrm{fix}, \gamma}=\mst_n)=1.$$
        \item[b)] For $1<\gamma<2$,        $$\P_{s}(f_{i}= g_{i}\, \forall i\in \{1, \ldots, n-1\})=0\text{, but still }\lim_{n\rightarrow\infty}\P(\wst_{n}^{\mathrm{fix}, \gamma}=\mst_n)=1.$$
        \item[c)] For $\frac{1}{3}<\gamma<1$, although
        $$\lim_{n\rightarrow\infty}\P(\wst_{n}^{\mathrm{fix}, \gamma}=\mst_n)=0,$$
        for any $\eps>0$ there exist $c=c(\eps), C=C(\eps)$ such that
 $$\P\left(cn^{1/3}\le \diam(\wst_{n}^{\mathrm{fix}, \gamma})\le Cn^{1/3}\right)\ge 1-\eps. $$
    \end{itemize}
\end{theorem}
In Subsection \ref{subsec:heavy_tail}, we check that this follows from our results on $\wst_n^{\beta_n}$ for some $\beta_n=n^{\gamma+1+o(1)}$.

\paragraph{\bf Open problems.} We list some open questions.

\noindent {\bf Problem 1.} What is the typical diameter of $\wst_n^{\beta_n}(K_n)$ for $\frac{n}{\log n}\ll \beta_n\ll n^{4/3}\log n$? We expect that it grows like $\Theta(n^{\alpha})$ for some $\alpha \in (\frac{1}{3}, \frac{1}{2})$. However, one has to understand the underlying random electric network more carefully in order to obtain results about the typical diameter for these parameters. It could be easier to show that, for any $\delta>0$, there exists an $\eps>0$ such that if $\gamma\in (0, \eps)$ and $\beta_n=n^{1+\gamma}$, then the typical diameter of $\diam(\wst_n^{\beta_n})$ grows at least like $n^{1/2-\delta}$, since the random walk on some parts of our electric network could be analyzed for these parameters, see Remark \ref{rmrk:intermediate}.

\noindent {\bf Problem 2.} In Theorems \ref{thm:ab_vs_inv} and \ref{thm:wst_vs_mst}, we have determined the parameters where the phase transitions occur up to a logarithmic factor. By understanding the behavior of the random walk on these random electric networks more carefully, can we improve these results to obtain the parameters for the phase transition up to a constant factor?

In \textbf{Problem 3} of the first version of our paper, we asked if Theorem \ref{thm:length} c) can be improved from $\E\left[L(\wst_n^{\beta_n})\right]=\Omega\left(\frac{n}{\beta_n+1}\right)$ to $\E\left[L(\wst_n^{\beta_n})\right]=\Theta\left(\frac{n}{\beta_n+1}\right)$ for $\beta_n\ll \frac{n}{\log n}$. We were optimistic, since in order to determine the expected total length, one has to understand the effective resistances between the endpoints of the edges, and  it is shown with a robust computation, that in any sequence of $d_n\rightarrow\infty$-regular (unweighted) graphs, the effective resistance between the endpoints of an edge is typically $(1+o(1))\frac{2}{d_n}$ \cite{nachmias2022local}. In Theorem 1.4 of \cite{makowiec2024local}, our question was answered positively:  by studying the effective resistances, 
they could show that $\E[L(\wst_n^{\beta_n})]=(1+o(1))\frac{n}{\beta_n}\frac{1-\beta_ne^{-\beta_n}-e^{-\beta_n}}{1-e^{-\beta_n}}$ for $\beta_n\ll \frac{n}{\log n}$.

\subsection{Outline of the paper}
In Section \ref{sec:det}, after recalling the background about the spanning tree models on deterministic electric networks, we introduce Lemma \ref{lemma:middlecond} which is our key tool to study the similarities of the  algorithms: Prim's invasion algorithm is a greedy algorithm, and this lemma can be seen as a `relaxed version of greediness’ for the Aldous-Broder algorithm. In Subsections \ref{subsec:determ_ab_vs_inv} and \ref{subsec:determ_wst_vs_mst}, we apply this lemma to characterize when the algorithms have the same building steps and when they build the same model.

In Section \ref{sec:wst_vs_mst}, we apply the results of Section \ref{sec:det} to our random electric network. In Subsection \ref{subsec:ideas}, we collect the intuition behind the behavior of the model for different parameters highlighting why the phase transitions around these specific $\beta_n$'s happen. Then we prove the phase transitions regarding the algorithms and the models in Subsections \ref{subsec:ab_vs_invasion} and \ref{subsec:wst_vs_mst}. 

The second half of Section \ref{sec:wst_vs_mst} is about studying the model when it is not the $\mst_n$, but $\beta_n$ still grows fast. As a corollary of Lemma \ref{lemma:middlecond}, we give an important connection between the connected components of the Erdős-Rényi graph and the edges of $\wst_n^{\beta_n}$ in Lemma \ref{lemma:E-R_comps_A-B}, which is used to study the expected total length in Subsection \ref{subsec:mst_length} and the typical diameter in Subsection \ref{subsec:wst_diam_abr}.

In Section \ref{sec:wst_vs_ust}, we study the behavior of the model for slowly growing $\beta_n$'s. We give a reasonable bound on the spectral gap of our random electric network in Subsection \ref{subsec:spectral_gap}, hence we can use the results on the diameter of $\wst$ on balanced expanders from the literature in Subsection \ref{subsec:ust_diam}. The expected total length of $\wst$ for slowly growing $\beta_n$'s is studied in Subsection \ref{subsec:ust_length}.

\subsection{Connection to \texorpdfstring{\cite{makowiec2024diameter}}{}}\label{subsec:similar_papers}
During the very last revision of our paper, two preprints \cite{makowiec2024diameter, makowiec2024local} about the same model containing very similar results to ours have been uploaded to arXiv. The main overlaps are the results about the typical diameters.

However, we generate $\wst$'s by the Aldous-Broder algorithm, unlike their Wilson's algorithm approach. Our paper puts more emphasis on fast-growing $\beta_n$'s: we determine the phase transitions about the agreement of the Aldous-Broder and Prim's invasion algorithms and about the agreement of the models. Moreover, we study the behavior of the expected total length of $\wst_n^{\beta_n}$.

Unlike our work, in \cite{makowiec2024diameter}, it is shown that the scaling limit of $\wst_n^{\beta_n}$ is Aldous's Brownian CRT w.r.t.~the GHP convergence for constant $\beta$'s. Moreover, they formulate a conjecture about the exact value of the exponent of the typical diameter in the intermediate diameter setting. In \cite{makowiec2024local}, the local convergence of the model for slowly-growing and fast-growing $\beta_n$ is shown. In order to do that, they study the effective resistances and the edge overlaps. 

The effective resistance results of \cite{makowiec2024local} improve our results on the expected total lengths, see Theorem 1.4 of the latest version of \cite{makowiec2024local} or our Theorem \ref{thm:length}. In the latest version of our paper, we include a proof about the agreement of the local limits of $\mst_n$ and $\wst_n^{\beta_n}$ which holds under a milder condition on $\beta_n$ than Theorem 1.3. of \cite{makowiec2024local}.

\subsection{Acknowledgements}

This work was partially supported by the ERC Synergy Grant No. 810115–DYNASNET.

I am extremely grateful to my PhD advisor, Gábor Pete, for introducing to me this model and asking about the comparison of the Aldous-Broder and the invasion algorithms and about the behavior of the typical diameter. Moreover, I thank him for the useful discussions and for the comments on the manuscript.

\section{\texorpdfstring{$\wst(\c)$}{} on deterministic electric networks}\label{sec:det}
\subsection{Basics of weighted spanning trees of deterministic electric networks}\label{subsec:det_intro}
\paragraph{Notation} An \emph{electric network} is a graph $G$ with nonnegative weights assigned to the edges of $G$. The weight of an edge $e=(u_1 ,u_2)\in E(G)$ is called \emph{conductance} and denoted by $c(e)=c(u_1,u_2)=c(u_2,u_1)\ge 0$. We use the notation of Chapter 2 of \cite{PTN}.

In this paper, the collection of conductances is denoted by  $\c:=\{c(e)\}_{e\in E(G)}$, and the electric network on $G$ with conductances $\c$ by $(G, \c)$. 

We consider the \emph{random walk on the electric network $\c$ started at a vertex $v\in V(K_n)$}, i.e.,~the reversible Markov chain on $V(G)$ with transition probabilities $p(x,y)=c(x,y)/\sum_{(x,y)\in E(G)}c(x,y)$, and the probabilities of the events in this walk are denoted by $\P_{\c, v}(\cdot)$. We write $\P_{v}(\cdot)$ instead of $\P_{\c, v}(\cdot)$ if it does not make confusions. This random walk is interesting for us since the Aldous-Broder algorithm generating the weighted spanning trees on electric networks uses it.

For $S\subseteq V(G)$ the edge boundary of $S$ is $\partial S:=E(S, S^c)$, and the induced subgraph of $G$ on $S$ is $G[S]$.

As usual, for two sequences of positive numbers $\{a_n\}_{n}, \{b_n\}_{n}$, by $a_n\gg b_n$ and by $b_n=o(a_n)$, we mean $\lim_{n\rightarrow \infty}b_n/a_n=0$. We write $b_n=O(a_n)$ and $a_n=\Omega(b_n)$ if $\limsup_{n\rightarrow\infty}b_n/a_n<\infty$. If both  $a_n=O(b_n)$ and $a_n=\Omega(b_n)$ holds, then it is abbreviated as $a_n=\Theta(b_n)$.

\paragraph{The weighted spanning tree and the Aldous-Broder algorithm}
\begin{defi}
The \emph{weighted spanning tree $\wst(\c)$}   puts to any spanning tree $T$ of $K_n$ the probability
$$\P(\wst(\c)=T):= \frac{\prod_{e\in T}c(e)}{Z(\c)} \text{ with }Z(\c):=\sum_{\substack{T'\text{ is a}\\\text{ spanning tree of } K_n}}\prod_{e\in T'}c(e).$$    
\end{defi}

Then, considering some i.i.d.~$\unif[0,1]$ random variables $\{U_e\}_{e\in E(K_n)}$, $\wst^{\beta_n}(K_n)$ introduced in Definition \ref{def:wst_beta} is $\wst(\{\exp(-\beta_nU_{e})\}_{e\in E(K_n)})$, i.e.,~it is the weighted spanning tree w.r.t.~the random environment $\{\exp(-\beta_nU_{e})\}_{e\in E(K_n)}$.

It is known that $\wst(\c)$ can be generated by Wilson's and the Aldous-Broder algorithms, detailed in for example Theorem 4.1 and Corollary 4.9 of \cite{PTN}. They are useful for both simulations and proving theorems. The Aldous-Broder algorithm is more suitable for our purposes, because of its similarity to the Prim's invasion algorithm as we will see later.

\begin{algorithm}[Aldous-Broder algorithm] Given the electric network $\c$ on $K_n$ and a starting vertex $s\in V(K_n)$, we define the \emph{vertices} $v_0, v_1, \ldots, v_{n-1}$
and the \emph{edges of the tree} $f_1,\ldots, f_{n-1}$ built by the Aldous-Broder algorithm recursively. We begin with $v_0:=s$. 

For $0\le k\le n-2$, suppose that we already have defined $v_0, \ldots, v_{k}$ and $f_1, \ldots, f_{k}$. With the notation $V_k:=\{v_0, \ldots, v_{k}\}$ and $F_k:=\{f_1, \ldots, f_{k}\}$, let $f_{k+1}$ be the first edge between $V_k$ and $V(K_n)\setminus V_k$ crossed by the random walk on the electric network $\c$, and $v_{k+1}$ be simply the endpoint of $f_{k+1}$ which is not in $V_k$. 
\end{algorithm}

\begin{theorem}[\cite{aldous, broder1989generating}] Consider an electric network $\c$ on $K_n$ and a starting vertex $s\in V(K_n)$. We denote by $F_{n-1}$ the edges of the tree generated by the Aldous-Broder algoritm.

Then $(V(K_n), F_{n-1})$ is distributed as $\wst(\c)$.
\end{theorem}
This statement can also be found in \cite{PTN} as Corollary 4.9.

Electric networks are especially useful for us, since there is a fundamental relationship between weighted spanning trees and electric networks first discovered by Kirchhoff in 1847 in the case of the $\ust$. The effective resistance between two vertices $u, v\in V(K_n)$ can be defined in different ways, we choose the definition $\mathcal{R}_{\mathrm{eff}}(u\leftrightarrow v):=(C_u\P_u(\tau_v<\tau^+_u))^{-1}$, where $X_u$ is the random walk w.r.t.~$\c$, $\tau_u^+:=\min\{n\ge 1:\; X_n=u\}$, $\tau_v:=\min\{n\ge 0:\; X_n=v\}$ and $C_u=\sum_{u\sim v}c(u, v)$ is the invariant distribution. The heuristics about the effective resistances come from understanding electric currents and voltages. Kirchhoff’s formula for $\wst$'s can be found in Section 4.2 of \cite{PTN}: 
\begin{theorem*}[Kirchhoff’s Effective Resistance Formula, \cite{kirchhoff1847ueber}] For any $u, v\in V(K_n)$, one has that
$\P((u, v)\in \wst(\c))=c(u,v)\mathcal{R}_{\mathrm{eff}}(u\leftrightarrow v).$    
\end{theorem*}

\paragraph{The maximum spanning tree and Prim's invasion algorithm}
\begin{assum}\label{assum}
We suppose that $V(G)$ is connected w.r.t.~the edge set $\{e:\; c(e)>0\}$, and all the positive conductances are different.     
\end{assum}
\begin{defi}
We define $\maxst(\c)$ as the spanning tree $T$ which maximizes $\prod_{e\in T}c(e)$. 
\end{defi}
Under Assumption \ref{assum}, $\maxst(\c)$ is well-defined since there is a unique spanning tree $T$ maximizing $\sum_{e\in T}c(e)$.

We introduced $\maxst$ to have $\mst(\{U_{e}\}_{e\in E(K_n)})=\maxst(\{\exp(-\beta_nU_{e})\}_{e\in E(K_n)})$.

Although the maximum spanning tree is conventionally defined to maximize $\sum_{e\in T}c(e)$ instead of $\prod_{e\in T}c(e)$, this is not an important detail. Indeed, during the generation of the maximum sum spanning tree with Prim's algorithm, only the ordering of $\c$ is important and not the exact values of the conductances. Maximizing $\prod_{e\in T}c(e)$ is the same as maximizing $\sum_{e\in T}\log c(e)$, and $x\rightarrow \log c(x)$ is strictly increasing. Hence $\maxst(\c)$ is also the spanning tree maximizing $\sum_{e\in T}c(e)$.

\begin{algorithm}[Prim's algorithm/invasion tree]
Given an electric network $\c$ on $K_n$ and a starting vertex $s\in V(K_n)$, we define the \emph{vertices} $w_0, w_1, \ldots, w_{n-1}$ and the \emph{edges of the invasion tree} $g_1,\ldots, g_{n-1}$ recursively. We start with $w_0:=s$. 

For $0\le k\le n-2$, suppose that we already have defined $w_0, \ldots, w_{k}$ and $g_1, \ldots, g_{k}$. With $W_k:=\{w_0, \ldots, w_{k}\}$ and $G_k:=\{g_1, \ldots, g_{k}\}$, let $g_{k+1}$ be the edge with maximal conductance between $W_k$ and $V(K_n)\setminus W_k$, and $w_{k+1}$ is the endpoint of $g_{k+1}$ which is not in $W_k$.

We call $(V(K_n), G_{n-1})$ as the \emph{invasion tree}. 
\end{algorithm}

\begin{prop}[\cite{prim1957shortest}] Consider an electric network $\c$ on $K_n$. We denote by $G_{n-1}$ the edges of the invasion tree. Then $(V(K_n), G_{n-1})=\maxst(\c)$.
\end{prop}

\paragraph{Notation regarding the algorithms}
\begin{notation}\label{notation:alg_v-e} Consider an electric network $(G, \c)$ satisfying Assumption \ref{assum}.

We denote by $V_k=\{v_{0}, \ldots, v_{k}\}$ and $F_k:=\{f_1, \ldots , f_k\}$ the first $k$
vertices and edges of $\wst(\c)$ in the order they were built by the Aldous-Broder algorithm on $(G, \c)$ started at $s$.

We write $W_k=\{w_{0}, \ldots, w_{k}\}$ and $G_k=\{g_{1}, \ldots, g_{k}\}$ for the first $k$ vertices and edges added to $\maxst(\c)$ by Prim's invasion algorithm on $(G, \c)$ started at $s$, in the order they were chosen by the algorithm. 

The edge with second largest conductance from $\partial W_k$ is denoted by $h_k$.
\end{notation}
Of course, $g_k$ and $h_k$ are the edges with the two largest conductances of $\partial W_k$.

\begin{remark*}
    Although in this section we formulate 
    our statements for $K_n$, these still hold for any $G$ with $|V(G)|=n$. Indeed, any electric network $(G, \c)$ can be extended to an electric network on $K_n$ by choosing $c(e)=0$ for $e\in E(K_n)\setminus E(G)$, not modifying the transition probabilites of the random walk.
\end{remark*}

\paragraph{Ideas of the section}
In Subsection \ref{subsec:middlecond}, we introduce Lemma \ref{lemma:middlecond}, which is the key tool of this section connecting the algorithms generating $\maxst(\c)$ and $\wst(\c)$: Prim's invasion algorithm is by definition greedy, and we can think of the statement of Lemma \ref{lemma:middlecond} as `a relaxed version of greediness’ holds for the Aldous-Broder algorithm. To prove this, we study how small conductances can appear in the random walk on $\c$ until exploring fully a given vertex set.

In Subsection \ref{subsec:determ_ab_vs_inv}, we study the probability of the agreement of the Aldous-Broder and Prim's invasion algorithms in each building step. It turns out that the algorithms build the same edges all the time with high probability if and only if $c(g_j)$ and $c(h_j)$ significantly differ for every $j\le n-1$.

In Subsection \ref{subsec:determ_wst_vs_mst}, we estimate the probability of $\wst(\c)$ and $\mst(\c)$ having exactly the same edges. We obtain that the models are typically the same if and only if adding any external edge to the $\maxst(\c)$, in the appearing cycle---called the fundamental cycle---any other edge has a lot larger conductance than this external edge. 

At first it might seem surprising that the agreement of the algorithms and the agreement of the models depend on such different properties of the electric network. However, $\maxst(\c)$ consists of edges of big conductances, and Lemma \ref{lemma:middlecond} implies that the edges used by Aldous-Broder algorithm are not arbitary: they cannot have small conductances. We note that some electric networks naturally appear with the phenomenon that whenever the two algorithms differ, the Aldous-Broder algorithm still adds an edge of $\maxst(\c)$ and the current invasion edge will be added later by the Aldous-Broder algorithm. Indeed, in Section \ref{sec:wst_vs_mst}, we will see that $\c:=\{\exp(-\beta_nU_e)\}_{e\in E(K_n)}$ gives us examples of electric networks where these conditions are indeed really different, resulting that the Aldous-Broder and the Prim's invasion algorithm build the same edges but in different order for $n^2\log n\ll \beta_n\ll n^3$.

\subsection{The `relaxed greediness’ of the Aldous-Broder algorithm}
\label{subsec:middlecond}

\begin{notation} \label{notation:bcond} Consider an electric network $(G, \c)$, $a>0$ and $\eps_0\in (0,1)$. We define the edge sets
\begin{align*}
    &E_{\bcond}:=\{e\in E(G)\given c(e)\ge a\},\\
    &E_{\mcond}:=\{e\in E(G)\given \eps_0 a<c(e)<a\}\text{ and }\\
    &E_{\scond}:=\{e\in E(G)\given c(e)\le \eps_0 a\}.
\end{align*}

For an $E\subseteq E(G)$, $\C_s^E$ denotes the connected component of $s$ in the graph $(V(G), E)$.
\end{notation}

Prim's invasion algorithm is greedy, therefore if $u$ and $v$ are in the same connected component w.r.t.~$E_{\bcond}$, i.e.,~there is a path between $u$ and $v$ containig edges only from $E_{\bcond}$, then the unique path connecting $u$ and $v$ in $\maxst(\c)$ still consists of edges only from $E_{\bcond}$.

We show that a relaxation of this greediness holds also for the Aldous-Broder algorithm: with high probability, if $u$ and $v$ are in the same connected component w.r.t.~$E_{\bcond}$, then the unique path in $\wst(\c)$ connecting $u$ and $v$ contains edges only from $E_{\bcond}\cup E_{\mcond}$.

\begin{lemma}\label{lemma:middlecond}
We consider an electric network $(K_n, \c)$, $t\in V(K_n)$, $a>0$ and $\eps_0\in (0,1)$. We use Notation \ref{notation:bcond}. The hitting time of a vertex $v\in V(K_n)$ is denoted by $\tau_v$, and the first time crossing an edge $e\in E(K_n)$ is denoted by $\tau_e$.

Then, for the random walk on $\c$ started at $t$, we have that
\begin{gather}\label{formula:hitting}
    \P_{t}\left(\max_{u\in V\left(\C_{t}^{E_\bcond}\right)} \tau_u > \min_{e\in E_{\scond}}\tau_e\right) \le\eps_0^{1/2}n^{7/2}.
\end{gather}
Consequently,
\begin{gather}\label{formula:relaxed_greedy}
    \P\left(V\left(\C_{t}^{E_{\bcond}}\right) \subseteq V\left(\C_{t}^{E(\wst(\c))\setminus E_{\scond}}\right)\right) \ge 1-\eps_0^{1/2}n^{7/2}.
\end{gather}
\end{lemma}
\begin{remark*}
For our applications collected in Section \ref{sec:wst_vs_mst}, the exact values of exponents appearing in $\eps_0^{1/2}n^{7/2}$ are not important as soon as they are positive and finite. We will choose $\eps_0$ to be, for example, $o(n^{-\kappa})$ for a $\kappa>7/2$.
\end{remark*}
\begin{remark*} In words, (\ref{formula:relaxed_greedy}) tells us that the Aldous-Broder algorithm satisfies the `relaxed greediness’ mentioned right before the lemma.

In Figure \ref{fig:middlecond}, this `relaxed greediness’ is illustrated: on $\left\{V\left(\C_{t}^{E_{\bcond}}\right) \subseteq V\left(\C_{t}^{E(\wst(\c))\setminus E_{\scond}}\right)\right\}$, the subset $V\left(\C_{t}^{E_{\bcond}}\right)$ is connected w.r.t.~$E(\wst(\c))\setminus E_{\scond}$. Notice that $\wst(\c)\left[C_{t}^{E_{\bcond}}\right]$ can be disconnected, suggesting that, during the application of this property, one has to be more careful compared to the use of classical greediness, where $\mst(\c)\left[C_{t}^{E_{\bcond}}\right]$ is automatically connected.  
\end{remark*}

\begin{figure}\label{fig:middlecond}
\centering
\begin{tikzpicture}
\draw[blue, anchor=west] node at (-3, -1.2){$E_{\bcond}$};
\draw[lime!60!teal, anchor=west] node at (-1, -1.2){$E_{\mcond}$};
\draw[magenta, anchor=west] node at (1.2, -1.2){$E_{\scond}$};
\node[ellipse, draw,
    minimum width = 3cm, 
    minimum height = 1.8cm, blue, line width=1pt] at (0,0) {};

\coordinate (1) at (-1, 0);
\coordinate (2) at (-0.5, 0.4);
\coordinate (3) at (0.5, 0.3);
\coordinate (4) at (0.9, -0.2);
\coordinate (5) at (0, -0.7);
\coordinate (6) at (1.8, 0.8);
\coordinate (7) at (2, 0.5);
\coordinate (8) at (1.8, -0.6);
\coordinate (9) at (-2, -0.7);
\coordinate (10) at (-2.5, 0.3);
\coordinate (11) at (-2, 0.7);

\foreach \x in {1,...,11} \filldraw[] (\x) circle (1pt);

\draw[blue, line width=2pt](1) -- (2);
\draw[blue, line width=2pt](2) -- (3);
\draw[blue, line width=2pt](3) -- (4);
\draw[blue, line width=2pt](4) -- (5);
\draw[blue, line width=2pt](1) -- (4);
\draw[blue, line width=2pt](3) -- (5);
\draw[blue, line width=2pt](6) -- (7);
\draw[blue, line width=2pt](10) -- (11);

\draw[lime!60!teal, line width=1pt](2) -- (6);
\draw[lime!60!teal, line width=1pt](3) -- (7);
\draw[lime!60!teal, line width=1pt](11) -- (2);
\draw[lime!60!teal, line width=1pt](10) -- (1);
\draw[lime!60!teal, line width=1pt](2) -- (5);
\draw[lime!60!teal, line width=1pt](5) -- (1);
\draw[lime!60!teal, line width=1pt](1) -- (9);

\draw[magenta, line width=1pt, dotted](5) -- (9);
\draw[magenta, line width=1pt, dotted](4) -- (8);
\draw[magenta, line width=1pt, dotted](7) -- (8);
\draw[magenta, line width=1pt, dotted](1) -- (11);
\draw[magenta, line width=1pt, dotted](1) -- (3);
\end{tikzpicture} \ \ \ \
\begin{tikzpicture}
\node[ellipse, draw,
    minimum width = 3cm, 
    minimum height = 2cm, blue, line width=1pt] at (0,0) {};

\coordinate (1) at (-1, 0);
\coordinate (2) at (-0.5, 0.4);
\coordinate (3) at (0.5, 0.3);
\coordinate (4) at (0.9, -0.2);
\coordinate (5) at (0, -0.7);
\coordinate (6) at (1.8, 0.8);
\coordinate (7) at (2, 0.5);
\coordinate (8) at (1.8, -0.6);
\coordinate (9) at (-2, -0.7);
\coordinate (10) at (-2.5, 0.3);
\coordinate (11) at (-2, 0.7);

\foreach \x in {1,...,11} \filldraw[] (\x) circle (1pt); 

\draw[blue, line width=2pt](1) -- (2);
\draw[blue, line width=2pt](3) -- (4);
\draw[blue, line width=2pt](6) -- (7);

\draw[lime!60!teal, line width=1pt](2) -- (6);
\draw[lime!60!teal, line width=1pt](3) -- (7);
\draw[lime!60!teal, line width=1pt](11) -- (2);
\draw[lime!60!teal, line width=1pt](2) -- (5);
\draw[lime!60!teal, line width=1pt](1) -- (9);
\end{tikzpicture}
\caption{Left: $\mathcal{C}^{E_{\bcond}}_{t}$ of an electric network $(G, \c)$ is in the blue circle; right: the subtree generated by the Aldous-Broder algorithm until $\mathcal{C}^{E_{\bcond}}_{t}$ is completely explored.}
\end{figure}
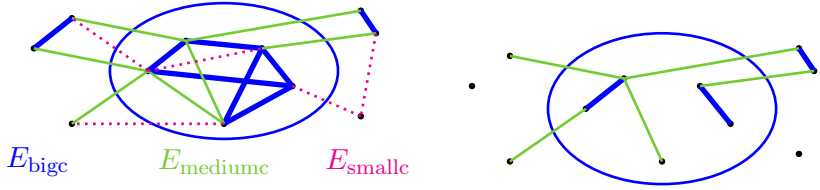

\begin{proof}[Proof of Lemma \ref{lemma:middlecond}]
Note that (\ref{formula:relaxed_greedy}) is a direct consequence of (\ref{formula:hitting}), so it is enough to prove (\ref{formula:hitting}).

We write $\C_{t}:=\C_{t}^{E_{\bcond}}$ to denote the connected component of $t$ w.r.t.~$E_\bcond$. We start by showing that, for $\forall u\in V(\C_{t})$ and $\forall (v_1, v_2)\in E_{\scond}$,
$$\P_{t}(\tau_{u}>\tau_{(v_1, v_2)})\le \sqrt{2\eps_0n}.$$

Firstly, we fix a $u\in V(\C_{t})$ and a $(v_1, v_2)\in E_{\scond}$, and construct $(G', \c)$ from $(K_n, \c)$. Let $V(G')=V(K_n)\cup \{w\}$, and the edges are obtained from $E(K_n)$ by adding $(w, v_1)$ and $(w, v_2)$ and by deleting $(v_1, v_2)$. We choose $c(w, v_1)=c(w, v_2)$ of $(G', \c)$ to be the same as $c(v_1, v_2)$ of $(K_n, \c)$, and denote by $v$ the voltage on $G'$ w.r.t.~$\P_{s, G'}$  with the boundary conditions $v(u)=0$, $v(w)=1$, which is a harmonic function on $V(G')\setminus \{u, w\}$.

By the natural coupling between the random walk on $K_n$ up to $\tau_{(v_1, v_2)}$ and the random walk on $G'$ up to $\tau_{w}$, and since the map $a\mapsto \P_{a, G'}(\tau_{u}>\tau_{w})$ is harmonic on $V(G')\setminus \{u, w\}$, it is clear that
\begin{gather*}
\P_{t, K_n}(\tau_{u}>\tau_{(v_1, v_2)})=\P_{t, G'}(\tau_{u}>\tau_{w})=v(t).
\end{gather*}

By Dirichlet's principle detailed in Exercise 2.13 of \cite{PTN},
\begin{gather}\label{formula:dirichlet}
\mathcal{R}_{G'}(u,w)^{-1}=\min_{F\text{ with }F(u)=0, F(w)=1}\sum_{e\in G'} c(e)(dF(e))^2,
\end{gather}
and the minimizer is the voltage $v$ defined above. Substituting in the function $F|K_n:\equiv 0$ and $F(w):=1$,
\begin{gather*}
    \mathcal{R}_{G'}(u,w)^{-1}\le \sum_{v\in V(K_n)} c(v, w)\cdot 1= c(w, v_1)+c(w, v_2)\le 2\eps_0a.
\end{gather*}

Since $u\in \C_{s}$, there exists a path $P(u, t)$ between $u$ and $t$ formed by edges only from $E_{\bcond}$, and
\begin{gather*}
2\eps_0a\ge \mathcal{R}_{G'}(u,w)^{-1}= \sum_{e\in G'} c(e)(dv(e))^2\ge \sum_{e\in P(u, t)}a(dv(e))^2\ge a\frac{v(t)^2}{n}.
\end{gather*}
The last inequality comes from the inequality between the arithmetic and the quadratic means, using that $\sum_{e\in P(u,t)}dv(e)=v(t)-v(u)=v(t)$, and that $P(u, t)$ contains at most $n$ edges. Hence, we indeed obtain that $$\P_{t}(\tau_{u}>\tau_{(v_1, v_2)})=v(t)\le\sqrt{2\eps_0n}.$$

We get (\ref{formula:hitting}) by the union bound:
\begin{gather*}
    \P_{t}\left(\max_{u\in V(\C_{t})} \tau_u > \min_{e\in E_{\scond}}\tau_e\right)\le \sum_{u\in V(\C_{t})}\sum_{e\in E_{\scond}}\P_{t}(\tau_{u}>\tau_{(v_1, v_2)})\le n\binom{n}{2}\sqrt{2\eps_0n}\le n^{7/2}\eps_0^{1/2}.
\end{gather*}
\end{proof}
 
\subsection{Agreement of the algorithms and the conductances on the boundaries}
\label{subsec:determ_ab_vs_inv}
We turn to give a sufficient condition on electric networks, stated in Proposition \ref{prop:ab_is_prim}, to guarantee the Aldous-Broder and Prim's invasion algorithm to build the same edges in every step.

\begin{lemma}\label{lemma:invasion_path}
Consider an electric network $(K_n, \c)$ satisfying Assumption \ref{assum}, $0\le k\le n-2$ and $s\in V(K_n)$. As in Notation \ref{notation:alg_v-e}, $(W_k, G_k)$ is the subtree built by the first $k$ steps of Prim's invasion algorithm.

Then, for $\forall a\in W_k$ and $\forall b\notin W_k$, we have that
$$c(e)>c(a,b) \text{ for any }e\in P(w_k, a),$$
where we denote by $P(w_k, a)$ the unique path between $w_k$ and $a$ using edges only from $G_k$.
\end{lemma}
\begin{proof}
Suppose that $a\ne w_k$, since for $a=w_k$ the statement is empty. As we will see below, the lemma will easily follow from that $w_k$ is added to the invasion tree after $a$ but before $b$.

We denote by $e^*$ the edge of $P(w_k, a)$ with minimal conductance, i.e.,~$c(e^*)=\min_{e\in P(w_k, a)}c(e)$. There is a unique $u\in P(w_k, a)$ such that $u$ is the common ancestor of $w_k$ and $a$ in the tree $(W_k, G_k)$ with root $s=w_0$. We consider
\begin{gather*}
    W^*:=\{v\in V(K_n)\given v\text{ is reachable from $u$ using only edges $e$ satisfying $c(e)>c(e^*)$}\}.
\end{gather*} 
It is easy to check that if $t_2\notin W^*$ is a descendant of $u$ in the invasion tree and $t_1\in W^*$, then $t_1$ is added earlier to the invasion tree than $t_2$.

If $e^*\in P(u, a)$, then $w_k\in W^*$ but $a\notin W^*$, contradicting $a$ being added earlier than $w_k$. So we obtain that $e^*\in P(w_k, u)$, hence $w_k\notin W^*$ and $a\in W^*$. Since $w_k$ is added to the invasion tree earlier than $b$, we have $b\notin W^*$, implying that $c(e^*)\ge c(a, b)$. Because of Assumption \ref{assum} this cannot be an equality.
\end{proof} 

\begin{prop}\label{prop:ab_is_prim}
Consider an electric network $(K_n, \c)$ satisfying Assumption \ref{assum}, $s\in V(K_n)$ and $0<\eps_0<1$ such that $2\eps_0^{1/2}n^{9/2}\le 0.7$. We use Notation \ref{notation:alg_v-e}.

If $\frac{c(h_j)}{c(g_j)}<\eps_0 \ \forall j=0, \ldots, n-2$, then
$$\P_{s}\left(f_j= g_j\ \forall j\in \{0, \ldots, n-1\}\right)\ge 1-2\eps_0^{1/2}n^{9/2}.$$
\end{prop}

\begin{proof} Firstly, fix a $0\le j\le n-2$ and use Notation \ref{notation:bcond} with $a:=c(g_{j+1})$. 

Then $\partial W_j\setminus E_{\scond}=\{g_{j+1}\}$. Moreover, Lemma \ref{lemma:invasion_path} and $\frac{h_{j+1}}{g_{j+1}}<\eps_0$ imply that $w_{j+1}\in \C_{w_j}^{E_{\bcond}}$, so $\C_{w_j}^{E_{\bcond}}\not\subseteq W_j$. Then we apply Lemma \ref{lemma:middlecond} with $t:=w_j$, and get that the first edge through which the random walk leaves $W_j$ is $g_{j+1}$ with probability $\ge 1-\eps_0^{1/2}n^{7/2}$, i.e.,
$$\P\left(f_{j+1}=g_{j+1}\Given f_i=g_i \;\forall i\le j\right)\ge 1-\eps_0^{1/2}n^{7/2}.$$

By the chain rule for conditional probabilities, $e^{-2x}\le 1-x\le e^{-x}$ for $0<x<0.7$ and $2\eps_0^{1/2}n^{9/2}\le 0.7$,
\begin{align*}
    \P_{s}(f_{j}=g_{j}\, \forall j\le n-1)&=\prod_{j=0}^{n-2} \P_{s}(f_{j+1}= g_{j+1}\, |\,f_{i}=g_{i}\,\forall i\le j)\ge\left(1-\eps_0^{1/2}n^{7/2}\right)^{n-1}\\
    &\ge \exp\left(-(n-1)2\eps_0^{1/2}n^{7/2}\right)\ge 1-(n-1)2\eps_0^{1/2}n^{7/2}\ge 1-2\eps_0^{1/2}n^{9/2}.
\end{align*}
\end{proof}

We show that a modification---replacing $n^{-9/2}$ by a constant $\in (0, 1)$---of the condition of the previous proposition is sufficient if we have i.i.d.~random labels. To do that, we do not prove the converse statement, but a sightly weaker one: we do not guarantee the disagreement of the algorithms for any electric network $(K_n, \c)$ not satisfying the conditions of Proposition \ref{prop:ab_is_prim}, but for a network obtained as an appropriate, slight reordering the labels of $\c$.

\begin{prop}\label{prop:ab_is_not_prim}
Consider an electric network $(K_n, \c)$ satisfying Assumption \ref{assum}, $s\in V(K_n)$, $0<\eps_0<1$ and $0\le m\le n-2$. We use Notation \ref{notation:alg_v-e}, and let $A_{m}$ be the event that the random walk first leaves $W_m$ through $f_{m+1}.$ 

If $\frac{c(h_{m+1})}{c(g_{m+1})}>\eps_0$, then
$$\P_{s}(f_{m+1}=g_{m+1}\given f_j=g_j \;\forall j\le m)=\P_{w_m}(A_{m})\le \frac{1}{1+\eps_0}$$
or
$$\P_{s}^{h_{m+1}\leftrightarrow g_{m+1}}(f_{m+1}=g_{m+1}\given f_j=g_j \;\forall j\le m)=\P_{w_m}^{h_{m+1}\leftrightarrow g_{m+1}}(A_{m})\le \frac{1}{1+\eps_0},$$
where $\P_{s}^{h_{m+1}\leftrightarrow g_{m+1}}$ denotes the probabilities on the electric network with conductances \begin{gather*}c^{h_{m+1}\leftrightarrow g_{m+1}}(e):=\begin{cases}c(h_{m+1})&\text{ if }e=g_{m+1},\\
    c(g_{m+1})&\text{ if }e=h_{m+1},\\
    c(e)&\text{otherwise.}
    \end{cases}
\end{gather*}
\end{prop}

\begin{remark*}
Applying these results w.r.t.~our random conductances $\{\exp(-\beta_nU_e)\}_{e\in E(K_n)}$ in Theorem \ref{thm:ab_vs_inv}, we will get the agreement of the algorithms for $\beta_n\gg n^3\log n$, and disagreement for $\beta_n \ll n^3$. Note that, for $n^3\ll \beta_n\ll n^3\log n$, we cannot compare the algorithms this way, and we will see that this extra $\log n$ factor for $\beta_n$ comes in since we need $\frac{h_{m+1}}{g_{m+1}}<n^{-9/2}$ for the agreement results, and $\frac{h_{m+1}}{g_{m+1}}>\eps_0$ with a fixed $0<\eps_0<1$ for the disagreement results.
\end{remark*}

\begin{proof}
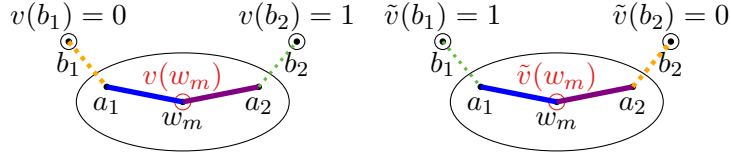
\begin{figure}\label{fig:ab_prim_noteq}
\centering
\begin{tikzpicture}
\filldraw[] (0, 0) circle (1pt) node[anchor=north] (w_k) {$w_m$};
\draw[gray!20!red] (0, 0) circle (3pt) node[anchor=south] {$v(w_m)$};
\filldraw[] (-1, 0.2) circle (1pt) node[anchor=north] (a_1) {$a_1$};
\filldraw[] (-1.5, 0.8) circle (1pt) node[anchor=north] (b_1) {$b_1$};
\draw[] (-1.5, 0.8) circle (3pt) node[anchor=south] {$v(b_1)=0$};
\filldraw[] (1, 0.2) circle (1pt) node[anchor=north] (a_2) {$a_2$};
\filldraw[] (1.5, 0.8) circle (1pt) node[anchor=north] (b_2) {$b_2$};
\draw[] (1.5, 0.8) circle (3pt) node[anchor=south] {$v(b_2)=1$};
\node[ellipse, draw,
    minimum width = 2.8cm, 
    minimum height = 1.3cm] at (0,0) {};

    \draw[orange!70!yellow, line width=1.7pt, dotted]         (-1.5, 0.8) -- (-1, 0.2);
    \draw[blue, line width=2pt]          (-1, 0.2) -- (0, 0);
    \draw[violet, line width=2pt]         (0,0) -- (1, 0.2);
    \draw[teal!50!lime, line width=1.2pt, dotted]         (1, 0.2) -- (1.5, 0.8);
\end{tikzpicture}
\begin{tikzpicture}
\filldraw[] (0, 0) circle (1pt) node[anchor=north] (w_k) {$w_m$};
\draw[gray!20!red] (0, 0) circle (3pt) node[anchor=south] (w_m) {$\tilde{v}(w_m)$};
\filldraw[] (-1, 0.2) circle (1pt) node[anchor=north] (a_1) {$a_1$};
\filldraw[] (-1.5, 0.8) circle (1pt) node[anchor=north] (b_1) {$b_1$};
\draw[] (-1.5, 0.8) circle (3pt) node[anchor=south] {$\tilde{v}(b_1)=1$};
\filldraw[] (1, 0.2) circle (1pt) node[anchor=north] (a_2) {$a_2$};
\filldraw[] (1.5, 0.8) circle (1pt) node[anchor=north] (b_2) {$b_2$};
\draw[] (1.5, 0.8) circle (3pt) node[anchor=south] {$\tilde{v}(b_2)=0$};
\node[ellipse, draw,
    minimum width = 2.8cm, 
    minimum height = 1.3cm] at (0,0) {};

    \draw[teal!50!lime, line width=1.2, dotted] (-1.5, 0.8) -- (-1, 0.2);
    \draw[blue, line width=2pt] (-1, 0.2) -- (0, 0);
    \draw[violet, line width=2pt] (0,0) -- (1, 0.2);
    \draw[orange!70!yellow, line width=1.7pt, dotted] (1, 0.2) -- (1.5, 0.8);
\end{tikzpicture}
\caption{In the proof of Proposition \ref{prop:ab_is_not_prim}, we control the probability of `the Aldous-Broder algorithm does not build the invasion edge in at least one of the electric networks’ by checking that $v(w_m)+\tilde{v}(w_m)$ is big enough. The same color of segments denotes the same difference of voltages.}
\end{figure}

Use the notation $g:=g_{m+1}=(a_1, b_1)$, $h:=h_{m+1}=(a_2, b_2)$, with $a_1, a_2\in W_m$, $L:=E(K_n[W_m])$ and $L':=L\cup(a_1, b_1)\cup (a_2, b_2)$.

Suppose that $b_1\ne b_2$. Consider a potential $v$ w.r.t.~$\P_{s, L'}$ with $v(b_1)=0$, $v(b_2)=1$, and a potential $\tilde{v}$ w.r.t.~$\P^{h\leftrightarrow g}_{s, L'}$ with $\tilde{v}(b_1)=1$, $\tilde{v}(b_2)=0$. We start by giving an upper bound on $v(w_m)+\tilde{v}(w_m)$. Figure \ref{fig:ab_prim_noteq} shows which voltage differences agree in $v$ and in $\tilde{v}$.
The same amount of current flows through $b_1$ (or $b_2)$ in both networks, which is denoted by $i$, and the resistances by $r(e):=\frac{1}{c(e)}$. By Chapter 2 of \cite{PTN},
\begin{align}
v(a_1)-v(b_1)=\tilde{v}(a_2)-\tilde{v}(b_2)&=ir(g),\nonumber\\
v(b_2)-v(a_2)=\tilde{v}(b_1)-\tilde{v}(a_1)&=ir(h),\nonumber\\
v(a_2)-v(a_1)&=-(\tilde{v}(a_2)-\tilde{v}(a_1)),\nonumber\\
1=v(b_2)-v(b_1)&=ir(h)+v(a_2)-v(a_1)+ir(g),\label{formula:potential}\\
i&\le \frac{1}{r(h)+r(g)}\label{formula:currentbound},
\end{align}
where (\ref{formula:currentbound}) comes from (\ref{formula:potential}) by $v(a_2)-v(a_1)\ge 0$, since $v(b_2)-v(b_1)=1$ and the only edges outside $L$ are $(a_1, b_1)$ and $(a_2, b_2)$, implying that a positive amount of current flows from $a_2$ to $a_1$. Then
\begin{gather*}
v(w_m)=v(w_m)-v(a_1)+ir(g)\text{ and }\\
\tilde{v}(w_m)=\left(\tilde{v}(w_m)-\tilde{v}(a_2)\right)+\left(\tilde{v}(a_2)-\tilde{v}(b_2)\right)=v(a_2)-v(w_m)+ir(g),
\end{gather*}
therefore their sum is
\begin{align*}
v(w_m)+\tilde{v}(w_m)&=v(a_2)-v(a_1)+2ir(g)\stackrel{(\ref{formula:potential})}{=}1-i(r(h)-r(g))\\
&\stackrel{(\ref{formula:currentbound})}{\ge} 1-\frac{r(h)-r(g)}{r(h)+r(g)}=\frac{2r(g)}{r(h)+r(g)}\ge \frac{2\eps_0}{1+\eps_0},
\end{align*}
using that $r(h)\ge r(g)\ge  \eps_0 r(h)$. Hence
\begin{gather}
    1-\P_{s, L\cup\{h_{m+1}, g_{m+1}\}}\left(A_{m}\right)=v(w_m)\ge \frac{\eps_0}{1+\eps_0}\ \text{ or }\label{formula:prob_differing}\\
    1-\P_{s, L\cup\{h_{m+1}, g_{m+1}\}}^{h_{m+1}\leftrightarrow g_{m+1}}\left(A_{m}\right)=\tilde{v}(w_m)\ge \frac{\eps_0}{1+\eps_0} \label{formula:prob_differing_2}.
\end{gather}

Running the random walk on $K_n$ instead of $L\cup\{h_{m+1}, g_{m+1}\}$, the probability of $A_m$ can only decrease, resulting in the inequalities of the proposition.

If $b_1=b_2$, then we can add a new vertex $b'$ with $c(a_2, b'):=c(a_2, b_2)$ and repeat the argumentation on $L\cup (a_1, b_1)\cup (a_2, b')$ instead of $L'$ to obtain (\ref{formula:prob_differing}) and (\ref{formula:prob_differing_2}).
\end{proof}

\subsection{Agreement of the models and the conductances on the fundamental cycles}
\label{subsec:determ_wst_vs_mst}
If there exists a spanning tree $T\ne \maxst(\c)$ with $\prod_{e\in T}c(e)\ge \frac{1}{2}\prod_{e\in \maxst(\c)}c(e)$, then by estimating the partition function $Z(\c)$, it is clear that $\P(\wst(\c)=\maxst(\c))\le \frac{1}{1+1/2}=\frac{2}{3}$. 

How can one find a tree satisfying this condition? If there is an edge $(u_1, u_2)=e\notin \maxst(\c)$ such that the unique path in $\maxst(\c)$ connecting $u_1$ and $u_2$ contains an edge $f$ with $\frac{1}{2}c(f)\le c(e)$, then $T=\maxst\setminus\{f\}\cup \{e\}$ is such a tree.

Interestingly, in Subsection \ref{subsec:wst_vs_mst}, it turns out that this naive reasoning and Lemma \ref{lemma:middlecond} together on our random electric network $\c:=\left\{\exp(-\beta_nU_e)\right\}_{e\in E(K_n)}$ are strong enough to prove that $\P\left(\wst(\c)=\maxst(\c)\right)$ tends to 0 for $\beta_n\ll n^2$ and tends to 1 for $\beta_n\gg n^2\log n$.

\begin{defi}
    Consider a spanning tree $T$ of $K_n$ and an $e=(u, v)\notin T$. 
    
    We denote by $\path_{T}(u, v)$ the unique path connecting $u$ and $v$ in $T$.
    
    \emph{The fundamental cycle} of the edge $(u, v)$ is $\path_{T}(u, v)\cup (u, v)$.
\end{defi} 

\begin{defi}
    We call an edge $e=(u,v)\notin \maxst(c)$ \emph{$\eps$-significant} if $\exists f\in \path_{\maxst(c)}(u, v)$ satisfying $\frac{c(e)}{c(f)}
    \ge \eps$. The \emph{set of $\eps$-significant external edges} is denoted by $E_{\mathrm{sig}}(\eps):=\{e\notin \maxst(c):\; e\text{ is $\eps$-significant}\}$.
\end{defi}

\begin{prop}\label{prop:wst_vs_mst}
Consider an electric network $(K_n, \c)$ satisfying Assumption \ref{assum} and $\eps_0\in (0, 1)$.
    \begin{itemize}
        \item[a)] If $\left|E_{\mathrm{sig}}\left(\eps_0\right)\right|>0$, then
        $$\P(\wst(\c)=\maxst(\c))\le \frac{1}{1+ \left|E_{\mathrm{sig}}\left(\eps_0\right)\right|\eps_0}.$$
        \item[b)] We have that
        $$\E[|E(\wst(\c))\setminus E(\maxst(\c))|]\le \left|E_{\mathrm{sig}}\left(\eps_0\right)\right|+\eps_0^{1/2}n^{11/2}.$$
        Consequently, if $\left|E_{\mathrm{sig}}\left(\eps_0\right)\right|=0$, then $$\P(\wst(\c)=\maxst(\c))\ge 1-\eps_0^{1/2}n^{11/2}.$$
    \end{itemize}
\end{prop}
\begin{proof}
    a) This follows from $\P(\wst(\c)=\maxst(\c))=Z(\c)^{-1} \prod_{f\in \maxst(\c)}c(f)$ since \begin{gather*}
        Z(\c)=\sum_{T\text{ is a spanning tree}}\prod_{f\in T}c(f)\ge \sum_{\substack{T \text{ is a spanning tree}\\ |E(T)\setminus E(\maxst(\c))|\le 1}}\prod_{f\in T}c(f)\ge (1+\left|E_{\mathrm{sig}}\left(\eps_0\right)\right|\eps_0)\prod_{f\in \maxst(c)}c(f).
    \end{gather*}

    b) For an external edge $e_{\mathrm{ext}}=(u_1, u_2)\notin E(\maxst(\c))\cup E_{\mathrm{sig}}(\eps_0))$, we choose $k$ such that $v_k$ is the first vertex of $\path_{\maxst(\c)}(u_1,u_2)$
     explored by the Aldous-Broder algorithm. Let $a:=\min_{f\in\path(u_1, u_2)}c(f)$, $E_{\bcond}:=\{e\in E(K_n)\given c(e)\ge a\}$ and $E_{\scond}:=\{e\in E(K_n)\given c(e)\le \eps_0 a\}$. 
     
     Then $e_{\mathrm{ext}}\in E_{\scond}$ as it is not significant, so by Lemma \ref{lemma:middlecond} for $t:=v_k$, all the vertices of $\path_{\maxst(\c)}(u_1,u_2)$ are explored before crossing $e_{\mathrm{ext}}$ with probability $\ge 1-\eps_0^{1/2}n^{7/2}$. Then,
\begin{align*}
    \E\left[|E(\wst(\c))\setminus E(\maxst(\c))|\right]&\le \left|E_{\mathrm{sig}}\left(\eps_0\right)\right|+\sum_{e_{\mathrm{ext}}\in E(K_n)\setminus (E(\maxst(\c))\cup E_{\mathrm{sig}}(\eps_0))}\P(e_{\mathrm{ext}}\in \wst(\c))\\
&\le \left|E_{\mathrm{sig}}\left(\eps_0\right)\right|+\binom{n}{2}\eps_0^{1/2}n^{7/2} \le \left|E_{\mathrm{sig}}\left(\eps_0\right)\right|+\eps_0^{1/2}n^{11/2}.
\end{align*}
By $\P(\wst(\c)\ne \maxst(\c))\le \E\left[|E(\wst(\c))\setminus E(\maxst(\c))|\right]$, The second inequality follows.
\end{proof}

\section{\texorpdfstring{$\wst^{\beta_n}_n$}{} on random electric networks for large \texorpdfstring{$\beta_n$}{}'s and the Erdős-Rényi graph}
\label{sec:wst_vs_mst}

In this section, we study how the results of Section \ref{sec:det} can be applied to the random electric network $\c=\{\exp(-\beta_nU_e)\}_{e\in E(K_n)}$.

We determine in terms of the $U_e$'s the possible edges that the Aldous-Broder algorithm can use as it builds $\wst_{n}^{\beta_n}$. As $\beta_n$ grows more and more slowly, one can give a characterization of this larger and larger edge set. All the descriptions listed below are meant to hold with probability $\rightarrow 1$.
 \begin{itemize}
     \item[$\circ$] If $\beta_n\gg n^{3}\log n$, then in each building step, the Aldous-Broder algorithm can use only the current edge of Prim's invasion algorithm. We also show that this is not true for $\beta_n\ll n^{3}$.
    \item[$\circ$] If $\beta_n\gg n^{2}\log n$, then the Aldous-Broder algorithm can choose only from the edges of $\mst_n$ whenever it builds an edge. We also prove that this does not hold for $\beta_n\ll n^{2}$.
    \item[$\circ$] If $\beta_n\gg n^{4/3}\log n$, then denoting by $H_{n, p}$ the largest connected component of the Erdős-Rényi graph $G_{n, p}$, one can give some $p_0\le p_{1/2} \le p_1\le \ldots \le p_{t_n}$ such that $H_{n, p_0}\subseteq H_{n, p_{1/2}} \subseteq H_{n, p_1} \subseteq \ldots \subseteq H_{n, p_{t_n}}$, the Aldous-Broder algorithm uses edges only from $H_{n, p_{i+1/2}}$ until $V(H_{p_{i-1}})$ is not completely explored, and $H_{n, p_i}[ V(K_n)\setminus V(H_{n, p_{i-1}})]$ contains only short paths for any $i$. As detailed later, this is the key tool to determine the growth of the diameter of the $\wst^{\beta_n}_n$. The reasoning is based on \cite{addario2009critical}.
    \item[$\circ$] If $\beta_n\gg n\log n$, then the label of any edge built by the Aldous-Broder algorithm can be at most $\frac{\kappa \log n}{\beta_n}$ larger than the minimal option in each building step with probability $\rightarrow 1$ really fast. Hence, the sum of $U_e$'s in the $n-1$ steps differ at most    $\sum_{e\in \wst_n^{\beta_n}}U_e-\sum_{e\in \mst_n}U_e=(n-1)O(\frac{\log n}{\beta_n})\rightarrow 0$.
 \end{itemize}

\subsection{Heuristics behind the different phases for large \texorpdfstring{$\beta_n$}{}'s}\label{subsec:ideas}
We collect the ideas of this section here, giving some intuitive arguments on why the phase transitions around $\beta_n\approx n^3$ and $\beta_n\approx n^2$ happen and how the growth of the diameter for $\beta_n\gg n^{4/3}\log n$ and the expected total length for $\beta_n\gg n\log n$ are obtained. 
	\paragraph{Our random conductances and uniform labels}
	
	For our random $c(e)=\exp(-\beta_nU_e)$ and any $g, e \in E(K_n)$, the ratio of the conductances can be written as  
	$\frac{c(e)}{c(g)}=\exp(-\beta_n(U_e-U_g))$, so 
	\begin{itemize}
		\item[$\circ$] $\frac{1}{2}<\frac{c(e)}{c(g)}<1 \Longleftrightarrow \frac{\log 2}{\beta_n}>U_e-U_g>0$, and
		\item[$\circ$] $\frac{c(e)}{c(g)}<n^{-\kappa}\Longleftrightarrow U_e-U_g>\frac{\kappa \log n}{\beta_n}$ for any $\kappa>0$.
	\end{itemize}
	We call $\frac{\log 2}{\beta_n}>U_e-U_g>0$ the case of \emph{comparable} conductances and $\frac{c(e)}{c(g)}<n^{-\kappa}$ the case of \emph{distinguishable} conductances, for an appropriately chosen $\kappa>0$.

The first idea is that one can use Proposition \ref{prop:ab_is_prim} if $|U_e-U_g|$ is big for any $e\ne g$. Therefore it is reasonable to study the ordered sample of the labels.

 \begin{notation*}We denote by $U_1^{*, M}\le \ldots \le U_{M}^{*, M}$ the ordered sample of $M$ i.i.d.~$\unif[0,1]$ random variables.
\end{notation*}

\begin{lemma}\label{lemma:size_of_gaps}
Consider some integers $0<i<M$ and $0<\delta<1$ satisfying $ M\delta\le 0.35$. Then $$\frac{M\delta}{2}\le  \P\left(U_{i+1}^{*, M}-U_{i}^{*, M}<\delta\right)\le 2M\delta.$$
\end{lemma}

\begin{proof}
It is known that the distribution of gaps of the ordered sample are the same,
$U_{i+1}^{*, M}-U_{i}^{*, M}\sim U_{1}^{*, M}\;\forall i$, and $\P\left(U_{1}^{*, M}>\delta \right)=\left(1-\delta\right)^{M}.$ 

Since $1-x\le e^{-x}\le 1-\frac{x}{2}$ for $0<x<0.7$, for $2M\delta\le 0.7$ we have that
\begin{gather*}
1-2M\delta\le e^{-2M\delta}\le \left(1-\delta\right)^{M}=\P\left(U_{i+1}^{*, M}-U_{i}^{*, M}>\delta\right)\le e^{-M\delta}\le 1-\frac{M\delta}{2}.
\end{gather*}
\end{proof}
 
\paragraph{Heuristics behind Theorem \ref{thm:ab_vs_inv} a)} For $\beta_n\gg n^{3}\log n$, we check that Proposition \ref{prop:ab_is_prim} with any fixed $\kappa>0$ and $\eps_0:=n^{-\kappa}$ can be applied to our random electric network in every building step with probability $\rightarrow 1$. 

\begin{proof}[Proof of Theorem \ref{thm:ab_vs_inv}. a)] As in Notation \ref{notation:alg_v-e}, $g_j$ and $h_j$ are the edges from $\partial W_j$ with the largest and second largest conductances. We write $U_i^*:=U_i^{*,\binom{n}{2}}$, and introduce $\{i(j)\}_{j\le n-1}$ to be the indices of $g_j$ in the ordered sample, i.e.,~$U_{g_j}=U_{i(j)}^*$. 

Since $U_{g_j}=U_{i(j)}^*$ and
$U_{h_j}\ge U_{i(j)+1}^*$, we have that $U_{h_j}-U_{g_j}\ge U_{i(j)+1}^*- U_{i(j)}^*$, and the distribution of the latter can be easily handled. Indeed, Prim's invasion algorithm depends only on the ordering of the conductances and not on the exact values of them, so the distribution $(U_{i(j)}^*, U_{i(j)+1}^*)$ given $i(j)=i$ is simply $(U_{i}^*, U_{i+1}^*)$ if $\P(i(j)=i)>0$.

Therefore, for $\{\exists\textrm{ a small gap}\}:=\{\exists j\le n-1: \;U_{h_j}-U_{g_j}< \frac{\kappa\log n}{\beta_n}\}$,
\begin{equation}\label{formula:number_small_gaps}
\begin{aligned}
\P(\exists\text{ a small gap})&\le \P\left(\exists j\le n-1 :\;U_{i(j)+1}^{*}-U_{i(j)}^{*}<\frac{\kappa \log n}{\beta_n}\right)\\
&\stackrel{\cup\text{-bound}}{\le} \sum_{j=1}^{n-1}\P\left(U_{i(j)+1}^{*}-U_{i(j)}^{*}<\frac{\kappa \log n}{\beta_n}\right)\\
&\stackrel{\text{Lemma \ref{lemma:size_of_gaps}}}{\le} (n-1)2\binom{n}{2}\frac{\kappa \log n}{\beta_n}= \Theta\left(n^3\frac{\log n}{\beta_n}\right)\rightarrow 0
\end{aligned}
\end{equation}
as $n\rightarrow\infty$ for $\beta_n \gg n^3\log n$.

We obtained that the conditions of Proposition \ref{prop:ab_is_prim} are satisfied for $\eps_0:=n^{-11}$ with probability $\rightarrow 1$. This choice of $\eps_0$ guarantees that $\eps_0^{1/2}n^{9/2}=n^{-1}\rightarrow 0$ as $n\rightarrow \infty$ and then the probability of the existence of a step when the two algorithms build different edges is at most $\Theta\left(\frac{n^3\log n}{\beta_n}\right)+\frac{1}{n}\rightarrow 0$.
\end{proof}

\paragraph{Heuristics behind Theorem \ref{thm:ab_vs_inv} b)} For $\beta_n\ll n^3$, we are showing that, with probability $\rightarrow 1$, there exists a step when the two algorithms build different edges. 

Using Notation \ref{notation:alg_v-e}, if we are able to check that there are $k_n\rightarrow\infty$ steps $j$ such that $g_j$ and $h_j$ are comparable, then Proposition \ref{prop:ab_is_not_prim} implies that the probability of `the two algorithms always build the same edges’ is at most
$c^{k_n}\rightarrow 0$ for a $c<1$.

In the proof of Proposition \ref{prop:number_of_small_gaps}, we show for some fixed $a_1, a_2\in (0, 1)$ the existence of $\Theta(n)$ building steps $ j\in [a_1n,a_2n]$  such that the distribution of the labels of $\partial W_j$, conditioned on what has happened until that point, are i.i.d.~$\unif[p_j, 1]\approx\unif[0, 1]$ for a common (random) small $p_j$. The edge boundary $|\partial W_j|$ is at least $a_1a_2n^2$ for $j\in [a_1n,a_2n]$, so
\begin{equation}\label{formula:number_small_gaps_second}
\begin{aligned}
\E\left[\text{\# of $j$'s: $h_j, g_j$ are comparable}\right]&\ge 
\sum_{j=[a_1n]}^{[a_2n]}\P\left(U_{h_j}-U_{g_j}<\frac{\log 2}{\beta_n}\right)\\& \stackrel{\text{Lemma \ref{lemma:size_of_gaps}}}{\ge}\Theta(n)\Theta\left(a_1a_2n^2\frac{\log 2}{\beta_n}\right)= \Theta\left(\frac{n^3}{\beta_n}\right)\rightarrow \infty
\end{aligned}
\end{equation}
for $\beta_n\ll n^3$ as $n\rightarrow\infty$. Note that this computation also shows that in (\ref{formula:number_small_gaps}) studying $U_{i(j)+1}^*-U_{i(j)}^*$ for $n-1$ steps instead of $U_{h_j}-U_{g_j}$ for $\Theta(n)$ steps modifies our result only by a constant factor, and the extra $\log n$ factor in  (\ref{formula:number_small_gaps}) compared to (\ref{formula:number_small_gaps_second}) comes in since we study different sizes of gaps of labels.

We will also see that the $\{U_{h_j}-U_{g_j}\}_j$'s behave like independent random variables in these $\Theta(n)$ steps, hence we can indeed use Proposition \ref{prop:ab_is_not_prim} $k_n\rightarrow\infty$ times and get that the algorithms differ with probability $\rightarrow 1$. Because of this independence for $\Theta(n)$ steps, it is not surprising that the union bound captures the reality in (\ref{formula:number_small_gaps}). 

Notice that (\ref{formula:number_small_gaps}) and (\ref{formula:number_small_gaps_second}) are strong enough to show us that the phase transition regarding the algorithms happens around $\beta_n\approx |V(K_n)||E(K_n)|$ up to a $\log n$ factor.

How do we find these $\Theta(n)$ steps? If $W_j$ is the vertex set of a connected component of the Erdős-Rényi graph $G_{n, p}$ with edges $E(G_{n, p})=\{e\given U_e\le p\}$ for a $p$, then the labels of $\partial W_j$ can indeed be thought as i.d.d.~$\unif[p, 1]$ random variables. This situation happens for example right after an isolated vertex of $G_{n, q}$ is joined to the giant component of $G_{n, q}$ by the invasion algorithm, which occurs $\Theta(n)$ times.
	
\paragraph{Heuristics behind Theorem \ref{thm:wst_vs_mst} a)} 
For $n^2\log n\ll \beta_n\ll n^3$, it turns out that  the Aldous-Broder and Prim's invasion algorithm add the same edges, but in different orderings. To prove this, we check for any fixed $\kappa$, that there do not exist any $n^{-\kappa}$-significant external edges of $\mst_n$ with probability $\rightarrow 1$. Then Proposition \ref{prop:wst_vs_mst} b) implies that $\P(\wst_n^{\beta_n}=\mst_n)\rightarrow 1$.

The significance of an edge for our random conductances $\c=\{\exp(-\beta_n U_e)\}_{e\in E(K_n)}$ can be phrased with the $U_e$'s: an $e=(u,v)\notin \mst_n$ is $n^{-\kappa}$-significant if there exists an $f\in \path_{\mst_n}(u, v)$ satisfying $U_f-U_e\le \frac{\kappa \log n}{\beta_n}$. The number of significant edges becomes easily understood under appropriate conditioning.
\begin{lemma}\label{lemma:conditioning}
Under conditioning on $\{(f, U_f)\}_{f\in E(\mst_n)}$, we have that $\{U_{e}\}_{e\notin E(\mst_n)}$ are independent $\unif[m_e, 1]$ random variables, where $m_e:=\max_{f\in \path_{\mst_n}(u, v)}U_f$ for $e=(u, v)$.
\end{lemma}
\begin{proof}
Under this conditioning, $U_e>m_e$ must hold for any $e\notin \mst_n$. On the other hand, if $U_e>m_e$ is satisfied for any $e\notin \mst_n$, then Prim's invasion algorithm gives us the edges of $\mst_n$. 

Therefore, the joint distribution of $\{U_e\}_{e\notin E(\mst_n)}$ under conditioning on $\{(f, U_f)\}_{f\in E(\mst_n)}$ can be obtained simply by conditioning the $\unif\left(\prod_{e\notin E(\mst)} [0, 1]\right)$ random variable to
 the rectangle
$\prod_{e\notin E(\mst)}[m_e, 1]$, giving us independent $\unif[m_e, 1]$ random variables for the marginals.
\end{proof}

We will see that $m_e\ll 1$ for $\forall e\notin \mst_n$, so
considering the random variables $\chi_e\sim \unif[m_e, 1]$, we have that $\P(e\text{ is $n^{-\kappa}$-significant})=\P\left(\chi_e-m_e<\frac{\kappa \log n}{\beta_n}\right)\sim \frac{\kappa \log n}{\beta_n}$. By Proposition \ref{prop:wst_vs_mst} b), for $\kappa>11$, we have that $\P(\wst^{\beta_n}_n\ne  \mst_n)\le \P(\exists n^{-\kappa}\text{-significant edges})+o(n^{-11/2}n^{11/2})$. Then, by the union bound,
\begin{equation*}
    \begin{aligned}
\P\left(\wst_n^{\beta_n}\ne \mst_n\right)&=o(1)+\P(\exists n^{-\kappa}\text{-significant edges})\le
\sum_{e\notin \mst_n}\P(e\text{ is $n^{-\kappa}$-significant})+o(1)\\
&\le \binom{n}{2}\Theta\left(\frac{\kappa\log n}{\beta_n}\right)+o(1)=\Theta\left(\frac{n^2\log n}{\beta_n}\right)+o(1)\rightarrow 0
\end{aligned}
\end{equation*}
for $\beta_n\gg n^2\log n$ as $n\rightarrow \infty$.

This computation also shows for $\beta_n\gg n\log n$, that
\begin{equation}\label{formula:expected_diff}
    \begin{aligned}
    \E\left[\left|E(\wst^{\beta_n}_n)\setminus E(\mst_n)\right|\right]&\le \E\left[ \left|\left\{n^{-\kappa}\text{-significant edges}\right\}\right|\right]+o(1)\\
    &=\Theta\left(\frac{n^2\log n}{\beta_n}\right)+o(1)=o(n),    
\end{aligned}
\end{equation}
and that for $X\sim \unif(V(K_n))$, the edge sets $\{(X, u): \;u\in V(K_n)\}$ and $E(\wst^{\beta_n}_n)\setminus E(\mst_n)$ are disjoint with probability $\rightarrow 1$, which will be useful when we compare the expected total lengths of $\wst^{\beta_n}_n$ and $\mst_n$, as we will see later.

\paragraph{Heuristics behind Theorem \ref{thm:wst_vs_mst} b)} If $\beta_n\ll n^2$, then of course 
\begin{gather}\label{formula:prob_wstnotmst}
\E\left[\left|\left\{\text{$\frac{1}{2}$-significant edges}\right\}\right|\right]=\sum_{e\notin \mst_n}\P\left(e\text{ is $\frac{1}{2}$-significant}\right)=\Theta\left(n^2\right)\frac{\log 2}{\beta_n}=\Theta\left(\frac{n^2}{\beta_n}\right)\rightarrow \infty.
\end{gather}

Under conditioning on $\{(f, U_f)\}_{f\in E(\mst_n)}$, the $\frac{1}{2}$-significance of the edges $e\notin \mst_n$ are independent of each other by Lemma \ref{lemma:conditioning}, so $\P\left(\text{\# of $\frac{1}{2}$-significant edges}\ge k_n\right)\rightarrow 1$ is easy for some $k_n\rightarrow \infty$.

These computations and Proposition \ref{prop:wst_vs_mst} imply that the phase transition regarding the agreement of the models happens around $\beta_n\approx |E(K_n)|$ up to a logarithmic factor.

\paragraph{Heuristics behind Theorem \ref{thm:length} a) and b)} Typically $\max_{e\in \mst_n}U_e\le \frac{2\log n}{n}$, therefore for $\beta_n\gg n\log^2 n$, any $n^{-\kappa}$-significant edge has label $\le \frac{2\log n}{n}+\frac{\kappa\log n}{\beta_n}=\Theta\left(\frac{\log n}{n}\right)$ and
\begin{align*}
\E\left[L(\wst^{\beta_n}_n)\right]\stackrel{(\ref{formula:expected_diff})}{\le} \E\left[L(\mst^{\beta_n}_n)\right]+\Theta\left(\frac{\log n}{n}\right)\Theta\left(\frac{n^2\log n}{\beta_n}\right)=\zeta(3)+o(1).
\end{align*}
In Subsection \ref{subsec:mst_length}, we give an other proof which works even for $\beta_n\gg n\log n$.

\paragraph{Ideas behind Theorem \ref{thm:has_diam_mst}} For $n^{4/3}\log n\ll \beta_n\ll n^2$, it turns out that although we have $\P(\wst_n^{\beta_n}\ne \mst_n)\rightarrow 1$, their global geometries behave similarly: in \cite{addario2009critical} it is shown that $\E[\diam(\mst_n)]=\Theta(n^{1/3})$ and we are proving that $\diam(\wst_n^{\beta_n})$ is typically $\Theta(n^{1/3})$, as well.

The connection between $\mst_n$ and the Erdős-Rényi graphs is that the connected components of $G_{n,p}$ and of $\mst_n\cap G_{n,p}$ are the same for any $p\in [0, 1]$. Indeed, for any $u, v\in V(K_n)$ if there is a path between $u$ and $v$ using only edges $e$ with $U_e\le p$, then Prim's invasion algorithm connects $u$ and $v$ with a path containing edges $e$ with $U_e\le p$. 

Denote by $H_{n,p}$ the largest connected component of $G_{n,p}$, and by $\lp$ the length of the longest path. In order to simplify the formulation, now we cheat a little bit and suppose that $H_{n, p}\subseteq H_{n, p'}$ for any $\frac{1}{n}\le p\le p'$,
which in reality holds only for $p-\frac{1}{n}\gg \frac{1}{n^{4/3}}$ by \cite{luczak1990component}. In Subsection \ref{subsec:wst_diam_abr}, we will deal with this technical detail. 

To study the growth of the diameter of $\mst_n$ in \cite{addario2009critical}, $p_i=\frac{1}{n}+\left(\frac{5}{4}\right)^i\frac{1}{n^{4/3}}$ was chosen interpolating between the critical $G_{n,p_0}$ and the supercritical $G_{n,\frac{\log n}{n}}$ Erdős-Rényi graphs. They have shown the following:
\begin{itemize}
\item[a)]  Inside the critical $H_{n, p_0}$: $\diam(\mst_n\cap H_{n, p_0})\le \lp(H_{n, p_{0}})=O(n^{1/3})$ with probability $\rightarrow 1$.
\item[b)] From critical $H_{n, p_0}$ to supercritical $H_{n, \log n /n }$: they controlled the difference of the diameters $|\diam(\mst_n\cap H_{n, p_t})-\diam(\mst_n\cap H_{n, p_0})|$ by $\sum_{i}\lp (H_{n, p_{i+1}}\setminus H_{n, p_i})$. They proved that this latter sum is $O(n^{1/3})$ with probability $\rightarrow 1$.
\item[c)] From supercritical $H_{n, \log n /n }$ to the whole $K_n$: $\E[\diam(\mst_n)-\diam(\mst_n\cap H_{n, \log n /n })]=o(n^{1/3})$, since the giant $H_{n, \log n /n }$ is reached from any vertex in $o(n^{1/3})$ steps by Prim's invasion algorithm.
\end{itemize}

Our main result is a modification of this reasoning which works for $\wst_n^{\beta_n}$ if $\beta_n\gg n^{4/3}\log n$. We have to be careful since the connection between Erdős-Rényi graphs and $\wst_{n}^{\beta_n}$ is more difficult than the one for $\mst_n$: for example, $\wst_{n}^{\beta_n}\cap H_{n, p}$ is not necessary connected. What is worse, even $\wst_{n}^{\beta_n}[V(H_{n, p})]$ can be disconnected.  However, Lemma \ref{lemma:H_comps} below
is still strong enough to make the analogue of a) and b) for $\wst_{n}^{\beta_n}$ work. Part c) with the building steps of the Aldous-Broder algorithm behaves quite similarly to the case of Prim's invasion algorithm.

It will be clear that the property we use is that $E(K_n)\setminus E(G_{n, p_{i+1/2}})$ and $E(G_{n, p_i})$ are distinguishable for any $i$, which is equivalent to $p_{i+1/2}-p_{i}=\frac{1}{2}\left(\frac{5}{4}\right)^{i}\frac{1}{n^{4/3}}\gg \frac{\kappa \log n}{\beta_n} \ \forall i$, so we need that $\beta_n\gg n^{4/3}\log n$.

\subsection{Comparing the Aldous-Broder and Prim's invasion algorithms}\label{subsec:ab_vs_invasion}

In Subsection \ref{subsec:ideas}, we have shown Theorem \ref{thm:ab_vs_inv} a) about the two algorithms always having the same building steps for $\beta_n\gg n^3\log n$. In order to prove that they differ for $\beta_n\ll n^3$, we find some steps of Prim's invasion algorithm for which the joint distribution of $\{U_e\}_{e\in \partial W_i}$ is understood.

\begin{prop}\label{prop:number_of_small_gaps}
Condsider $\beta_n\ll n^3$, $c(e):=\exp(-\beta_n U_e)$ and $s\in V(K_n)$. 

Then $\exists C>0$ such that for Prim's invasion algorithm started at $s$
\begin{gather}\label{formula:cointing_smallgaps}
    \P\left( \Lvert \left\{1\le i\le n-1:\;\frac{1}{2}<\frac{c(h_i)}{c(g_i)}<1\text{ and }g_i\in \partial V(H_{n, 2/n})\right\} \Lvert \ge C\min\left(\frac{n^3}{\beta_n}, n\right) \right)\rightarrow 1.
\end{gather}
\end{prop}

We prove Theorem \ref{thm:ab_vs_inv} b) at the end of this subsection using Proposition \ref{prop:number_of_small_gaps}.

The proof of Proposition \ref{prop:number_of_small_gaps}  is based on two observations and can be found after Lemma \ref{lemma:number_of_small_gaps_2ndmom}. To simplify this description, we suppose now that $s\in H_{n, 2/n}$. The giant component and the isolated vertices of $G_{n, p}$ are denoted by $H_{n, p}$ and $I_{n, p}$.
\begin{itemize}
    \item[$\circ$] Firstly, in Lemma \ref{lemma:number_of_small_gaps_coupling}, we will see that if $g_i\in \partial V(H_{n, 2/n})$, then, even conditioned on what
has happened until that point, we can handle the distribution of the labels on $\partial W_{i-1}$. If additionally $|\partial W_{i-1}|>Dn^2$ holds for a $D>0$, then we can estimate the conditional probability of $\left\{\frac{c(h_i)}{c(g_i)}>\frac{1}{2}\right\}$ easily, and for such $i$'s, the $\left\{\frac{c(h_i)}{c(g_i)}\right\}$'s are bounded by i.i.d.~random variables.
    \item[$\circ$] Secondly, if the smallest labeled edge $e$ coming from a $v\in I_{n, 2/n}$ goes directly to $H_{n, 2/n}$, then there exists an $i$ such that $g_i=e\in \partial V(H_{n, 2/n})$. Using this, we will show in Lemma \ref{lemma:number_of_small_gaps_2ndmom} that, with probability $\rightarrow 1$, there are $\Theta(n)$ $i$'s such that $g_i\in \partial V(H_{n, 2/n})$ and $|\partial W_{i-1}|>Dn^2$.
\end{itemize} 

\begin{notation}
\label{notation:subsec_ab_inv}
In this subsection, for any permutation $\pi$ of $E(K_n)$, we can consider the event $E(\pi):=\{\text{the ordering of $E(K_n)$ is }\pi\}$, where ordering means the increasing ordering w.r.t.~$\{U_e\}_{e}$.

We choose $t(0)$ to satisfy $t(0):=\min\{j:\;V(H_{n, 2/n})\subseteq W_{j}\}$ and $t(i):= \min\{j>t(i-1):\;g_{j}\in \partial V(H_{n, 2/n})\}$ for any $i\ge 1$, recursively. Observe that if Prim's invasion algorithm is started from an $s\in V(H_{n, 2/n})$, then $W_{t(0)}=V(H_{n, 2/n})$ because of the greediness of the algorithm. 

Note that, conditioning on $E(\pi)$, $W_i$ and $t(i)$ are deterministic for any $i$.
\end{notation}

\begin{lemma}\label{lemma:number_of_small_gaps_coupling}
We fix some $n, j\in \N$ and $D\in (0, 1)$. We start Prim's invasion algorithm on our random electric network from an $s\in V(H_{n, 2/n})$. We use Notation~\ref{notation:subsec_ab_inv}.  

We also define $N:=N(n, \beta_n, j):=\left|\left\{i\le j: \frac{c(h_{t(i)})}{c(g_{t(i)})}>\frac{1}{2}\right\}\right|$ where $c(e)=\exp(-\beta_nU_e)$, and $N^{\xi}:=N^{\xi}(n, \beta_n, j, D):=\left|\left\{i\le j: \xi_i<\frac{\log 2}{\beta_n}\right\}\right|$ for the i.i.d.~$\xi_i\sim U_2^{*, [Dn^2]}-U_1^{*, [Dn^2]}$.

Consider a permutation $\pi$ such that $|\partial W_{t(i)-1}|\ge Dn^2$ is satisfied for any $i\le j$ on $E(\pi)$.

Then $N$ stochastically dominates $N^{\xi}$ on $E(\pi)$.\end{lemma}
\begin{proof}
Since $V(H_{n, 2/n})\subseteq W_{t(i-1)}$ and one of the endpoints of $g_{t(i)}$ is in $V(H_{n, 2/n})$, we have that $U_{g_{t(i-1)}}=\max_{\ell<t(i)}U_{g_{\ell}}$ and $W_{t(i)-1}=V\left(H_{n, U_{g_{t(i-1)}}}\right)$, meaning that right after the  $t(i)-1^{st}$ step of Prim's invasion algorithm, the random walk has explored exactly the giant component of an Erdős-Rényi graph for an appropriate $p$. 

Consequently, just right after the $t(i)-1^{st}$ step, the joint distribution of the labels of $\partial W_{t(i)-1}$ becomes simple: conditioned on $\left(W_{t(i)-1}, U_{g_{t(i-1)}}\right)$, the labels $\{U_e\}_{e\in \partial W_{t(i)-1}}$ are i.i.d.~$\unif[U_{g_{t(i-1)}}, 1]$ random variables, since $e\in \partial W_{t(i)-1}$ is not added before the $t(i)^{th}$ step if and only if $U_e> U_{g_{t(i-1)}}$.

For a permutation $\pi$ and $\mathbf{a}(i):=(a_1\ldots, a_i)\in \{0, 1\}^{i}$, we introduce the event $E(\pi, \mathbf{a}(i))$ as $E(\pi, \mathbf{a}(i)):=E(\pi)\cap \left\{\mathds{1}\left[\frac{c(h_{t(\ell)})}{c(g_{t(\ell)})}>\frac{1}{2}\right]=a_\ell \;\forall \ell\le i \right\}$. For any $\mathbf{a}(i-1)\in \{0, 1\}^{i-1}$ and $\pi$ satisfying the condition of the lemma,
\begin{align*}
    &\P\left(\frac{c(h_{t(i)})}{c(g_{t(i)})}>\frac{1}{2} \Given E(\pi, \mathbf{a}(i-1))\right)\stackrel{\text{by def.}}{=}\P\left(U_{h_{t(i)}}-U_{g_{t(i)}}<\frac{\log 2}{\beta_n}\Given E(\pi, \mathbf{a}(i-1))\right)\\
    &\ge \P\left(U_2^{*, |\partial W_{t(i)-1}|}-U_1^{*, |\partial W_{t(i)-1}|}<\frac{\log 2}{\beta_n}\Given E(\pi, \mathbf{a}(i-1))\right)\ge\P\left(\xi_i<\frac{\log 2}{\beta_n}\right),
\end{align*}
where the first inequality comes from that, under conditioning on $E(\pi)$, $|\partial W_{t(i)-1}|$ is deterministic, and $\left(1-U_{g_{t(i-1)}}\right)^{-1}\left(U_{h_{t(i)}}-U_{g_{t(i)}}\right)$ are distributed as independent $U_2^{*, |\partial W_{t(i)-1}|}-U_1^{*, |\partial W_{t(i)-1}|}$ random variables, since conditioning on the ordering does not modify the ordered sample. The second inequality holds since $|\partial W_{t(i)-1}|>Dn^2$ on $E(\pi)$.

We obtained that $\P\left(E(\pi, (a_1, \ldots, a_{i-1}, 1)) \Given E(\pi, (a_1, \ldots, a_{i-1}))\right)\ge \P\left(\xi_i<\frac{\log 2}{\beta_n}\right)$, implying the stochastic dominance on any $E(\pi)$ satisfying the condition of the lemma.
\end{proof}

\begin{lemma}\label{lemma:number_of_small_gaps_2ndmom} We start Prim's invasion algorithm from an $s\in V(H_{n, 2/n})$. We use Notation \ref{notation:subsec_ab_inv}.

Then $\exists C', D>0$ such that 
$\P\left(\left|\partial W_{t(i)-1}\right|\ge Dn^2 \; \forall i\le C'n\right)\rightarrow 1$.
\end{lemma}

\begin{proof}
We denote by $A_n$ the event that `$|V(H_{n, 2/n})|>d_1n$ and $|I_{n, 2/n}|>d_2n$’, where $d_1, d_2\in (0, 1)$ are chosen such that $\P(A_n)\rightarrow 1$ as $n\rightarrow \infty$. This can be done by Theorem 2.3.2 and Subsection 2.8  of \cite{durrett2010random}. 

If the smallest labeled edge $e$ coming from a $v\in I_{n, 2/n}$ goes directly to $H_{n, 2/n}$, then by the greediness of Prim's invasion algorithm, $v$ is joined to the invasion tree through $e$, i.e., $\exists i: \; g_{t(i)}=e\in \partial V(H_{n, 2/n})$. Hence it is reasonable to study
$$X_n:=\sum_{v\in I_{n, 2/n}}X_v\text{, where }X_v:=\mathds{1}[\text{the smallest labeled edge $e$ of $v$ runs to $H_{n, 2/n}$}].$$

We define $\P_{\omega}(\cdot):=\P(\cdot\given H_{n, 2/n}=\omega)$. Under this conditioning we estimate $\E_{\omega}(\cdot)$, $\var_{\omega}(\cdot)$ and $\cov_{\omega}(\cdot, \cdot)$ of some random variables. In this proof, writing $\Theta(\cdot)$, we mean it uniformly on $A_n$. For $\omega\in A_n$, we have $\E_{\omega}[X_n]=\Theta(n)$, since $|V(H_{n, 2/n})|\in [d_1n, n]=\Theta(n)$ and $|I_{n, 2/n}|\in [d_2n, n]=\Theta(n)$,
\begin{gather*}
\E_{\omega}X_v=\frac{|V(H_{n, 2/n})|}{n-1}=\Theta(1)\text{ and }\E_{\omega}[X_n]=\Theta(n).
\end{gather*}
Moreover, $\text{Var}_{\omega}(X_n)=\Theta(n)$ holds. Indeed, $\text{Var}_{\omega}(X_v)=\Theta(1) \; \forall v\in I_{n, 2/n}$, and for $u\ne v\in I_{n, 2/n}$, 
\begin{gather*}
    \cov_{\omega}(X_{u}, X_{v})\stackrel{(*)}{=}\frac{2|V(H_{n, 2/n})|}{2n-3}\frac{|V(H_{n, 2/n})|}{n-1}-\left(\frac{|V(H_{n, 2/n})|}{n-1}\right)^2=\frac{|V(H_{n, 2/n})|^2}{(n-1)^2(2n-3)}=\frac{E_{\omega}[X_{u}]^2}{2n-3},\\
    \text{Var}_{\omega}(X_n)=\sum_{v\in I_{n, 2/n}}\text{Var}_{\omega}(X_v)+\sum_{\substack{u\ne v\\\in I_{n. 2/n}}}\text{Cov}_{\omega}(X_u, X_v)=\Theta \left(n\right)\Theta(1)+\Theta \left(n^2\right)\Theta\left(\frac{1}{2n-3}\right)=\Theta(n),
\end{gather*}
where $(*)$ is computed the following way: denoting the smallest labeled edge of $E(\{u, v\}, V(K_n))$ by $(a_1, b_1)$ with $a_1\in \{u, v\}$, we have that $\P\left(b_1\in H_{n, 2/n}\right)=\frac{|E(\{u, v\}, V(H_{n, 2/n}))|}{|E(\{u, v\}, V(K_n))|}=\frac{2|V(H_{n, 2/n})|}{2n-3}$. Given $\{b_1\in H_{n, 2/n}\}$, the probability of the smallest labeled edge of $E(\{u, v\}\setminus\{a_1\}, V(K_n))$ having endpoint in $H_{n, 2/n}$ is $\frac{|V(H_{n, 2/n})|}{n-1}$, so $\E[X_{u}X_{v}]=\frac{2|V(H_{n, 2/n})|}{2n-3}\frac{|V(H_{n, 2/n})|}{n-1}$.

By $\E_{\omega}[X_n]=\Theta(n)$, $\exists C'>0$ such that $\frac{1}{2}\E_{\omega}[X_n]\ge 2C'n \ \forall \omega\in A_n$. Hence by Chebyshev's inequality
\begin{align}\label{formula:steps_cheby}
\P_{\omega}\left(X_n\le 2C'n\right)\le \P_{\omega}\left(|X_n-\E_{\omega}[X_n]|\ge \frac{1}{2}\E_{\omega}[X_n]\right)\le  \frac{\var_\omega(X_n)}{\frac{1}{4}\E_{\omega}[X_n]^2}\rightarrow 0
\end{align}
uniformly on $A_n$. Since $\P(A_n)\rightarrow 1$, we also have that $\P\left(X_n>2C'n\right)\rightarrow 1.$

Hence, for any $i\le [C'n]$, we have that $|I_{n, U_{g_{t(i-1)}}}|\ge 2C'n-i\ge C'n$ and $|\partial W_{t(i)-1}|\ge |W_{t(i)-1}||I_{n, U_{g_{t(i-1)}}}|\ge d_1nC'n$ hold, with probability $\rightarrow 1$. So we get the statement of the lemma for $D:=d_1C'$.

Then, with the notation of the heuristics appearing in Subsection \ref{subsec:ideas}, choosing $a_1:=d_1$ and $a_2:=1-C'$, we indeed obtained $C'n$ steps $j\in [a_1n, a_2n]$ where the distribution of the labels of $\partial W_j$ can be estimated easily and independently for different $j$'s.
\end{proof}

\begin{proof}[Proof of Proposition \ref{prop:number_of_small_gaps}] We use Notation \ref{notation:subsec_ab_inv}. Firstly, we suppose that $s\in V(H_{n, 2/n})$. Because of the previous lemma, there exist $C', D>0$ such that for $E_n:=\{\left|\partial W_{t(i)-1}\right|\ge Dn^2 \; \forall i\le C'n\}$, $\P(E_n)\rightarrow 1$ holds.

We study the i.i.d.~random variables $\xi_j\sim U_2^{*, [Dn^2]}-U_1^{*, [Dn^2]}$ for $j\le C'n$. If $\beta_n\gg n^2$, then $\frac{n^2}{\beta_n}=o(1)$, so Lemma \ref{lemma:size_of_gaps} can be applied for large enough $n$'s, giving us
\begin{equation}
\begin{aligned}\label{formula:prob_one_big_gap}
    \P&\left(\xi_i<\frac{\log 2}{\beta_n}\right)\stackrel{\text{Lemma \ref{lemma:size_of_gaps}}}{=}&
   \Theta\left([Dn^2]\frac{\log 2}{\beta_n}\right)=\Theta\left(\frac{n^2}{\beta_n}\right).
  \end{aligned}
\end{equation}

Since $\E\left[\sum_{j\le C'n}\mathds{1}\left[\xi_j<\frac{\log 2}{\beta_n}\right]\right]=\Theta(n)\Theta \left(\frac{n^2}{\beta_n}\right)$, there exists $C>0$ such that the inequality
$\frac{1}{2}\E\left[\sum_{j\le C'n}\mathds{1}\left[\xi_j<\frac{\log 2}{\beta_n}\right]\right]>C\frac{n^3}{\beta_n}$ holds. Using the notation $N(n, \beta_n, j)$ from Lemma \ref{lemma:number_of_small_gaps_coupling}, by the law of large numbers
\begin{align*}
    &\P\left( \Lvert \left\{i: \frac{c(h_i)}{c(g_i)}>\frac{1}{2}\text{ and }g_i\in \partial V(H_{n, 2/n})\right\} \Lvert < \frac{Cn^3}{\beta_n} \right)\\
    &\hspace{5cm}\le \P[E_n^c]+\sum_{\pi: E(\pi)\subseteq E_n}\P\left(N(n, \beta_n, [C'n])<\frac{Cn^3}{\beta_n}\Given E(\pi)\right)\P(E(\pi))\\
    &\hspace{5cm}\stackrel{\text{Lemma \ref{lemma:number_of_small_gaps_coupling}}}{\le} \P[E_n^c]+\P[E_n]\P\left(\sum_{i\le C'n}\mathds{1}\left[\xi_i<\frac{\log 2}{\beta_n}\right]< \frac{Cn^3}{\beta_n}\right)\stackrel{\text{LLN}}{\rightarrow} 0.
\end{align*}

For $\beta_n=O(n^2)$, we cannot use Lemma \ref{lemma:size_of_gaps}, but (\ref{formula:prob_one_big_gap}) can be replaced with $\P
\left(\xi_i<\frac{\log 2}{\beta_n}\right)=\Theta(1)$, and then $\E[\sum_{i\le C'n}\xi_i]=\Theta \left(n\right)$, yielding the statement of the proposition for $Cn$ for a $C>0$.

If $s\notin V(H_{n, 2/n})$, then let $C_0$ be the connected component of $G_{n, 2/n}$ containing $s$, and recursively we choose $C_i$ as the union of $C_{i-1}$ and the connected component of $G_{n, 2/n}$ containing the endpoint outside $C_{i-1}$ of the edge with minimal label from $\partial V(C_{i-1})$. Then, even conditioning on the past, this endpoint is uniformly distributed on $V(K_n)\setminus C_{i-1}$, hence
\begin{gather}\label{formula:stepping_into_H}
    \P\left(C_{[\log n]}\cap V(H_{n, 2/n})=\emptyset\Given |V(H_{n, 2/n})|\ge d_1n \right)\le \left(1-\frac{d_1n}{n}\right)^{[\log n]}=\left(1-d_1\right)^{[\log n]}\rightarrow 0.
\end{gather}

Let $J$ be the smallest index which satisfies $C_J\cap H_{n, 2/n}\ne \emptyset$, and we denote by $\tilde{s}$ the first vertex of $H_{n, 2/n}$ built by Prim's invasion algorithm started at $s$. For an edge $e\in \partial V(H_{n, 2/n})$ which does not have an endpoint in $C_{J-1}$, we show that if $e=g_i\in \partial V(H_{n, 2/n})$ for an $i$ and $U_{h_i}-U_{g_i}<\frac{\log 2}{\beta_n}$ in the invasion started at $\tilde{s}$, then these properties are inherited for the invasion started at $s$. 

To be more precise here, we emphasize that the invasion is started at $v\in V(K_n)$ by the notation $g_i(v)$ and $h_i(v)$. If $e$ is added by Prim's invasion algorithm started at $v$, then it is also added by the invasion started at any vertex as the edge set generated by the invasion does not depend on the starting vertex, and we write $i(v, e)$ for the index satisfying $e=g_{i(v,e)}(v)$. If $e\in \partial V(H_{n, 2/n})$ is added by the invasion and it does not have an endpoint in $C_{J-1}$, then it is added later than $\tilde{s}$ by the invasion stated at $s$, so $U_{h_{i(s,e)}(s)}<U_{h_{i(\tilde{s},e)}(\tilde{s})}$ because we are taking minimum in case of $s$ on a larger set than in case of $\tilde{s}$. Hence, the condition $U_{h_{i(v, e)}(v)}-U_{g_{i(v, e)}(v)}<\frac{\log 2}{\beta_n}$ holds easier for $v=s$ then for $v=\tilde{s}$, which we wanted to show.

So, with exception of at most $L$ indices, each index appearing in (\ref{formula:steps_cheby}) for the algorithm started at $\tilde{s}$ gives us an index in (\ref{formula:steps_cheby}) for starting at $s$. If $d_1$ is small enough, then $\P(|V(H_{n, 2/n})|\ge d_1n)\rightarrow 1$ by Theorem 2.3.2 of \cite{durrett2010random}, and $\P(L\le [\log n]\given |V(H_{n, 2/n})|\ge d_1n)\rightarrow 1$ by (\ref{formula:stepping_into_H}), so we obtained $2C'n-[\log n]=2C'n(1-o(1))$ $i$'s in (\ref{formula:steps_cheby}) for $s\notin H_{n, 2/n}$ with probability $\rightarrow 1$. Then, we can finish the proof for these indices likewise we did for $s\in H_{n, 2/n}$.
\end{proof}

\begin{proof}[Proof of Theorem \ref{thm:ab_vs_inv}. b)] For some $\{k_n\}_n$, the event that there are $k_n$ steps $a(1)<\ldots <a(k_n)$ such that $\frac{c(h_{a(j)})}{c(g_{a(j)})}>\frac{1}{2}$ and $g_{a(j)}\in \partial V(H_{n, 2/n})$ is denoted by $E_n$. For $\beta_n\gg n^3$ and $k_n:=C\min \left(\frac{n^3}{\beta_n}, n\right)\rightarrow \infty$,  we have that $\P(E_n)\rightarrow 1$ by Proposition \ref{prop:number_of_small_gaps} for a small enough $C$. 

Let $A_{a(3i+1)}$ be the event that the random walk leaves $W_{a(3i+1)}$ first through $f_{a(3i+2)}$. We check that $\sum_{i=0}^{[k_n/3]-1}\mathds{1}\left[\P_{w_{a(3i+1)}}(A_{a(3i+1)})\ge \frac{2}{3}\right]$ stochastically dominates $\mathrm{Bin}([k_n/3], 1/2)$ on $E_n$. Since $U_{h_{a(3(i-1)+2)}}=U_{h_{a(3i-1)}}\le U_{g_{a(3i)}}<U_{g_{a(3i+1)}}$, turning from $\c$ to $\c^{g_{a(3(i-1)+2)}\leftrightarrow h_{a(3(i-1)+2)}}$ does not change $w_{a(3i+1)}$ or $W_{a(3i+1)}$, and then $\P_{w_{a(3i+1)}}(A_{a(3i+1)})=\P_{w_{a(3i+1)}}^{g_{a(3(i-1)+2)}\leftrightarrow h_{a(3(i-1)+2)}}(A_{a(3i+1)})$. Therefore, independently for each $i$, $\P_{w_{a(3i+1)}}(A_{a(3i+1)})\ge \frac{1}{\frac{1}{2}+1}=\frac{2}{3}$ holds with probability at least $\frac{1}{2}$ by Proposition \ref{prop:ab_is_not_prim}. 

We denote by $E_n'$ the event that there are at least $\frac{1}{7}k_n$ $i$'s satisfying $\P_{w_{a(3i+1)}}(A_{a(3i+1)})\ge \frac{2}{3}$. For $\chi_n\sim \mathrm{Bin}([k_n/3], 1/2)$, since  $\frac{k_n}{7}\le \left(\frac{1}{2}-\eps\right)\left[\frac{k_n}{3}\right]$ for a small $\eps>0$, $\P(E_n'^c)\le \P(\chi_{n}< k_n/7)+\P(E_n^c)\rightarrow 0$ by the law of large numbers. By the chain rule for conditional probabilities,
\begin{align*}
    \P_s(f_j=g_j \;\forall j\le n-1)&\le \P(E_n'^c)+\prod_{j=0}^{n-2} \P_{s}(f_{j+1}=g_{j+1}\given E_n', f_i=g_i \;\forall i\le j)\le \P(E_n'^c)+\left(\frac{2}{3}\right)^{k_n/7},
\end{align*}
tending to 0.
\end{proof}

\subsection{Comparing \texorpdfstring{$\wst^{\beta_n}_n$}{} and \texorpdfstring{$\mst_n$}{}}\label{subsec:wst_vs_mst}

\begin{proof}[Proof of Theorem \ref{thm:wst_vs_mst}] Given $U(\mst):=\{(f, U_f)\}_{f\in E(\mst_n)}$, the conditional probability and conditional expectation are abbreviated as $\P_{U(\mst)}(\cdot)$ and $\E_{U(\mst)}(\cdot)$. It is useful to observe that, for $E_n:=\{U(\mst) \given U_f\le \frac{2\log n}{n} \;\forall f\in E(\mst_n)\}$, we have $\P(E_n)\rightarrow 1$, since the Erdős-Rényi graph $G_{n, 2\log n/n}$ is connected with probability $\rightarrow 1$ by Theorem 2.8.1 of \cite{durrett2010random}.

Under conditioning on $\{(f, U_f)\}_{f\in E(\mst_n)}$, for any $e=(u,v)\notin \mst_n$, consider the random variable $\chi_e\sim \unif[m_e, 1]$ with $m_e:=\max_{f\in \path_{\mst_n}(u,v)}U_f$. For any $\omega\in E_n$, by Lemma \ref{lemma:conditioning}
$$\P_{\omega}\left(e\text{ is $n^{-\kappa}$-significant}\right)=\P_{\omega}\left(\frac{e^{-\chi_e\beta_n}}{e^{-m_e\beta_n}}>n^{-\kappa}\right)=
\P_{\omega}\left(\chi_e-m_e<\frac{\kappa \log n}{\beta_n}\right)\le \frac{1}{1-\frac{2\log n}{n}}\frac{\kappa \log n}{\beta_n}$$
and similarly $\P_{\omega}\left(e\text{ is $n^{-\kappa}$-significant}\right)\ge \frac{\kappa \log n}{\beta_n}$. 

\emph{a)} For $\beta_n\gg n^2\log n$ and $\kappa=13$, by Proposition \ref{prop:wst_vs_mst} b) and the union bound,
\begin{align*}
\P(\wst_n^{\beta_n}\ne \mst_n)&\le\P( \exists \;n^{-13}\text{-significant edges})+n^{-13/2}n^{11/2}\\
&\le\sum_{e\notin \mst_n}\P(e\text{ is $n^{-13}$-significant}\given E_n)+\P(E_n^c)+n^{-1}\\
&=\Theta\left(\frac{n^2\log n}{\beta_n}\right)+\P(E_n^c)+n^{-1}\rightarrow 0.
\end{align*}

\emph{b)} Conditioning on any $U(\mst)$, we have that $\P_{U(\mst)}\left(e\text{ is $\frac{1}{2}$-significant}\right)\ge \frac{\log 2}{\beta_n}$ for any $e\notin \mst_n$, and the $\frac{1}{2}$-significance of the external edges is independent by Lemma \ref{lemma:conditioning}. Therefore, given $U(\mst)$, the
number of $\frac{1}{2}$-significant edges stochastically dominates $X_n:=\sum_{e\notin \mst_n}\xi_e$ for the i.i.d.~random variables with distribution $\xi_e\sim \ber\left(\frac{\log 2}{\beta_n}\right)$. Then $\exists C>0$ such that $\frac{1}{2}\E_{U(\mst)}[X_n]\ge C\frac{n^2}{\beta_n}$, so $\P\left(X_n< C\frac{n^2}{\beta_n}\right)\rightarrow 0$ by the law of large numbers. Then Proposition \ref{prop:wst_vs_mst} a) with $\eps_0:=1/2$ and $|E_{\mathrm{sig}}\left(\eps_0\right)|\ge \left[C\frac{n^2}{\beta_n}\right]\rightarrow \infty$ implies that
\begin{align*}
    \P(\wst_n^{\beta_n}=\mst_n)&\le  \P\left(\text{\# of $\frac{1}{2}$-significant edges}< \frac{Cn^2}{\beta_n}\right)+\frac{1}{1+\frac{1}{2}\left[\frac{Cn^2}{\beta_n}\right]}\\
    &\le \P\left(X_n< \frac{Cn^2}{\beta_n}\right)+o(1)\rightarrow 0.
\end{align*}
\end{proof}

We close our subsection with a proposition about how far the labels of the edges running between different connected components of the Erdős-Rényi graphs can be from the labels appearing in $\mst_n$ connecting these components. It is a simple corollary of our computations on the significant edges. In Section \ref{sec:wst_vs_mst}, for each property of $\wst_n^{\beta_n}$, we also phrase the corresponding property of the Erdős-Rényi graphs. Hence, we include Proposition \ref{prop:wst_mst_E-R}, although this is not used later in the paper.

\begin{prop}\label{prop:wst_mst_E-R} For a sequence $\delta_1, \delta_2, \ldots$ of positive numbers, we consider $$B_n:=\left\{E(G_{n, p+\delta_n})\mathbin{\big\backslash}\bigcup_{\substack{C\text{ is conn.}\\\text{comp. of }G_{n, p}}}E(G_{n, p+\delta_n}[V(C)])\subseteq \mst_n \; \forall p\in [0, 1]\right\}.$$

\begin{itemize}
    \item[]    If $\delta_n\ll \frac{1}{n^2}$, then $\P(B_n)\rightarrow 1$.
    \item[]    If $\delta_n\gg \frac{1}{n^2}$, then $\P(B_n)\rightarrow 0$.
\end{itemize}
\end{prop}

Before the proof of this proposition, we check the following equivalence.
\begin{lemma}
    \label{lemma:equv}
    The following are equivalent:
    \begin{itemize}
        \item[a)] $E(G_{n, p+\delta_n})\mathbin{\big\backslash}\bigcup_{\substack{C\text{ is conn.}\\\text{comp. of }G_{n, p}}}E(G_{n, p+\delta_n}[V(C)])\subseteq \mst_n \; \forall p\in [0, 1]$,
        \item[b)] $\forall e=(u, v)\notin \mst_n, \forall f\in \path_{\mst_n}(u, v)$ we have that $U_e-U_f\ge \delta_n$.
    \end{itemize}
\end{lemma}
\begin{proof}
\emph{a)$\Rightarrow$b)} For any fixed $e\notin \mst_n$, consider $f'\in \path_{\mst_n}(e)$ with $U_{f'}=\max_{f\in \path_{\mst_n}(e)}U_f$. Then for $p<U_{f'}$ the edges $f'$ and $ e$ run outside of any connected component of $G_{n, p}$. Since $e\notin \mst_n$, a) implies that $U_e-p> \delta_n$. Letting $p\rightarrow U_{f'}$  we get that $U_e-U_{f'}\ge \delta_n$.

\emph{b)$\Rightarrow$a)} For any fixed $p$, consider an edge $e=(u, v)\notin \mst_n$ running between different connected components of $G_{n, p}$. Then there exists $f\in \path_{\mst_n}(u, v)$ which runs between different connected components of $G_{n, p}$. Then b) implies that $\delta_n\le U_e-U_{f}<U_e-p$.
\end{proof}

\begin{proof}[Proof of Proposition \ref{prop:wst_mst_E-R}]
    If $\delta_n\ll \frac{1}{n^2}$, then the computations of Theorem \ref{thm:wst_vs_mst} a) show that $$\P\left(\exists e\notin \mst_n:\;\chi_e-m_e<\delta_n\right)\le \P(E_n^c)+\Theta(n^2\delta_n)\rightarrow 0,$$
    where $E_n=\{U_f\le \frac{2\log n}{n} \;\forall f\in E(\mst_n)\}$ as before. Hence Lemma \ref{lemma:equv} b)  is satisfied with probability $\rightarrow 1$ which is equivalent to a), i.e.,~$\P(B_n)\rightarrow 1$ as $n\rightarrow \infty$.
    
    Repeating the proof of Theorem \ref{thm:wst_vs_mst} b), if $\delta_n\gg \frac{1}{n^2}$, then Lemma \ref{lemma:equv} b) is not satisfied with probability $\rightarrow1$ giving us $\P(B_n)\rightarrow 0$. 
\end{proof}

\subsection{The local weak limit and expected total length and for large \texorpdfstring{$\beta_n$}{}'s}\label{subsec:mst_length}

\begin{proof}[Proof of Theorem \ref{thm:local_limit}]
For $\beta_n\gg n\log n$, by formula (\ref{formula:expected_diff}), there exist some $t_n=o(n)$ such that 
\begin{gather}\label{formula:for_local_lim}
\P(|E(\wst^{\beta_n}_n)\triangle E(\mst_n)|\le t_n\given \{U_e\}_{e\in E(K_n)})\rightarrow 1
\end{gather}
for a.a.s.~labels $\{U_e\}_{e\in E(K_n)}$, i.e.~with probability $\rightarrow 1$. Therefore, the local convergence follows from the following observation:

\emph{Consider a sequence of graphs $G_n$ with a locally finite limit local weak limit.
Suppose that there exist some  $t_n=o(|V (G_n)|)$ such that $\P(|E(G_n)\triangle E(H_n)| \le t_n)\rightarrow 1$. Then $H_n$ has the same local weak limit as $G_n$.}

In order to show this statement,  denote by $S_n$ the set of vertices formed by the endpoints of $E(G_n)\triangle E(H_n)$. Then $\P(|S_n|<2t_n) \rightarrow 1$.
By the local finiteness, for any $r>0$, we have that
$$\lim_{M\rightarrow\infty} \sup_n \P(|V (B_{G_n}(o, r))| > M) = 0.$$
Notice that $\P(B_{G_n}(o, r) \cap S_n\ne \emptyset, |V (B_{G_n}(o, 2r))| \le M) \le \frac{|S_n|M}{|V (G_n)|}$ which follows from the double counting of $A_n:=\{ (o,s) : o \in V_n, s \in S_n\cap B_{G_n}(o ,r), |B_s(r)| \le M \}$:
$$|V_n|\P(B_{G_n}(o, r) \cap S_n\ne \emptyset, |V (B_{G_n}(o, 2r))| \le M)\le |A_n| \le |S_n|M.$$

Therefore,
\begin{align*}
\P(B_{G_n}(o, r) \cap S_n\ne \emptyset)&\le 
    \P(B_{G_n}(o, r) \cap S_n\ne \emptyset, |V (B_{G_n}(o, 2r))| \le M)+\P(|V (B_{G_n}(o, 2r))| >M) \\
    &\le \frac{|S_n|M}{|V (G_n)|}+\P(|V (B_{G_n}(o, 2r))| >M)\rightarrow \P(|V (B_{G_n}(o, 2r))| >M) 
\end{align*}
on $\{|S_n|<2t_n\}$. Letting  $M\rightarrow\infty$ gives us that $\P(B_{G_n}(o, r) \cap S_n\ne \emptyset)\rightarrow 0$.

From this statement, we obtain Theorem \ref{thm:local_limit} the following way: for a given $M_r$, we denote by $A_n$ the set of labels which satisfy $\P(|V (B_{\mst_n}(o, 2r))| >M_r|\given \{U_e\}_e)\rightarrow 0$ and (\ref{formula:for_local_lim}). Note that $M_r$ can be chosen such that $\P(A_n)\rightarrow 1$, e.g.~$M_r=r^{2}\log r$ is a good choice by Theorem 1.3 of \cite{addario2013local}. Then,
\begin{align*}
&\E\left[\P\left(B_{\mst_n}(o, r)\ne B_{\wst_n^{\beta_n}}(o, r)\Given \{U_e\}_{e\in E(K_n)}\right)\right]\\
    &\hspace{5cm}\le P(A_n^c)+\P(|V (B_{\mst_n}(o, 2r))| >M_r|\given A_n)+o(1)=o(1).
\end{align*}
\end{proof}

To study the expected total length and the typical diameter of the minimum spanning tree, an essential connection between the minimum spanning tree and the Erdős Rényi graphs is used: for any $p$, the vertices of the connected components of $G_{n, p}$ and of $\mst_n\cap G_{n, p}$ are the same. This connection follows from the greediness of Prim's invasion algorithm and can be formulated as $V\left(\mathcal{C}_v^{E(G_{n, p})}\right)=V\left(\mathcal{C}_v^{E(G_{n, p})\cap E(\mst_n)}\right)\; \forall v\in V(K_n), \forall p$.

This property does not hold for $\wst^{\beta_n}_n$, but we can prove a relaxed version of it, illustrated in Figure \ref{fig:wst_mst_layers}, which is suitable to study the expected total length and the typical diameter of $\wst_n^{\beta_n}$.

In Subsection \ref{subsec:ideas}, we have already given a short proof of Theorem \ref{thm:length} a) for $\beta_n\gg n\log^2 n$. In this subsection, the new proof works even for $\beta_n\gg n\log n$, which does not seem a big improvement. On the other hand, the work done here is necessary for Subsection \ref{subsec:wst_diam_abr}: Lemma \ref{lemma:E-R_comps_A-B} will be used to prove Lemma \ref{lemma:H_comps}.

\begin{lemma}\label{lemma:E-R_comps_A-B} 
We write $\P_{\mathbf{u}}(\cdot)=\P(\;\cdot\given U_e=u_e \forall e\in E(K_n))$, and $\mathcal{C}_v^E$ for the connected component w.r.t.~$E$ containing $v$. For any $E\subseteq E(K_n)$ that gives a
connected subgraph and any $V\subseteq V(K_n)$, we denote by $\langle V\rangle_E$ the smallest connected graph w.r.t.~$E$ containing $V$.

Then there exists $N_0$ such that for any $\mathbf{u}\in [0,1]^{\binom{n}{2}}$ and $n\ge N_0$
  \begin{gather*}
  {\P}_{\mathbf{u}}\left(\left\langle V\left(\mathcal{C}_v^{E(G_{n, p})}\right)\right\rangle_{E\left(\wst_{n}^{\beta_n}\right)}\subseteq G_{n, p+\frac{50\log n}{\beta_n}} \ \forall v\in V(K_n) \;\forall p\in[0, 1]\right)\ge 1-n^{-2}.
	\end{gather*}
\end{lemma}

\begin{proof}
Fix $n$, $\mathbf{u}\in [0,1]^{\binom{n}{2}}$, $\kappa>0$, $v\in V(K_n)$ and $i\le [\delta_n^{-1}]$, where $\delta_n:=\frac{\kappa \log n}{\beta_n}$ and $p_{n, i}:=i\delta_n$. We run the Aldous-Broder algorithm on $\c_{\mathbf{u}}:=\{\exp(-\beta_nu_e)\}$ starting at $v$. By Lemma \ref{lemma:middlecond} with 
\begin{align*}
	&E_{\bcond}:=\{e\in E(K_n)\given u_e\le p_{n, i}\}=\{e\in E(K_n)\given \c_{\mathbf{u}}(e)\ge \exp(-\beta_np_{n, i})\}, \\
	&E_{\scond}:=\left\{e\in E(K_n)\given u_e\ge p_{n, i}+\frac{\kappa\log n}{\beta_n}\right\}=\{e\in E(K_n)\given \c_{\mathbf{u}}(e)\le \exp(-\beta_np_{n, i}-\kappa\log n)\}
\end{align*}
and $\eps_0=\exp(-\kappa\log n)=n^{-\kappa}$, we obtain that
\begin{align*}
\P_{v, c_\mathbf{u}}\left( \langle V(\mathcal{C}_v^{E(G_{n, p_{n, i}})})\rangle_{E\left(\wst_{n}^{\beta_n}\right)}\not \subseteq G_{n, p_{n, i+1}}\right)=\P\left(V\left(\C_{v}^{E_{\bcond_i}}\right) \not\subseteq \C_{v}^{E(\wst_{\beta_n})\setminus E_{\scond_i}}\right)\le n^{(-\kappa+7)/2}.
\end{align*}

If $(i-1)\delta_n\le p\le i\delta_n$, then we have that $\mathcal{C}_v^{E(G_{n, p})}\subseteq \mathcal{C}_v^{E(G_{n, i\delta_n})}$, so the previous line gives us that $\langle V(\mathcal{C}_v^{E(G_{n, p})})\rangle_{E\left(\wst_{n}^{\beta_n}\right)}\subseteq G_{n, p+2\delta_n}$ with probability $\ge 1-n^{(7-\kappa)/2}$.

The $n$'s with $\beta_n\ge n^6$ also satisfy $\beta_n\gg n^5\log n$, so the Aldous-Broder and Prim's invasion algorithm build the same edges with probability $1-O(n^{-2})$ by (\ref{formula:number_small_gaps}), and we even have $\langle V(\mathcal{C}_v^{E(G_{n, p})})\rangle_{E\left(\wst_{n}^{\beta_n}\right)}\subseteq G_{n, p}$ on this event.

Hence, we may assume that $\beta_n\le n^6$. Choosing $\kappa:=25$, by the union bound for $v\in V(K_n)$ and $i=1, \ldots , n^6$,
\begin{align*}
&\P_{\mathbf{u}}\left(\exists v\in V(K_n) \;\exists p\in[0, 1]:\ \langle V(\mathcal{C}_v^{E(G_{n, p})})\rangle_{E\left(\wst_{n}^{\beta_n}\right)}\not \subseteq G_{n, p+2\delta_n}\right)\\
&\hspace{3cm}\le \P_{\mathbf{u}}\left(\exists v\in V(K_n) \;\exists i\le n^6:\ \left(\langle V(\mathcal{C}_v^{E(G_{n, p_{n, i}})})\rangle_{E\left(\wst_{n}^{\beta_n}\right)}\right)\not \subseteq G_{n, p_{n, i+1}}\right)\\
&\hspace{3cm}\le n^7 n^{(7-\kappa)/2}\le n^{(21-\kappa)/2}= n^{-2}.
\end{align*}
\end{proof}

\begin{proof}[Proof of Theorem \ref{thm:length} a) and b)]
We show that $\E\left[L(\wst_n^{\beta_n})\right]-\E\left[L(\mst_n)\right]=O\left(\frac{\log n}{\beta_n}\right)+O(n^{-1})$.

We denote by $k(\wst^{\beta_n}_n,p)$ and $k(\mst_n,p)$ the number of the connected components of $\wst^{\beta_n}_n \cap G_{n, p}$ and $\mst_n \cap G_{n, p}$, respectively. We also use the notation $N(\wst^{\beta_n}_n, p):=|E(\wst^{\beta_n}_n \cap G_{n, p})|$ which can be written also as $N(\wst^{\beta_n}_n, p)=\sum_{e\in \wst^{\beta_n}_n}\mathds{1}[U_e\le p]$.

The integral (8) of \cite{steele2002minimal} expressing $L\left(\mst_n\right)$ can be generalized to the weighted spanning trees as  $L\left(\wst_n^{\beta_n}\right)=\int_{0}^1 k(\wst^{\beta_n}_n, t) \;d t-1$. Indeed,
\begin{align*}
L\left(\wst_n^{\beta_n}\right)&=\sum_{e\in \wst^{\beta_n}_n}U_e=\sum_{e\in \wst^{\beta_n}_n}\int_0^1 \mathds{1}\left[U_e>t\right] \;dt\\
&=\int_0^1 n-1-N(\wst^{\beta_n}_n, t) \;dt=\int_0^1 k(\wst^{\beta_n}_n, t)-1 \;dt.
\end{align*}

Lemma \ref{lemma:E-R_comps_A-B} with $\delta_n=\frac{25\log n}{\beta_n}$ implies that for large $n$ and any $\{u_e\}_{e\in E(K_n)}$,
$$\P\left(k(\wst^{\beta_n}_n, p+2\delta_n)\le k\left(\mst_n, p\right) \ \forall p\in [0, 1]\given U_e=u_e \;\forall e\in E(K_n)\right)\ge 1-n^{-2},$$
so our statement for $\beta_n\ge \log n$ is obtained by
\begin{align*}
\E\left[L(\wst_n^{\beta_n})-L(\mst_n)\right]&\le \E\left[\int_{0}^1 k_{\mst_n}\left(t-2\delta_n\right)-k_{\mst_n}\left(t\right) \;dt\right]+n\cdot n^{-2}\\
&\le\E\left[\int_{0}^{2\delta_n}n \;dp\right]+n^{-1}=O\left(\frac{n\log n}{\beta_n}\right)+n^{-1}
\end{align*}
For $\beta_n\le \log n$, the statement follows from the trivial estimate $\E[L(\wst_n^{\beta_n})]\le n-1$.
\end{proof}

\subsection{The typical diameter of the \texorpdfstring{$\wst^{\beta_n}_n$}{} for large \texorpdfstring{$\beta_n$}{}'s}
\label{subsec:wst_diam_abr}

The main result of \cite{addario2009critical} is the following.
\begin{theorem*}[Theorem 1 of \cite{addario2009critical}]
	$$\E[\diam(\mst_n)]=\Theta(n^{1/3}).$$  
\end{theorem*}

Our key observation is that the proof appearing in that paper can be modified to prove our Theorem \ref{thm:has_diam_mst}, claiming that, for any $\beta_n\gg n^{4/3}\log n$ and $\eps>0$, there exist $c=c(\eps)$ and $C=C(\eps)$ such that $$\P(cn^{1/3}\le \wst_{n}^{\beta_n}\le Cn^{1/3})\ge 1-\eps.$$ However, similarly to Subsection \ref{subsec:mst_length}, an extra difficulty appears compared to the case of $\mst_n$: in \cite{addario2009critical}, the main connection between $\mst_n$ and the Erdős-Rényi graphs is that $\mst_n\cap H_{n, p_{i}}$ is connected, which does not hold for $\wst_n^{\beta_n}$. Fortunately, this can be replaced by Lemma \ref{lemma:H_comps}, which
is the key novel ingredient in
our proof.

The idea behind the lower bound of the typical diameter is the following. We will see in Lemma~\ref{lemma:H_comps} that $V(H_{1/n-2a_n/n^{4/3}})$ is connected w.r.t.~$E(G_{n, 1/n-a_n/n^{4/3}}\cap \wst^{\beta_n}_n)$ with probability $\rightarrow$ 1 if $\beta_n\ll n^{4/3}\log n$ and $a_n\ge 1$. On the other hand, by Subsection 5.3 of \cite{janson2011random}, $H_{1/n-2a_n/n^{4/3}}$ has $\Theta\left(\frac{n^{2/3}}{a_n^2}\log a_n\right)$ vertices and even the connected component of $G_{1/n-a_n/n^{4/3}}$ containing $H_{1/n-2a_n/n^{4/3}}$ is a tree for any $1\ll a_n\ll n^{1/3}$ with probability $\rightarrow 1$. Therefore, the edges of the $\wst$ and the $\mst$ are the same inside $H_{1/n-2a_n/n^{4/3}}$, and distributed as the uniform spanning tree on $\Theta(n^{2/3}a_n^{-2}\log a_n)$ vertices. Since $a_n$ can tend to infinity arbitrary slowly, it is reasonable to believe that the typical diameter of $\wst_n^{\beta_n}$ grows at least like $\left(\Omega\left(n^{2/3}\right)\right)^{1/2}=\Omega\left(n^{1/3}\right)$. The precise proof is detailed at the end of this subsection.

The upper bound for the typical diameter of $\wst_{n}^{\beta_n}$ can also be done similarly to the one for $\mst_n$.  We summarize now the proof of \cite{addario2009critical}, rephrasing it to deal with $\E[\diam(\wst_{n}^{\beta_n})|A^{\mathrm{incr}}_n]$ for some $A^{\mathrm{incr}}_n$ with $\P(A^{\mathrm{incr}}_n)\rightarrow 1$.
We accept the results of \cite{addario2009critical} on the Erdős-Rényi graph as facts, and we only detail how the result on the diameter of $\wst_n^{\beta_n}$ (or 
$\mst_n$) follows from these facts.

\begin{notation}\label{notation:p_i's}
Consider $f_0:=f_0^{(n)}\gg 1$ and $f_i:=f_i^{(n)}:=(5/4)^i f_0^{(n)}$ for $i\in \frac{\N}{2}$. 

Let $t:=t^{(n)}$ be the first integer such that $f_t^{(n)}\ge \frac{n^{1/3}}{\log n}$.

Let $p_i:=p_i^{(n)}:=\frac{1}{n}+\frac{f_i}{n^{4/3}}$ for $i=0, \frac{1}{2}, \ldots, t$.
In order to shorten the formulation of Lemma~\ref{lemma:H_comps}, slightly inconsistently, we also introduce $p_{-1/2}=\frac{1}{n}-\frac{f_0}{n^{4/3}}$ and $p_{-1}=\frac{1}{n}-\frac{2f_0}{n^{4/3}}$.

For any $E\subseteq E(K_n)$ that gives a
connected subgraph and any $V\subseteq V(K_n)$, we denote by $\langle V\rangle_E$ the smallest connected graph w.r.t.~the edges $E$ containing the vertices $V$.
\end{notation}

Then we have a sequence of probabilities $\{p_i\}_{i=0, \frac{1}{2}, \ldots, t}$ interpolating between $p_0=\frac{1}{n}+\frac{f_0}{n^{4/3}}$ and $p_t=\frac{1}{n}+\frac{1}{n \log n}$.

\begin{lemma}\label{lemma:H_comps} We consider $\beta_n\gg n^{4/3}\log n$, use Notation \ref{notation:p_i's} and write $\P_{\mathbf{u}}(\cdot):=\P[\cdot|U_e=u_e \;\forall e\in E(K_n)]$. Then as $n\rightarrow \infty$ we have that
  \begin{gather*}
		\inf_{\mathbf{u}\in [0,1]^{\binom{n}{2}}}{\P}_{\mathbf{u}}\left(\forall i\in \left\{-1, -\frac{1}{2}, \ldots, t\right\}:\;\langle V(H_{n, p_i})\rangle_{E\left(\wst_{n}^{\beta_n}\right)}\subseteq G_{n, p_{i+\frac{1}{2}}}\right)\ge 1-o(n^{-2}).
	\end{gather*}
\end{lemma}
\begin{proof}
This is a simple corollary of Lemma \ref{lemma:E-R_comps_A-B}, since $(p_{i+1/2}-p_i)\ge p_{1/2}-p_0=\Theta(f_0 n^{-4/3})\ge \frac{50\log n}{\beta_n}$ if $\beta_n\gg n^{4/3}\log n$ and $n$ is large enough.
\end{proof}

\begin{notation*}
Consider $T_{n, p_i}:=\langle V(H_{n, p_i})\rangle_{E(\wst_{n}^{\beta_n})}$. For its diameter, we write $D_i:=\diam(T_{n, p_i})$. 
 
The length of the longest path in a not necessarily connected graph is denoted by $\lp$.

Let $A^{\mathrm{incr}}_n$ be the event $A^{\mathrm{incr}}_n:=A^{\mathrm{incr}}_{n, f_0}:=\bigcap_{i=0, \frac{1}{2}, \ldots, t} \left\{E(T_{n, p_i})\subseteq E(G_{n, p_i+\frac{1}{2}})\right\}\cap \left\{H_{p_i}\subseteq H_{p_{i+\frac{1}{2}}}\right\}.$
\end{notation*}

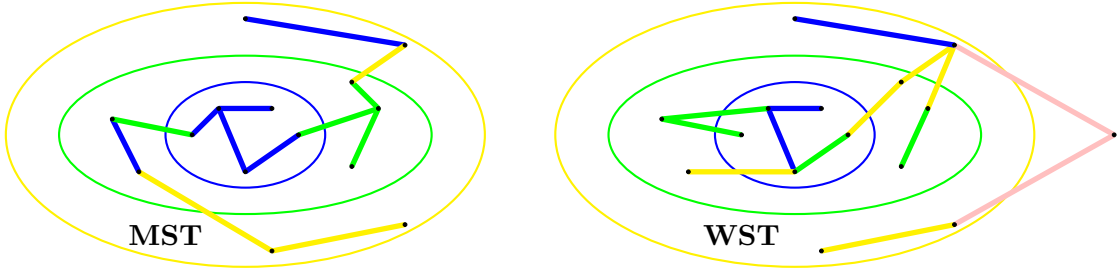
\begin{figure}\label{fig:wst_mst_layers}
\centering
\begin{tikzpicture}[scale=0.7]
\coordinate (1) at (-1, 0);
\coordinate (2) at (-0.5, 0.5);
\coordinate (3) at (0.5, 0.5);
\coordinate (4) at (1, 0);
\coordinate (5) at (0, -0.7);
\coordinate (6) at (2, 1);
\coordinate (7) at (2.5, 0.5);
\coordinate (8) at (2, -0.6);
\coordinate (9) at (-2, -0.7);
\coordinate (10) at (-2.5, 0.3);
\coordinate (11) at (0.5, -2.2);
\coordinate (12) at (0, 2.2);
\coordinate (13) at (3, -1.7);
\coordinate (14) at (3, 1.7);

\draw[thick, yellow] (0, 0) ellipse (4.5 and 2.5); 
\draw[thick, green] (0, 0) ellipse (3.5 and 1.5);
\draw[thick, blue] (0, 0) ellipse (1.5 and 1); 

\draw[blue, line width=2pt](1) -- (2);
\draw[blue, line width=2pt](2) -- (3);
\draw[blue, line width=2pt](2) -- (5);
\draw[blue, line width=2pt](5) -- (4);
\draw[blue, line width=2pt](9) -- (10);
\draw[blue, line width=2pt](12) -- (14);

\draw[green, line width=2pt](7) -- (8);
\draw[green, line width=2pt](6) -- (7);
\draw[green, line width=2pt](1) -- (10);
\draw[green, line width=2pt](7) -- (4);

\draw[yellow, line width=2pt](9) -- (11);
\draw[yellow, line width=2pt](6) -- (14);
\draw[yellow, line width=2pt](11) -- (13);

\foreach \x in {1,...,14} \filldraw[] (\x) circle (1pt); 
\node[] at (-1.5, -1.9) {\textbf{MST}};
\end{tikzpicture}
\qquad
\begin{tikzpicture}[scale=0.7]
\coordinate (1) at (-1, 0);
\coordinate (2) at (-0.5, 0.5);
\coordinate (3) at (0.5, 0.5);
\coordinate (4) at (1, 0);
\coordinate (5) at (0, -0.7);
\coordinate (6) at (2, 1);
\coordinate (7) at (2.5, 0.5);
\coordinate (8) at (2, -0.6);
\coordinate (9) at (-2, -0.7);
\coordinate (10) at (-2.5, 0.3);
\coordinate (11) at (0.5, -2.2);
\coordinate (12) at (0, 2.2);
\coordinate (13) at (3, -1.7);
\coordinate (14) at (3, 1.7);
\coordinate (15) at (6, 0);

\draw[thick, yellow] (0, 0) ellipse (4.5 and 2.5); 
\draw[thick, green] (0, 0) ellipse (3.5 and 1.5);
\draw[thick, blue] (0, 0) ellipse (1.5 and 1); 

\draw[blue, line width=2pt](2) -- (3);
\draw[blue, line width=2pt](2) -- (5);

\draw[green, line width=2pt](1) -- (10);
\draw[green, line width=2pt](2) -- (10);
\draw[green, line width=2pt](5) -- (4);
\draw[green, line width=2pt](7) -- (8);

\draw[yellow, line width=2pt](4) -- (6);
\draw[yellow, line width=2pt](6) -- (14);
\draw[yellow, line width=2pt](7) -- (14);
\draw[yellow, line width=2pt](5) -- (9);
\draw[yellow, line width=2pt](11) -- (13);

\draw[pink, line width=2pt](13) -- (15);
\draw[pink, line width=2pt](14) -- (15);
\draw[blue, line width=2pt](12) -- (14);
\foreach \x in {1,...,15} \filldraw[] (\x) circle (1pt);
\node[] at (-1, -1.9) {\textbf{WST}};
\end{tikzpicture}
\caption{$\mathcolor{blue}{\mathbf{E(G_{n, p_0})}}$, $\mathcolor{green}{\mathbf{E(G_{n, p_{1/2}})\setminus E(G_{n, p_{0}})}}$, $\mathcolor{black!10!yellow}{\mathbf{E(G_{n, p_{3/2}})\setminus E(G_{n, p_{1}})}}$ and  $\mathcolor{black!5!pink}{\mathbf{E(G_{n, p_{2}})\setminus E(G_{n, p_{3/2}})}}$ are the edge sets denoted by the respective colors and we draw blue, green and yellow ellipses for $H_{n, p_{0}}$, $H_{n, p_{1/2}}$ and $H_{n, p_{1}}$. On $A^{\mathrm{incr}}_n$, the largest connected components $\{H_{n, p_{j/2}}\}_{j\ge 0}$ form a growing sequence of subgraphs of $K_n$, and can be thought as layers. The vertices of a layer can be connected by the edges of $\wst^{\beta_n}$ coming from at most one layer above. This is slightly more difficult compared to the case of $\mst_n$, where the vertices of each layer can be connected inside the layer by the edges of $\mst_n$. }
\end{figure}

\begin{prop}\label{prop:to_supcrit} For any $f_0=f_{0}^{(n)}\gg 1$ there exists $N_0$ such that for $n\ge N_0$,
$$\E[D_t\given A^{\mathrm{incr}}_n]\le 3f_0^4n^{1/3}.$$
\end{prop}
\begin{remark*}
Since $p_0^{(n)}-\frac{1}{n}\gg \frac{1}{n^{4/3}}$, we have that $\P\left(H_{n, p_i}\subseteq H_{n, p'}\; \forall p_0\le p\le p'\right)\rightarrow 1$ by \cite{luczak1990component}. This and Lemma \ref{lemma:H_comps} imply that $\P(A^{\mathrm{incr}}_n)\rightarrow 1$ for $\beta_n \gg n^{4/3}\log n$.
\end{remark*}
In order to prove Proposition \ref{prop:to_supcrit}, we recall some results of \cite{addario2009critical}.

\begin{defi}
	An instance of $G_{n, p_i}$ is \emph{well-behaved} if
	\begin{itemize}
		\item[(I)]$|H_{n, p_i}|\ge \frac{3}{2}n^{2/3}f_i$, $\lp(H_{n, p_i})\le f_i^4n^{1/3}$ and
		\item[(II)] $\lp(G_{n, p_{i+1}}[V(K_n)\setminus V(H_{n, p_i})])\le \frac{n^{1/3}}{\sqrt{f_i}}$.
	\end{itemize}
 We denote by $i^*$ the smallest half-integer such that $G_{n, p_j}$ is well-behaved for all $i^*\le j \le t-\frac{1}{2}$ (or $i^*=t$ if $G_{n, p_{t-\frac{1}{2}}}$ is not well behaved).
\end{defi}

\begin{lemma}[Lemma 2 of \cite{addario2009critical}]\label{lemma2_abr}
	Let $G\subseteq G'$ be graphs and consider some connected components $H\subseteq H'$ of $G$ and $G'$ respectively. Then $$\diam(H')\le \diam(H)+2\lp(G'[V(H')\setminus V(H)])+2.$$
\end{lemma}

\begin{lemma}[Lemma 3  of \cite{addario2009critical}]\label{lemma3_abr}
$$\P(G_{n, p_j}\text{ is not well-behaved})\le 6e^{-\sqrt{f_j/8}}.$$
\end{lemma}
Most of the computations of \cite{addario2009critical} is to prove Lemma 3.

\begin{proof}[Proof of Proposition \ref{prop:to_supcrit}.] In the proof of \cite{addario2009critical}, they compared $\diam(\mst_n\cap H_{n, p_i})$ for $i=0, 1, \ldots, t$. We are modifying their reasoning to $\diam(T_{n, p_i})=\diam\left(\langle V(H_{n, p_i})\rangle_{E(\wst_{n}^{\beta_n})}\right)$ for $i=0, \frac{1}{2}, \ldots, t$ on $A^{\mathrm{incr}}_n$.

\emph{1) On the event $A^{\mathrm{incr}}_n$, we have that ${D_t-D_{i^*}\le n^{1/3}}$ for large enough $n$'s.}
	
On $\{i\ge i^*\}\cap A^{\mathrm{incr}}_n$, the graphs $G:=G_{n, p_{i+\frac{1}{2}}}[V(T_{n, p_{i}})]$, $H=T_{n, p_i}$, $G'=G_{n, p_{i+1}}[V(T_{n, p_{i+\frac{1}{2}}})]$ and $H'=T_{n, p_{i+\frac{1}{2}}}$ satisfy the conditions of Lemma \ref{lemma2_abr}, $V(H_{n, p_{i}})\subseteq V(H)$ and $H'\subseteq H_{n, p_{i+1}}$, so 
\begin{gather}\label{formula:difference_diams}
    D_{i+\frac{1}{2}}-D_{i}\le 2\lp(G_{n, p_{i+1}}[V(H_{n, p_{i+1}})\setminus V(H_{n, p_i})])+2\le 2\frac{n^{1/3}}{\sqrt{f_i}}+2.
\end{gather}
	
Since $t=O(\log n)$ and $f_0\gg 1$, on $A^{\mathrm{incr}}_n$, for large enough $n$, we have that
	$$D_t-D_{i^ *}\le 2\sum_{i=i^*,i^*+\frac{1}{2},\ldots , t-\frac{1}{2}}\left(2\frac{n^{1/3}}{\sqrt{f_i}}+2\right)\le O(\log n)+2\frac{n^{1/3}}{\sqrt{f_{0}}}\sum_{i=0}^{\infty}\left(\frac{4}{5}\right)^{i/2}\stackrel{\text{large $n$}}{\le} n^{1/3}.$$ 
	
	\emph{2) It is also true that $\E[D_{i^*}|A^{\mathrm{incr}}_n]\le 2f_0^{4}n^{1/3}$ for large enough $n$.}
	
 By Lemma \ref{lemma3_abr}, with large enough $f_0^{(n)}$, and since $\lim_{x\rightarrow\infty}e^{-\sqrt{x}}x^5=0$,
\begin{align}\label{formula_istar}
     \P\left(i^*=j+1/2\right)\le \P\left(G_{n, p_j}\text{ is not well-behaved}\right)\le 6e^{-\sqrt{f_j/8}}\le \left(\frac{5}{4}\right)^{5/2}f_{j}^{-5}=f_{j+\frac{1}{2}}^{-5}.
 \end{align}
Moreover, $f_t\sim\frac{n^{1/3}}{\log n}$, so
\begin{align}\label{formula_istar2}
         &\P(i^*=t)\le 6e^{-\sqrt{f_{t-1/2}/8}}=O\left(\frac{1}{n}\right).
\end{align}
For large enough $n$, we have that $\P(A^{\mathrm{incr}}_n)\ge \frac{1}{2}$, hence
	\begin{align*}
		\E[D_{i^*}|A^{\mathrm{incr}}_n]&\le\E[\lp(H_{n, p_{i^*}})|A^{\mathrm{incr}}_n]\le f_0^4n^{1/3}+\sum_{i=\frac{1}{2},\ldots, , t-\frac{1}{2}}f_i^4n^{1/3}\P(i^*=i|A^{\mathrm{incr}}_n)+n\P(i^*=t|A^{\mathrm{incr}}_n)\\
		&\stackrel{(\ref{formula_istar}), (\ref{formula_istar2})}{\le} f_0^4n^{1/3}+\sum_{i=\frac{1}{2},1, \ldots}f_{i}^4n^{1/3}\frac{1}{f_{i}^5}\frac{1}{\P(A^{\mathrm{incr}}_n)}+n\frac{1}{n}\frac{1}{\P(A^{\mathrm{incr}}_n)}\\
        &\stackrel{\text{large }n}{\le} f_0^4n^{1/3}+2n^{1/3}\frac{1}{f_0}\sum_{i=0}^{\infty} \left(\frac{5}{4}\right)^{i/2}+2\stackrel{\text{large }n}{\le}  2f_0^4n^{1/3}.
	\end{align*}

Combining 1) and 2) gives us the statement for large enough $n$'s.
\end{proof}

In Lemma 4 of \cite{addario2009critical}, for $D_{t}^{(\mathsf{MST})}:=\mst_n\cap H_{n, p_t}$, it is shown that  $$\E\left[\diam(\mst_n)\right]-\E\left[D_{t}^{(\mathsf{MST})}\right]=O\left(n^{1/6}(\log n)^{5/2}\right).$$
The proof is based on that Prim's invasion algorithm started at $s\notin V(H_{n, t})$ reaches $H_{n, t}$ fast, since in any step, conditioning on the already built subtree, any vertex outside the subtree is added with the same probability. This holds also for the Aldous-Broder algorithm, so we can also prove the analugous statement for $\wst_n^{\beta_n}$.

\begin{lemma}[The analogue of Lemma 4 of \cite{addario2009critical}]\label{lemma:from_supcrit_to_whole} 
	$$\E\left[\diam(\wst_{n}^{\beta_n})\Given A^{\mathrm{incr}}_n\right]-\E\left[D_t\Given A^{\mathrm{incr}}_n\right]=O\left(n^{1/6}(\log n)^{5/2}\right).$$
\end{lemma}
\begin{proof}
    Consider $E_n:=\{G_{n, p_t}\text{ is well-behaved}\}$, which can be rewritten as
    $$E_n=\left\{|H_{n, p_t}|>\frac{3}{2}\frac{n}{\log n}\text{ and } \lp(G_{n, p_{t}}[V(K_n)\setminus V(H_{n, p_t})])\le n^{1/6}\sqrt{\log n}\right\},$$
    where $\P(E_n^c)=\P(G_{n, p_t}\text{ is not well-behaved}) \le 6e^{-\sqrt{f_t/8}}=O(1/n)$ by Lemma \ref{lemma3_abr} with $f_t\sim\frac{n^{1/3}}{\log n}$.
	
	Consider a connected component $C_0$ of $G_{n, p_t}\setminus V(H_{n, p_t})$, $b_0\in C_0$, and $r\in \mathbb{N}$. For $1\le i \le r$, let $(a_{i-1}, b_i)$ be the first edge of $\partial C_{i-1}$ crossed by the random walk on our electric network started at $b_{i-1}$, where $a_{i-1}\in C_{i-1}$ and $b_i\notin C_{i-1}$. We define $C_i$ as the union of $C_{i-1}$ and the connected component of $b_i$ in $G_{n, p_t}$. If $C_{r}\cap H_{n, p_t}\ne \emptyset$, then $b_0$ can be connected to $H_{n, p_t}$ with a path of length $\le r(n^{1/6}\sqrt{\log n}+1)$.
 
    Under conditioning on $(H_{n, p_t}, C_{i-1}, a_{i-1})$, the edge $(a_{i-1}, b_i)$ has distribution $ \unif\{(a_{i-1}, v):v\notin C_{i-1}\}$, so for any labeling $\in A^{\mathrm{incr}}_n\cap E_n$
	$$\P\left(C_r\cap H_{n, p_t}=\emptyset\Given A^{\mathrm{incr}}_n\cap E_n\right)\le \left(1-\frac{3n/(2\log n)}{n}\right)^r \le \exp\left(-\frac{3r}{2\log n}\right).$$
	For $r_n:=4(\log n)^2/3$, using the union bound,
	\begin{align*}
		&\P\left(\diam(\wst_{n}^{\beta_n})-D_t\ge r_n(n^{1/6}\sqrt{\log n}+1)\Given A^{\mathrm{incr}}_n\cap E_n\right)\\
        &\hspace{3cm}\le \P\left(
		\exists C_0: C_{r_n}\cap H_{n, p_t}=\emptyset\Lvert A^{\mathrm{incr}}_n\cap E_n\right)\le n \exp\left(-\frac{3r_n}{2\log n}\right)=n e^{-2\log n}=\frac{1}{n},
	\end{align*}
	and since $r_n(n^{1/6}\sqrt{\log n}+1)=O(n^{1/6}(\log n)^{5/2})$, 
	\begin{align*}  
        \E\left[\diam(\wst_{n}^{\beta_n})-D_t\Lvert A^{\mathrm{incr}}_n\right]&\le \E\left[\diam(\wst_{n}^{\beta_n})-D_t\Lvert A^{\mathrm{incr}}_n\cap E_n\right]+n\P(E_n^c|A^{\mathrm{incr}}_n)\\
        &=O(n^{1/6}(\log n)^{5/2}).
	\end{align*}
\end{proof}

\begin{proof}[Proof of Theorem \ref{thm:has_diam_mst}] By Theorem 1.1 of \cite{michaeli2021diameter}, for any $\eps>0$, there exist $0<\tilde{c}(\eps)<\tilde{C}(\eps)$ such that for any $n$,
\begin{gather}\label{formula:ust_typical_lower}
\P\left(\tilde{c}(\eps)n^{1/2}<\diam(\ust_n)<\tilde{C}(\eps)n^{1/2}\right)>1-\eps.
\end{gather}

Suppose that the lower bound on the typical diameter of $\wst_n^{\beta_n}$ does not hold. Then there exists $\eps>0$ such that if $g_n$ tends to 0 slowly enough, then $\P\left(\diam(\wst_n^{\beta_n})\le g_n n^{1/3}\right)>\eps$ holds for infinitely many $n$'s. Hence,  and since $\tilde{c}(\eps/2)$ is a constant,  there exists $1\ll f_0\ll n^{1/3}$ such that \begin{gather}\label{formula:lower_indir}
\P\left(\diam(\wst_n^{\beta_n})\le \tilde{c}\left(\frac{\eps}{2}\right)\frac{\sqrt{\log f_0}}{2f_0}n^{1/3}\right)>\eps
\end{gather} for infinitely many $n$'s.

Denoting by $C_n^{\mathrm{contain}}$ the connected component  of $G_{n, 1/n-f_0/n^{4/3}}$ containing the subgraph $H_n:=H_{n, 1/n-2f_0/n^{4/3}}$, we show that $\P(C_n^{\mathrm{contain}}\text{ is a tree})\rightarrow1$. To do that, we consider $$E_n:=\left\{\frac{n^{2/3}\log f_0}{4f_0^2}\le |V(H_{n})|\le 2\frac{n^{2/3}\log f_0}{f_0^2}\text{ and } H_n\text{ is a tree}\right\}.$$
By Theorem 5.6 of \cite{janson2011random}, $\P(E_n)\rightarrow 1$ as $n\rightarrow \infty$.

On $E_n$, using Lemma \ref{lemma:conditioning}, we can exclude the appearance of $e\in C_n^{\mathrm{contain}}\setminus E(\mst_n)$, i.e.,~the edges of labels $\frac{1}{n}-\frac{2f_0}{n^{4/3}}\le U_e\le \frac{1}{n}-\frac{f_0}{n^{4/3}}$:
\begin{align*}
\P\left(C_n^{\mathrm{contain}}\text{ is not a tree}\Given E_n\right)&\le \E\left[\sum_{e\in C_n^{\mathrm{contain}}\setminus E(\mst_n)}\P\left(\frac{1}{n}-\frac{2f_0}{n^{4/3}}\le U_e\le \frac{1}{n}-\frac{f_0}{n^{4/3}}\right)\Given E_n\right]\\
&\le\left(2\frac{n^{2/3}\log f_0}{f_0^2}\right)^2\Theta\left(\frac{f_0}{n^{4/3}}\right)=\Theta\left(\frac{\log^2 f_0}{f_0^3}\right)\rightarrow 0,
\end{align*}
obtaining also that $\P\left(C_n^{\mathrm{contain}} \text{ is a tree}\right)\rightarrow 1$.

Then, we get that  
$\P(T_{n, 1/n-2f_0/n^{4/3}}=H_n)\rightarrow1$, too. Indeed, applying Lemma \ref{lemma:H_comps} with $i=-1$, we have that $\P\left(T_{n, 1/n-2f_0/n^{4/3}}\subseteq C_n^{\mathrm{contain}}\right)\rightarrow 1$, and if $C_n^{\mathrm{contain}}$ is a tree, then the only way making $T_{1/n-2f_0/n^{4/3}}$ connected inside $C_n^{\mathrm{contain}}$ is to agree with $H_n$.

Under conditioning on $E_n$, $H_n$ is distributed as the $\ust$ on $|V(H_n)|\ge \frac{n^{2/3}\log f_0}{4f_0^2}$ vertices, so by the union bound, we have that
\begin{gather*}
\P\left(\diam\left(T_{n, \frac{1}{n}-\frac{2f_0}{n^{4/3}}}\right)\le \tilde{c}\left(\frac{\eps}{2}\right)\frac{n^{1/3}\sqrt{\log f_0}}{2f_0}\right)\stackrel{(\ref{formula:ust_typical_lower})}{\le} \frac{\eps}{2}+\P\left(E_n^c\right)+\P\left(T_{n, 1/n-2f_0/n^{4/3}}\ne H_n\right)\rightarrow \frac{\eps}{2},
\end{gather*} contradicting (\ref{formula:lower_indir}).

For the upper bound of the typical diameter, suppose that there exists an $\eps>0$ for which the upper bound does not hold. For this $\eps$, $\exists f_0\gg 1$ such that $$\P\left(\diam(\wst_n^{\beta_n})\ge \frac{8f_0^4}{\eps}n^{1/3}\right)>\eps$$ for infinitely many $n$'s. But for large $n$'s, combining Proposition \ref{prop:to_supcrit} and Lemma \ref{lemma:from_supcrit_to_whole}, we have that
$\E\left[\wst_n^{\beta_n}\Given A^{\mathrm{incr}}_{n, f_0}\right]\le 4f_0^4 n^{1/3}$, which is a contradiction. Indeed, by Markov's inequality,
$$\P\left(\wst_n^{\beta_n}\ge\frac{8f_0^4}{\eps}n^{1/3}\right)\le \P\left(\wst_n^{\beta_n}\ge\frac{8f_0^4n^{1/3}}{\eps}\Given A^{\mathrm{incr}}_n\right)+\P(\overline{A^{\mathrm{incr}}_n}) \stackrel{\text{large $n$}}{\le} 4f_0^4n^{1/3}\frac{\eps}{8f_0^4n^{1/3}}+\frac{\eps}{2}=\eps.$$
\end{proof}

\begin{remark}[The difference between $\mst_n$ and $\wst_n^{\beta_n}$]
Although $\E[\mst_n]=\Theta(n^{1/3})$ is proven in \cite{addario2009critical}, we have only shown $\E\left[\wst_n^{\beta_n}\Given A^{\mathrm{incr}}_{n, f_0}\right]\le 4f_0^4 n^{1/3}$ with some $\P(A^{\mathrm{incr}}_{n, f_0})\rightarrow 1$ for any $f_0\gg 1$. If $f_0\rightarrow \infty$ slowly, then the convergence $\P(A^{\mathrm{incr}}_{n, f_0})\rightarrow 1$ can be slow too, that is why $\E[\wst_n^{\beta_n}]$ is not handled. 

In \cite{addario2009critical}, it was not discussed in detail how to overcome this issue. Our model should be studied more carefully in this regard, since the layers in Figure \ref{fig:wst_mst_layers} for $\wst_n^{\beta_n}$ are less ordered than for $\mst_n$. Therefore, we elaborate here how to deal with this technical difficulty.

For $\mst_n$, a possibility to omit the conditioning on $A^{\mathrm{incr}}_{n}$ is to consider a monotone sequence $\{J_{n, p_{i}}\}_{i}$ of graphs even if sometimes $H_{n, p_i}\not\subseteq H_{n, p_{i+1}}$. Let $J_{n, p_{t}}:=H_{n, p_{t}}$ be the largest  connected component of $G_{n, p_{t}}$. If we have already defined $J_{n, p_{i+1}}$, let $J_{n, p_{i}}$ be the  largest connected component of $G_{n, p_{i}}[V(J_{n, p_{i+1}})]$. When we study the evolution of $\diam\left(J_{n, p}\cap \mst_n\right)$ instead of $\diam\left(H_{n, p}\cap \mst_n\right)$, the main difference is that in (\ref{formula:difference_diams}) one has to estimate $\lp(G_{n, p_{i+1}}[V(J_{n, p_{i+1}})\setminus V(J_{n, p_i})])$ instead of $\lp(G_{n, p_{i+1}}[V(H_{n, p_{i+1}})\setminus V(H_{n, p_i})])$. This is not a problem, since if $J_{n, p_i}\ne H_{n, p_i}$, then $H_{n, p_i}\cap J_{n, p_{i+1}}=\emptyset$, so we even have that
$$\lp(G_{n, p_{i+1}}[V(J_{n, p_{i+1}})])\le \lp(G_{n, p_{i+1}}[V(K_n)\setminus V(H_{n, p_{i}})])\le \frac{n^{1/3}}{\sqrt{f_i}}.$$

Unfortunately, this idea cannot be generalized to $\wst_n^{\beta_n}$: if we define $J'_{n, p_{i}}$ as the largest connected component of $G_{n, p_{i}}[V(J'_{n, p_{i+{1/2}}})]$, then we cannot estimate $\lp(G_{n, p_{i+1}}[V(J'_{n, p_{i+1}})\setminus V(J'_{n, p_i})])$ nicely: it could happen that $J'_{n, p_i}\ne H_{n, p_i}$ and $H_{n, p_i}\subseteq J'_{n, p_{i+1}}\setminus J'_{n, p_{i+\frac{1}{2}}}$. Then a long path of $H_{n, p_i}$ could appear in $G_{n, p_{i+1}}[V(J'_{n, p_{i+1}})\setminus V(J'_{n, p_i})]$. That is why we need the conditioning on $A^{\mathrm{incr}}_{n}$.
\end{remark}

\subsection{A related work: fixed, heavy tailed conductances} \label{subsec:heavy_tail}

In Subsection 6.1 of \cite{makowiec2023diameter}, the random electric network with $\c_{\mathrm{fix}, n, \gamma}:=\left\{\exp\left(U_e^{-\gamma}\right)\right\}_{e\in E(K_n)}$ is studied. In this subsection, we show that there are similar phases for $\c_{\mathrm{fix}, n, \gamma}$ and for $\c_{\mathrm{run}, n,  \gamma+1}:=\left\{\exp\left(-n^{\gamma+1}U_e\right)\right\}_{e\in E(K_n)}$, simply by comparing the ratio of conductances in these two electric networks.

Indeed, we are checking that, for $U_f, U_e\le \frac{\log n}{n}$, it is easier to distinguish the edges $e$ and $f$ in $\c_{\mathrm{fix}, n, \gamma}$ than in $\c_{\mathrm{run}, n,  \gamma+1-a_n}$ for some positive $a_n=o(1)$, so
our proofs using the distinguishability of the edges can be extended to the model of \cite{makowiec2023diameter}. These proofs are about the properties (algorithms, models, typical diameters) being the same for the $\wst$ and the $\mst$.

Moreover, for $\frac{2}{n}\le U_f<U_e$, the comparability of the edges $e$ and $f$ happens easier in $\c_{\mathrm{fix}, n, \gamma}$ then in $\c_{\mathrm{run}, n,  \gamma+1}$ if we choose the constants in the comparabilities appropriately. Hence, the examples of the edges where the algorithms/models differ in $\c_{\mathrm{run}, n, \gamma+1}$ gives examples also in $\c_{\mathrm{fix}, n, \gamma}$.

These phases are formalized in Theorem \ref{thm:fixed_distr}, which will be an easy consequence of Theorems \ref{thm:ab_vs_inv}, \ref{thm:wst_vs_mst}, \ref{thm:has_diam_mst} and the following lemma. 

\begin{lemma}\label{lemma:connection-gamma-alpha}
Consider $c_{\mathrm{fix}, n, \gamma}(e):=\exp\left(U_e^{-\gamma}\right)$ and $c_{\mathrm{run}, n,  \alpha}(e):=\exp\left(-n^{\alpha}U_e\right)$ for any $e\in E(K_n)$.
\begin{itemize}
    \item [a)] For any $\gamma>\alpha-1>0$, there exists an $N_0:=N_0(\gamma, \alpha)$ such that for any $n\ge N_0$ and $e, f\in E(K_n)$ with $0<U_{f}\le U_{e}\le \frac{2\log n}{n}$, we have that
    $$ 0\le \frac{c_{\mathrm{fix}, \gamma}(e)}{c_{\mathrm{fix}, \gamma}(f)}\le \frac{c_{\mathrm{run}, n,  \alpha}(e)}{c_{\mathrm{run}, n,  \alpha}(f)}\le 1.$$
    \item [b)] For any $\gamma>0$, there exists a $c_0:=c_0(\gamma)>0$ such that for any $n$ and $e, f\in E(K_n)$ with $\frac{2}{n}\le U_f\le U_e\le 1$, we have that
    $$0\le \left(\frac{c_{\mathrm{run}, n,  \gamma+1}(e)}{c_{\mathrm{run}, n,  \gamma+1}(f)}\right)^{c_0}\le \frac{c_{\mathrm{fix}, n, \gamma}(e)}{c_{\mathrm{fix}, n, \gamma}(f)}\le 1.$$
\end{itemize}
\end{lemma}
\begin{proof} Let $g(x):=x^{-\gamma}$, then $|g'(x)|=|\gamma x^{-\gamma-1}|$. To check a), for $0<U_f\le U_e\le \frac{2 \log n}{n}$,
\begin{gather*}
    U_{f}^{-\gamma}-U_{e}^{-\gamma}\ge \gamma (2\log n)^{-\gamma-1}n^{\gamma+1}(U_{e}-U_{f})\stackrel{n\ge N_0}{\ge} n^{\alpha}(U_{e}-U_{f}).
\end{gather*}
We obtain the statement by plugging in this inequality to the decreasing $x\mapsto \exp(-x)$.

We prove b) similarly. For $x\ge\frac{2}{n}$, we have that $|g'(x)|\le g'\left(\frac{2}{n}\right)$, hence for $\frac{2}{n}\le U_f\le U_e$,
$$U_{f}^{-\gamma}-U_{e}^{-\gamma}\le \frac{\gamma}{2^{\gamma+1}}n^{\gamma+1}(U_{e}-U_{f}),$$
plugging in to $x\mapsto \exp(-x)$
gives us the result of the lemma with $c_0(\gamma):=\frac{\gamma}{2^{\gamma+1}}$.
\end{proof}

\begin{proof}[Proof of Theorem \ref{thm:fixed_distr}]
In the proof of Theorem \ref{thm:ab_vs_inv} a) in Subsection \ref{subsec:ideas}, in order to show that the Aldous-Broder and Prim's invasion algorithms typically do the same building steps on $\c_{\mathrm{run}, n, \alpha}$ for $\alpha>3$, we checked that $\frac{\c_{\mathrm{run}, n, \alpha}(h_j)}{\c_{\mathrm{run}, n, \alpha}(g_j)}\le n^{-11/2}\;\forall j\le n-1$ with probability $\rightarrow 1$. By Theorem 2.8.1 of \cite{durrett2010random}, with probability $\rightarrow 1$, the Erdős-Rényi graph $G_{n, 2\log n/n}$ is connected, implying that $\max_{e\in \mst_n }U_{e}\le 2\log n/n$. Combining these with Lemma \ref{lemma:connection-gamma-alpha} a), for any $\gamma>2$ we also have that $\frac{\c_{\mathrm{fix}, n, \gamma}(h_j)}{\c_{\mathrm{fix}, n, \gamma}(g_j)}\le n^{-11/2}\;\forall j\le n-1$, so the algorithms build the same edges in each step also in $\c_{\mathrm{fix}, n, \gamma}$ for $\gamma>2$.

On the other hand, if $\gamma<2$, then $\alpha=\gamma+1<3$. In the proof of Theorem \ref{thm:ab_vs_inv} b) in Subsection~\ref{subsec:ab_vs_invasion}, in order to show that the Aldous-Broder and Prim's invasion algorithms typically differ in some building steps on $\c_{\mathrm{run}, n, \alpha}$, we had to check that there are at least $k_n\rightarrow \infty$ invasion steps $g_j$ with $\frac{\c_{\mathrm{run}, n, \alpha}(h_j)}{\c_{\mathrm{run}, n, \alpha}(g_j)}>\eps_0$ for a fixed positive $\eps_0$ with probability $\rightarrow 1$. We did it for $\eps_0=1/2$, but this specific value is not important. Since the $k_n$ pairs of edges in the proof of Theorem \ref{thm:ab_vs_inv} b) have labels $\ge\frac{2}{n}$, by Lemma \ref{lemma:connection-gamma-alpha} b), these edges also satisfy $\frac{\c_{\mathrm{fix}, n, \gamma}(h_j)}{\c_{\mathrm{fix}, n, \gamma}(g_j)}> \eps_0$ for $\eps_0:=\frac{1}{2^{c_0}}$, giving us that the two algorithms typically differ also in $\c_{\mathrm{fix}, n, \gamma}$.

One can similarly argue that the phase transition for $\c_{\mathrm{run}, n, \alpha}$ around $\alpha \approx 2$ regarding the agreement of the models gives us a similar phase transition for $\c_{\mathrm{fix}, n, \gamma}$ around $\gamma\approx 1$. The only missing detail is that we can apply Lemma \ref{lemma:connection-gamma-alpha} b), or in other words, (\ref{formula:prob_wstnotmst}) about the divergence of the expected number of $\frac{1}{2}$-significant edges can be strengthened to hold even for the $1/2$-significant edges $e$ with $U_{e}\ge 2/n$. To check this, note that there exist $d_1, d_2\in (0, 1)$ such that the event $A_n$ defined as `$|V(H_{n, 2/n})|>d_1n$ and $|I_{n, 2/n}|>d_2n$’ satisfies $\P(A_n)\rightarrow 1$ by Theorem 2.3.2 and Subsection 2.8  of \cite{durrett2010random}. Since $|\{(h, i)\notin E(\mst_n):\;h\in V(H_{n, 2/n}), i\in I_{n, 2/n}\}|=\Theta(n^2)$ on $A_n$, running the summation on this edge set in (\ref{formula:prob_wstnotmst}) can modify the expectation only by a constant factor, and the label of these edges automatically satisfy $U_{(h, i)}>\max_{f\in \path(h, i)}U_f>2/n$.

In the proof of the growth of the typical $\diam(\wst(\c_{\mathrm{run}, n, \alpha}))$ for $\alpha>\frac{4}{3}$, the exact values of the conductances are used only in  Lemma \ref{lemma:H_comps}, where we need that the edges from $E(K_n)\setminus E(G_{n, p_{1/2}})$ and from $E(G_{n, p_0})$ are distinguishable, i.e.,~$\kappa \log n\le n^{\alpha}|p_{1/2}-p_0|=n^{\alpha}\Theta\left(\frac{1}{n^{4/3}}\right)$. To prove the similar statement for $\wst(\c_{\mathrm{fix}, n, \gamma})$, repeating the computation of the proof of Lemma \ref{lemma:connection-gamma-alpha} a), we need that $\kappa \log n \le |p_{1/2}^{-\gamma}-p_{0}^{-\gamma}|=\Theta\left(n^{\gamma+1}(p_{1/2}-p_0)\right)=\Theta\left(n^{\gamma+1-4/3}\right)$ which holds for $\gamma>\frac{1}{3}$.
\end{proof}

\begin{remark}[Conjecture: locally $\mst_n$, but globally of intermediate diameter growth]\label{rmrk:intermediate} For $\beta_n\gg n\log n$, we expect the local properties of $\wst_n^{\beta_n}$ to behave like the ones for $\mst_n$  by (\ref{formula:expected_diff}).

We think that this does not hold for the global geometry: for $0<\gamma <1/3$ and $\beta_n=n^{1+\gamma}$, we believe that Theorem \ref{thm:has_diam_mst} cannot be extended but typically $\diam(\mst_n)\ll \diam(\wst_n^{\beta_n})\ll \diam(\ust_n)$ holds for these parameters. However, the structure of our random electric network has to be understood carefully to determine the typical diameter. By Lemma \ref{lemma:connection-gamma-alpha}, this is in line with the conjecture of Remark 6.2 of \cite{makowiec2023diameter}: for any $0<\gamma<1/3$ and $c_{\mathrm{fix}, n, \gamma}(e)=\exp\left(U_e^{-\gamma}\right) \;\forall e\in E(K_n)$, they believe that the typical diameter of $\wst(c_{\mathrm{fix}, n, \gamma})$ grows like $n^{\delta}$ for some $\delta\in (1/3, 1/2)$.

It would also be nice to show that, for any $\delta>0$, there exists an $\eps>0$ such that if $\gamma\in (0, \eps)$ and $\beta_n=n^{1+\gamma}$, then the typical diameter of $\diam(\wst_n^{\beta_n})$ grows at least like $n^{1/2-\delta}$. This would result in an interesting model where the diameter typically grows almost at least as the diameter of $\ust_n$ but the model behaves like $\mst_n$ locally. This conjecture might be easier compared to determining the typical diameter for $\beta_n=n^{1+\gamma}$ $\forall \gamma\in (0, 1/3)$, since $H_{n, 1/n+\log n/\beta_n}$ has a small mixing time if $\gamma$ is close to 0: it is of order $n^{3\gamma+o(1)}$ by \cite{ding2011anatomy}, therefore there are similarities to the case of slowly growing $\beta_n$'s.
\end{remark}

\section{\texorpdfstring{$\wst^{\beta_n}_n$}{} for small \texorpdfstring{$\beta_n$}{}'s and balanced expander random networks}\label{sec:wst_vs_ust}
In Section \ref{sec:wst_vs_mst}, we determined the parameters for which $\wst_n^{\beta_n}$ behaves similarly to $\mst_n$ in different aspects:  regarding the generating algorithms, the models, the typical diameters and the expected total lengths. These similarities are based on the `relaxed greediness’ of the Aldous-Broder algorithm, making it possible to understand $\wst_n^{\beta_n}$ for fast growing $\beta_n$'s using the properties of the Erdős-Rényi graphs.

In Theorem \ref{thm:wst_vs_mst}, we could prove that the $\wst^{\beta_n}_n$ and $\mst_n$ contain the same edges for $\beta_n\gg n^2\log n$. Now, we show that the analogous theorem to couple the $\ust_n$ and the $\wst^{\beta_n}_n$'s for slowly growing $\beta_n$'s cannot hold unless $\beta_n=O(n^{-1/2})$. More precisely, if $\beta_n\gg n^{-1/2}$, then it is impossible to have $\P(\ust_n=\wst(\{\exp(-\beta_nU_e)\}_e)\given \{U_e\}_e)\rightarrow 1$ for most of the $\{U_e\}_e$'s, since for $L(T)=\sum_{e\in T}U_e$, we have that 
$\P\left(L(\ust_n)<\frac{n-1}{2}-\sqrt{n-1}\right)=\P\left(L(\ust_n)>\frac{n-1}{2}+\sqrt{n-1}\right)>C$ for a $C>0$ not depending on $n$, while
$$\frac{\P\left(L(\wst_n^{\beta_n})<\frac{n-1}{2}-\sqrt{n-1}\right)}{\P\left(L(\wst_n^{\beta_n})>\frac{n-1}{2}+\sqrt{n-1}\right)}\le e^{-2\beta_n \sqrt{n-1}}\frac{\P\left(L(\ust_n)<\frac{n-1}{2}-\sqrt{n-1}\right)}{\P\left(L(\ust_n)>\frac{n-1}{2}+\sqrt{n-1}\right)}=e^{-2\beta_n \sqrt{n-1}}\rightarrow 0.$$

Therefore, for slowly growing $\beta_n$'s, the interesting questions are about the typical diameter of $\wst^{\beta_n}_n$ and about $\E[L(\wst^{\beta_n}_n)]$.

In Proposition \ref{prop:expansion}, we prove that $\{\exp(-\beta_nU_e)\}_e$ forms a balanced expander electric network for $\beta_n\ll n/\log n$, which is enough to conclude that typically $\diam(\wst^{\beta_n}_n)=\Theta(n^{1/2})$ and to give a lower bound on $\E[L(\wst^{\beta_n}_n)]$.

\subsection{Balanced expanders and the typical diameter of the \texorpdfstring{$\wst_n^{\beta_n}$}{} for small \texorpdfstring{$\beta_n$}{}'s}

As we mentioned in Section \ref{sec:intro}, studying $\diam(\ust(G))$ for different $G$'s has a long history. We collect some results when $G$ is expander. The first result on general regular expanders was to determine $\E[\diam(\ust(G))]$ up to a polylogarithmic factor.

\begin{theorem*}[Theorem 15 of \cite{aldous}]\label{subsec:ust_diam} There exists a $K>0$ such that for any regular graph $G$ with $n$ vertices and relaxation time $\tau=(1-\lambda_2)^{-1}$,
 $$\frac{n^{1/2}}{K\tau \log n}\le\E [\diam(\ust(G))] \le K\tau^{1/2} n^{1/2}\log n.$$
\end{theorem*}

For scaling limit results, one has to know the growth of the typical diameter up to a constant factor, which is definitely a greater challenge. Similarly to $\ust(K_n)$, one would expect that the growth of the typical diameter is like $\Theta(|V(G)|^{1/2})$. However, in Subsection 1.4 of \cite{michaeli2021diameter}, a sequence of regular graphs $G_n$ with \emph{polylogarithmic} relaxation times is given for which the typical diameter of $\ust(G_n)$ is not $\Theta(|V(G_n)|^{1/2})$, showing that one has to be careful if they want to strengthen Theorem 15 of \cite{aldous}.

In \cite{michaeli2021diameter} some nice general conditions are given which are sufficient to show the growth of the typical diameter of $\ust(G)$ being $\Theta(|V(G)|^{1/2})$. Balanced expanders, hypercubes and tori of dimension $\ge 5$ satisfy these conditions. Recently, in \cite{makowiec2023diameter} the results of \cite{michaeli2021diameter} are extended to the weighted spanning trees of electric networks.

In order to state the theorems precisely, we start with some notation. From a random walk on an electric network $\c$, one gets a lazy random walk if, at each time, independently with probability $\frac{1}{2}$, the step is taken according to the transition rule, and with probability $\frac{1}{2}$ the walker stays at the current vertex. Denote by $q_t$ the transition probabilities of the lazy random walk, by $\pi$ the stationary distribution, and by $t_{\mathrm{mix}}$ the uniform mixing time of the lazy random walk, i.e., 
\begin{gather*}
    t_{\mathrm{mix}}(G, \c) := \min \left\{ t \ge 0 : \max_{u,v \in V(G)} \left| \frac{q_t(u,v)}{\pi(v)} - 1\right| \le \frac{1}{2}\right\}.
\end{gather*}

Generalizing the conditions of \cite{michaeli2021diameter}, in \cite{makowiec2023diameter} the following conditions are considered for some $D, \alpha, \theta>0$:
\begin{itemize}
    \item[](bal) $(G, \c)$ is balanced if $ \frac{\max_{u \in V} \sum_{v} c(u, v)}{\min_{u \in V} \sum_{v} c(u, v)} \le D$,
    \item[](mix) $(G, \c)$ is mixing if $t_{\mathrm{mix}}(G, \c) \le |V(G)|^{\frac{1}{2}-\alpha}$,
    \item[](esc) $(G, \c)$ is escaping if $\mathcal{B}(G, \c):=\sum_{t=0}^{t_{\mathrm{mix}}} (t+1) \sup_{v \in V}q_t(v,v) \le \theta$.
\end{itemize}

\begin{theorem*}[Theorem 2.3 of \cite{makowiec2023diameter}, generalizing Theorem 1.1 of \cite{michaeli2021diameter}]
For any fixed $\alpha>0$, there exist $k, C, \gamma>0$ such that if an electric network $(G,\c)$ of $n$ vertices satisfies (mix) with $\alpha$ and 
conditions (bal) and (esc) with some $D=D(n)\le n^{\gamma}$ and $\theta=\theta(n) \le n^{\gamma}$, then for any $\eps>n^{-\gamma}$
\begin{gather*}
    \P\big( (CD \theta \epsilon^{-1})^{-k} \sqrt{n} \le \diam(\wst(\c)) \le (CD \theta \epsilon^{-1})^{k} \sqrt{n}\big) \ge 1 - \eps.    
\end{gather*}
\end{theorem*}

\begin{prop}\label{prop:expansion} If $\beta_n\ll n/\log n$, then for $\c_{n}:=\{\exp(-\beta_nU_e)\}_{e\in E(K_n)}$ ,
\begin{gather*}
\P\left(\lambda_2(\c_n)<0.92\text{ and }\frac{\max_{u \in V} \sum_{v} c_n(u, v)}{\min_{u \in V} \sum_{v} c_n(u, v)} \le 1.05\right)\rightarrow 1,
\end{gather*}
where $1=\lambda_1(\c_n)\ge\lambda_2(\c_n)\ge \ldots \ge\lambda_n(\c_n)$ are the eigenvalues  of the transition matrix of the random walk w.r.t.~$\c_n$.
\end{prop}
As we detail in Remark \ref{rmrk:sharpness_spectral}, this statement is not sharp.

\begin{proof}[Proof of Theorem \ref{thm:has_diam_ust}]
We have to check that the spectral gap implies (mix) and (esc), which we do likewise the reasoning of Subsection 1.3 of \cite{michaeli2021diameter}. We denote by $\lambda_1(q_n)\ge \ldots \ge \lambda_n(q_n)$ the eigenvalues of the transition matrix of the lazy random walk on $\c_n$.

By Proposition \ref{prop:expansion}, $\pi(v)\in \left[\frac{1}{1.05}\frac{1}{n}, 1.05\frac{1}{n}\right]$ $\forall v\in V(K_n)$ and the absolute spectral gap of the lazy random walk satisfies $\gamma^{*, \mathrm{lazy}}_n:=1-\max (\lambda_2(q_n),|\lambda_n(q_n)|)=\frac{1}{2}(1-\lambda_2(\c_n))>\frac{1}{2}(1-0.92)=0.04$. In (12.13) of \cite{levin2017markov},
$$\left|\frac{q^{(t)}(x,y)}{\pi(y)}-1\right|\le \frac{\exp(-t\gamma^{*, \mathrm{lazy}}_n)}{\min_x \pi(x)}$$
is obtained for any $x,y\in V(K_n)$, therefore with probability $\rightarrow 1$, we have  $t_{\mathrm{mix}}(G, \c_n)\le c_0\log n$ and $q_{t}^{(\c_n)}(v,v)\le c_1n^{-1}+c_2e^{-c_3t}$ for some constants $c_0,\ldots, c_3>0$ not depending on $n$, giving us the estimate $\mathcal{B}(G, \c_n)\le (c_0\log n+1)^2c_1n^{-1}+c_2\sum_{t=0}^{\infty}(t+1)e^{-c_3t}$ which is bounded as $n\rightarrow \infty$. Then (bal), (mix) and (esc) of Theorem 2.3 of \cite{makowiec2023diameter} are satisfied for some fixed $D, \alpha, \theta$, so the typical diameter grows like $\Theta(n^{1/2})$.
\end{proof}

\subsection{The proof of the spectral gap}\label{subsec:spectral_gap}
Now, our aim is to study the expansion of $\{\exp(-\beta_nU_e)\}_{e\in E(K_n)}$. The spectral gap for Erdős-Rényi graphs is well-studied. For us the following result is important.
\begin{theorem}[Theorem 1.1 of \cite{hoffman2021spectral}]\label{thm:hkp}
Fix $\delta>0$ and let
$ p \ge \left(\frac{1}{2}+\delta \right)\frac{\log n}{n}.$ Let $d=p(n-1)$ denote the expected degree of a vertex. For every fixed $\epsilon > 0,$
there is a constant $C=C(\delta, \epsilon),$
so that
$$\P\left(\lambda\left(H_{n, p}\right)<\frac{C}{\sqrt{d}}\right)\ge 1-Cne^{-(2-\epsilon)d}-Ce^{-d^{1/4}\log n},$$
where $1=\lambda_1(G)\ge \lambda_2(G)\ge \ldots$ are the eigenvalues of the transition matrix of the simple random walk on $G$ and $\lambda(G):=\max_{i\ge 2}|\lambda_i\left(G\right)|$.
\end{theorem}
Using that result, we can prove that our electric network is also an expander. For convenience, we switch to estimating the Cheeger constant instead of dealing with eigenvalues.

\begin{defi}\label{def:cheeger}For $A, B\subseteq V(K_n)$ not necessary disjoint sets, let us use the notation 
$$c_E(A, B):=\sum_{\substack{x\in A, y\in B,\\ (x,y)\in E}}c(x,y)\ \text{ and } \ c(A, B):=c_{E(K_n)}(A, B).$$
Then the Cheeger constant is defined as
$$h(\c):=\min\left\{\frac{c(S, S^c)}{c(S, V(K_n))}: \; S\subseteq V(K_n)\text{ with } \pi(S)\le \frac{1}{2}\right\}.$$

The eigenvalues of the transition matrix of the random walk on $\c$ are denoted by $1=\lambda_1(\c)\ge \lambda_2(\c)\ge \ldots \ge \lambda_n(\c)$.
\end{defi}
Then the stationary measure can be expressed simply as $\pi(S)=\frac{c(S, V(K_n))}{c(V(K_n), V(K_n))}$. 

Recall Cheeger's inequality connecting the Cheeger constant and the spectral gap.

\begin{theorem*}[Cheeger's inequality, Theorem 13.10 of \cite{levin2017markov}] Using the notation of Definition~\ref{def:cheeger}, for any electric network with conductances $\c$, 
    $$\frac{h^2(\c)}{2}\le 1-\lambda_2(\c)\le 2h(\c).$$
\end{theorem*}

\begin{notation}\label{notation:spectr}
Consider $p_n:=\lceil n/(K\log n)\rceil^{-1}\sim K\log n/n$, where we think of $K$ as a fixed, big number. We decompose the edges of $E(K_n)$ into
$$E_{n, j}:=\left\{e\in E(K_n): (j-1)p_n\le U_e<jp_n\right\} \ \text{ for }j=1, \ldots, \lceil n/(K\log n)\rceil.$$

The degree of $v$ in $G_{n, j}:=(V(K_n), E_{n, j})$ is denoted by $\deg^{(n, j)}(v)$. Let the eigenvalues of $G_{n, j}$ be $1=\lambda^{(n, j)}_1\ge \lambda^{(n, j)}_2, \ldots, \lambda^{(n, j)}_n$ and $\lambda^{(n, j)}:=\max_{i\ge 2}\left|\lambda^{(n, j)}_i\right|$. 
\end{notation}

\begin{lemma}\label{lemma:exp_and_deg} Using Notation \ref{notation:spectr}, for any $\eps>0$ there exist $K=K(\eps), C=C(\eps)$ such that
\begin{gather*}
    \P\left(\frac{\deg^{(n, j)}(v)}{\E \deg^{(n, j)}(v)}\notin \left(\frac{1}{1+\eps}, 1+\eps\right)\;\forall v\in V(K_n), \forall j\le \left\lceil \frac{n}{K\log n}\right\rceil\right)=1-o\left(\frac{1}{n}\right)\text{ and}\\
    \P\left(\lambda^{(n, j)}\le \frac{C}{{\sqrt{\log n}}} \;\forall j\le \left\lceil \frac{n}{K\log n}\right\rceil \right)=1-o\left(\frac{1}{n}\right).
\end{gather*}
\end{lemma}

\begin{proof}
In this proof we emphasize the dependence on $n$ and $K$ of the quantities by the notation $p(n, K)=\lceil n/(K\log n)\rceil^{-1}$ and $d(n, K)=(n-1)p(n, K)$.

The concentration of the degrees follows from Chernoff's bound. Namely, fix $v$ and $j$. Since $\deg^{(n, j)}(v)\sim \text{Bin}\left(n-1, p(n, K)\right)$, by Corollary A.1.14. of \cite{alon2016probabilistic}, there exists $f(\eps)>0$ such that
\begin{gather*}
    \P\left(\frac{\deg^{(n, j)}(v)}{\E \deg^{(n, j)}(v)}\notin \left(\frac{1}{1+\eps}, 1+\eps\right)\right)\le
    2e^{-f(\eps)d(n, K)}\le 2e^{-f(\eps)(K\log n)/2}\stackrel{\text{large }K}{=}O(n^{-3}),
\end{gather*}
where we also used that $d(n, K)\ge (K\log n)/2$.

By the union bound running on $v\in V(K_n)$ and $1\le j \le \lceil n/(K\log n)\rceil$,
\begin{gather*}\P\left(\exists v\in V(K_n), \exists j \le \left\lceil \frac{n}{K\log n}\right\rceil: \ \frac{\deg^{(n, j)}(v)}{\E \deg^{(n, j)}(v)}\notin \left(\frac{1}{1+\eps}, 1+\eps \right) \right)=o(n^2)O(n^{-3})=O(n^{-1}).
\end{gather*}

Since $d(n, K)\ge 3\log n$ for $K>3$, by Theorem \ref{thm:hkp}
with $\eps=1$ and by the union bound
\begin{align*}
    \P\left(\lambda^{(n, j)}\le \frac{C}{{\sqrt{\log n}}} \ \forall j\le \left\lceil \frac{n}{K\log n}\right\rceil\right)&\ge 1 - \frac{n}{K\log n} \left(Cne^{-(2 - \eps)d(n, K)} + Ce^{-d(n, K)^{1/4} \log n}\right)\\
    &=1-o(n^{-1}).
\end{align*}
\end{proof}

Now, we turn to the proof of Proposition \ref{prop:expansion}. 

\begin{remark}[Proposition \ref{prop:expansion} is not sharp]\label{rmrk:sharpness_spectral}
In order to make the proof more readable, we have chosen to write the proof with specific numbers. With the method of our proof, a more careful computation could result in $\P\left(1-\lambda_2(\c_n)\ge \frac{1}{8}\right)\rightarrow 1$. 

If one needs a delicate result on the second eigenvalue, then it is reasonable to estimate the eigenvalues directly from Rayleigh coefficients instead of using Cheeger's inequality. 

To understand the typical diameter of $\wst_n^{\beta_n}$, we only need the existence of the spectral gap and not the exact value of it, hence we do not care about the sharpness of this statement.
\end{remark}

\begin{proof}[Proof of Proposition \ref{prop:expansion}] 
Let $A_{n}^{\mathrm{deg}}:=\left\{\frac{\deg^{(n, j)}(v)}{\E[\deg^{(n, j)}(v)]}\in \left(\frac{99}{100}, \frac{100}{99}\right)\;\forall v\in V(K_n), \forall j\le \left\lceil \frac{n}{K\log n}\right\rceil \right\}$ and $A_n:=\left\{\lambda^{(n, j)}\le \frac{C}{{\sqrt{\log n}}}\; \forall j\le \left\lceil \frac{n}{K\log n}\right\rceil \right\}\cap A_{n}^{\mathrm{deg}}$, where $K$ and $C$ are chosen to have $\P(A_n)\rightarrow 1$ which is possible by Lemma~\ref{lemma:exp_and_deg}. Let $n$ be so large that $\left(1-\frac{C}{\sqrt{\log n}}\right)\ge 0.95$ and $e^{-\beta_n p_n}\ge 0.98$ for $p_n=\lceil n/(K\log n)\rceil^{-1}$. Abbreviating $V:=V(K_n)$ and $N_{E_{n, j}}(A, B):=|\{u\in A, v\in B: \;(u,v)\in E_{n, j}\}|$,
\begin{align*}
    A_{n} &\stackrel{\text{(C)}}{\subseteq} \left\{\frac{N_{E_{n, j}}(S, S^c)}{N_{E_{n, j}}(S, V)}\ge \frac{0.95}{2} \ \ \ \ \forall S \text{ with }|S|\le \frac{|V|}{2}, \ \forall j\le \left\lceil \frac{n}{K\log n}\right\rceil \right\}\cap A_{n}^{\mathrm{deg}}\\
    &\stackrel{\text{(*)}}{\subseteq} \left\{\frac{c_{E_{n, j}}(S, S^c)}{c_{E_{n, j}}(S, V)} \ge \frac{0.95\cdot 0.85}{2} \ \ \ \ \forall S \text{ with }\pi(S)\le \frac{1}{2}, \ \forall j\le \left\lceil \frac{n}{K\log n}\right\rceil\right\}\cap A_{n}^{\mathrm{deg}}\\
    &\subseteq \left\{\frac{\sum_{j=1}^{\lceil n/(K\log n)\rceil}c_{E_{n, j}}(S, S^c)}{\sum_{j=1}^{\lceil n/(K\log n)\rceil} c_{E_{n, j}}(S, V)} \ge \frac{0.95\cdot 0.85}{2} \ \ \ \ \forall S \text{ with }\pi(S)\le \frac{1}{2}\right\}\cap A_{n}^{\mathrm{deg}}\\ 
    &\stackrel{\text{(C)}}{\subseteq}\left\{1-\lambda_2^{(n)}\ge \frac{1}{2} \left(\frac{0.95\cdot 0.85}{2}\right)^{2} \right\}\cap A_{n}^{\mathrm{deg}}\subseteq\left\{1-\lambda_2^{(n)}\ge 0.08\right\}\cap A_{n}^{\mathrm{deg}},
\end{align*}
where in the steps marked by (C) we use Cheeger's inequality. Now, we check that $(*)$ holds. If $|S|>\frac{1}{2}\frac{100^2}{99^2}0.98^{-1}n$, then 
\begin{gather*}
    \pi_{E_{n, j}}(S)=\frac{c_{E_{n, j}}(S, V)}{c_{E_{n, j}}(V, V)}\ge \frac{\frac{99}{100}|S|(n-1)p_ne^{-jp_n\beta_n}}{\frac{100}{99}|V|(n-1)p_ne^{-(j-1)p_n\beta_n}}\ge\frac{99^2}{100^2}0.98\frac{|S|}{n}>\frac{1}{2}.
\end{gather*}
So $\pi(S)\le \frac{1}{2}$ can happen only for $|S|\le \frac{1}{2}\frac{100^2}{99^2}0.98^{-1}n<0.53n$. If $|S|\le \frac{1}{2}n$, then of course
\begin{gather*}
\frac{c_{E_{n, j}}(S, S^c)}{c_{E_{n, j}}(S, V)}\ge \frac{e^{-jp_n\beta_n}N_{E_{n, j}}(S, S^c)}{e^{-(j-1)p_n\beta_n}N_{E_{n, j}}(S, V)}\ge 0.98 \frac{N_{E_{n, j}}(S, S^c)}{N_{E_{n, j}}(S, V)},
\end{gather*}
and if $\frac{1}{2}n\le |S|\le 0.53n$, then $|S^c|\le 0.47n$, hence
\begin{align*}
    \frac{c_{E_{n, j}}(S, S^c)}{c_{E_{n, j}}(S, V)}&\ge 0.98\frac{N_{E_{n, j}}(S, S^c)}{N_{E_{n, j}}(S^c, V)} \frac{\frac{99}{100}|S^c|(n-1)p_n}{\frac{100}{99}|S|(n-1)p_n}\\
    &\ge 0.98\frac{99^2}{100^2}\frac{0.47}{0.53} \frac{N_{E_{n, j}}(S, S^c)}{N_{E_{n, j}}(S^c, V)}\ge 0.85 \frac{N_{E_{n, j}}(S, S^c)}{N_{E_{n, j}}(S^c, V)},
\end{align*}
therefore $(*)$ holds both for $|S|\le \frac{1}{2}n$ and for $\frac{1}{2}n\le |S|\le 0.53n$.

Moreover, on $A_{n}^{\mathrm{deg}}$, for $e^{-\beta_np_n}\le 0.98$, we have $\frac{\max_{u \in V} \sum_{v} c_n(u, v)}{\min_{u \in V} \sum_{v} c_n(u, v)}\le (0.99^2\cdot 0.98)^{-1}<1.05$.
\end{proof}

\subsection{Expected total length for small \texorpdfstring{$\beta_n$}{}'s} \label{subsec:ust_length}

In Subsection \ref{subsec:mst_length}, we have shown Theorem \ref{thm:length} a) and b) by giving an upper bound on the expected total length. In this Subsection, we prove Theorem \ref{thm:length} c) for $\beta_n\ll n/\log n$.

In Theorem \ref{thm:length} c) of the first version of this paper, we have only shown that $\E\left[L(\wst_n^{\beta_n})\right]=\Omega\left(\frac{n}{\beta_n+1}\right)$ by giving the estimate $\mathcal{R}_{\mathrm{eff}}(u\leftrightarrow v)\ge \frac{\beta_n}{1.1e n}$ for any $u, v\in V(K_n)$ with probability $1-O(n^{-2})$. Using the results of \cite{makowiec2024local} on effective resistances, one can obtain a better result, see also Theorem 1.4 of \cite{makowiec2024local}.

\begin{proof}[Proof of Theorem \ref{thm:length} c)] We consider
\begin{align*}
A_{n}:=\left\{\mathcal{R}_{\mathrm{eff}}(v\leftrightarrow w)\in \left(\frac{(1-\delta_n)2\beta_n}{(1-e^{-\beta_n})n}, \frac{(1+\delta_n)2\beta_n}{(1-e^{-\beta_n})n}\right) \; \forall v, w\in V(K_n)\right\},
\end{align*}
where we can choose $\delta_n\rightarrow 0$ slowly enough to have $\P(A_n)\ge 1-o(n^{-1})$ by Theorems 3.1 and 3.5 of \cite{makowiec2024local}.

Kirchhoff’s Effective Resistance Formula, described in Section 4.2 of \cite{PTN}, gives us that
\begin{align*}
    \P\left((v,w)\in \wst_n^{\beta_n}\right)=c(v, w)\mathcal{R}_{\mathrm{eff}}(v\leftrightarrow w).
\end{align*}

For any label configuration $\mathbf{u}:=\{u_e\}_{e\in E(K_n)}$ from $A_n$, for $\mathbf{E}_{\mathbf{u}}[\cdot]:=\E[\cdot|U_e=u_e \forall e\in E(K_n)]$,
\begin{align*}
\mathbf{E}_{\mathbf{u}}\left[L\left(\wst_n^{\beta_n}\right)\right]&=\sum_{(v,w)\in E(K_n)} \mathbf{E}_{\mathbf{u}}\left[U_{v,w}c(v, w)\mathcal{R}_{\mathrm{eff}}(v\leftrightarrow w)\right]=\binom{n}{2}\E\left[U_fe^{-\beta_nU_f}\right]\frac{(1+o(1))2\beta_n}{(1-e^{-\beta_n})n}\\
&=(1+o(1))\frac{n^2}{2}\frac{1-\beta_ne^{-\beta_n}+e^{-\beta_n}}{\beta_n^2}\frac{2\beta_n}{(1-e^{-\beta_n})n}, 
\end{align*}
giving us that
$$\E\left[L\left(\wst_n^{\beta_n}\right)\right]=\P(A_n^c)O(n)+\P(A_n)\E\left[L\left(\wst_n^{\beta_n}\right)\Given A_n\right]=(1+o(1))\frac{(1-\beta_ne^{-\beta_n}+e^{-\beta_n})n}{(1-e^{-\beta_n})\beta_n},$$
since $|L\left(\wst_n^{\beta_n}\right)|\le n-1$ and $\P(A_n^c)=o(n^{-1})$.
\end{proof}

\bibliographystyle{alpha}

\end{document}